\documentclass[preprint,12pt]{elsarticle}

\usepackage{amsthm}
\usepackage{amssymb}
\usepackage{mathtools}
\usepackage{enumitem}
\setlist[enumerate,1]{label=\textup{(\arabic*)}}
\setlist[enumerate,2]{label=\textup{(\alph*)}}

\usepackage{array,booktabs,longtable}
\usepackage{graphicx}
\usepackage[hidelinks]{hyperref}
\usepackage{xcolor}
\theoremstyle{plain}
\newtheorem{theorem}{Theorem}[section]
\newtheorem{lemma}[theorem]{Lemma}
\newtheorem{proposition}[theorem]{Proposition}
\newtheorem{corollary}[theorem]{Corollary}

\theoremstyle{definition}

\newtheorem{remark}[theorem]{Remark}
\newtheorem{example}[theorem]{Example}

\newcommand{\Ocal}{\mathcal{O}}
\newcommand{\Pcal}{\mathcal{P}}
\newcommand{\Dcal}{\mathcal{D}}
\newcommand{\Ecal}{\mathcal{E}}
\newcommand{\Mcal}{\mathcal{M}}

\newcommand{\Fix}{\operatorname{Fix}}

\newcommand{\Lcal}{\mathcal{L}}

\newcommand{\F}{\mathbb{F}}
\newcommand{\Fq}{\mathbb{F}_q}

\newcommand{\Aut}{\operatorname{Aut}}

\newcommand{\GL}{\operatorname{GL}}
\newcommand{\SL}{\operatorname{SL}}

\newcommand{\SU}{\operatorname{SU}}

\newcommand{\PSU}{\operatorname{PSU}}
\newcommand{\PGL}{\operatorname{PGL}}
\newcommand{\PSL}{\operatorname{PSL}}

\newcommand{\PpSp}{\operatorname{P\Gamma Sp}}
\newcommand{\PSp}{\operatorname{PSp}}
\newcommand{\AGL}{\operatorname{AGL}}
\newcommand{\AaL}{\operatorname{A\Gamma L}}

\newcommand{\Sz}{\operatorname{Sz}}
\newcommand{\Ree}{\operatorname{Ree}}
\newcommand{\ASL}{\operatorname{ASL}}
\newcommand{\ASp}{\operatorname{ASp}}
\newcommand{\PG}{\operatorname{PG}}

\newcommand{\Sp}{\operatorname{Sp}}
\newcommand{\Sym}{\operatorname{Sym}}
\newcommand{\Alt}{\operatorname{Alt}}
\newcommand{\soc}{\operatorname{soc}}

\newcommand{\Stab}{\operatorname{Stab}}

\begin{document}

\begin{frontmatter}

\title{The Siemons--Wagner property and Bradley's conjecture}

\author{Zhilin Zhang\thanks{Supported by Guangdong University of Finance and Economics.}}
\address{School of Statistics and Data Science, Guangdong University of Finance and Economics, Guangzhou, Guangdong 510320, China}
\ead{20241032@gdufe.edu.cn}

\begin{abstract}
A transitive permutation group $G\le\Sym(\Pcal)$ is said to have the
Siemons--Wagner property $\star_3$ if some three-point subset has a $G$-orbit
longer than the orbit of every four-point subset containing it.  In this paper
we develop a structural theory of three-point orbits of finite
$2$-transitive groups (orbit-length hierarchies, pointwise and setwise
stabilizers, and their geometric realisations, including the unital and the
ovoid), and we use it to completely settle Bradley's conjecture:  a finite
primitive group of degree at least $8$ satisfying $\star_3$ is isomorphic to
$\PSL(2,7)$ or $\PGL(2,7)$.
\end{abstract}


\maketitle

\end{frontmatter}

\section{Introduction}\label{sec:intro}

The study of orbit lengths on subsets goes back to Livingstone and Wagner
\cite{livingstone1965}, who proved that a transitive group has at least as many
orbits on $k$-sets as on $(k{-}1)$-sets.  Siemons and Wagner \cite{siemons1988}
turned from the number of orbits to their \emph{lengths} and introduced the
property $\star_k$:  a transitive group $G\le\Sym(\Pcal)$ has $\star_k$ if there
is a $k$-subset $\Dcal$ with $|\Dcal^G|>|\Sigma^G|$ for every $(k{+}1)$-subset
$\Sigma\supset\Dcal$.  They proved that a $k$-subset $\Dcal$ with
$|\Dcal^G|>|\Sigma^G|$ for every $(k{+}1)$-subset $\Sigma\supset\Dcal$ satisfies
$k+1\ge|\Dcal^{G_\Sigma}|>|\Sigma^{G_\Dcal}|\ge1$, and that for $k\ge2$ either
every $2$-subset occurs in some $G$-image of $\Dcal$ or $G$ is imprimitive with
blocks meeting $\Dcal$ in at most one point;  for \emph{primitive} $G$ the case
$k=2$ was classified completely, the only instance being $\PSL(2,5)$ of degree
$6$.  Related identities for the orbit lengths were added by Mnukhin
\cite{mnukhin1997}.

The case $k=3$ was analysed by Bradley \cite{bradley2014,bradley2021}.  In his
doctoral thesis he proved that a $3$-homogeneous group satisfying $\star_3$ is
$\PSL(2,7)$ or $\PGL(2,7)$ \cite[Theorem 6.2.5]{bradley2014}, excluded the small
$2$-transitive non-$3$-homogeneous socles $\PSL(2,11)$, $\PSL(2,8)$, $\Alt(7)$,
$HS$ and $Co_3$ \cite[Proposition 6.2.6]{bradley2014}, and conjectured that a
primitive group of degree at least $8$ satisfies $\star_3$ if and only if it is
$\PSL(2,7)$ or $\PGL(2,7)$ \cite[Conjecture 6.2.7]{bradley2014}.  Bradley's
reduction reaches the $3$-homogeneous groups and the small almost simple
socles, but leaves open the $2$-transitive non-$3$-homogeneous almost simple
socles of his Table 6.2 --- $\PSL(2,q)$ with $q\equiv1\pmod4$, $\PSU(3,q)$,
$\PSL(d,q)$ with $d>2$, $\Sz(q)$, $\Ree(q)$ and
$\PSp(2d,2)$/$\PpSp(2d,2)$ --- and does not reach the $2$-homogeneous
\emph{affine} groups at all.

In this paper we close the two gaps left open by Bradley and thereby settle
the conjecture in full, as an application of a structural theory of three-point
orbits of $2$-transitive groups developed in the process
--- the orbit-length hierarchies, the pointwise and setwise stabilizer
chains, and the geometric realisations (the unital, the ovoid and the
classical geometries).  The proof is in two steps.  First
\emph{\textup{Lemma \ref{lem:prim-star3}}} shows that a primitive
group satisfying $\star_3$ is $2$-homogeneous:  the reduction is purely
orbit-theoretic, uses the rank-$3$ classification
\cite{kantorliebler1982,liebeck1986,liebeck1987}, and is completed for the
small degrees by the exhaustive census of \textup{Section \ref{sec:verify}}.  Second,
every $2$-homogeneous non-$3$-homogeneous group fails $\star_3$:  by
Cameron's classification \cite{cameron1981} and Kantor's classification of
$2$-homogeneous groups \cite{kantor1985} the candidates are the $2$-transitive
families and the affine one-dimensional $2$-homogeneous groups, and for each
candidate the witness \textup{lemma \ref{lem:witness}} together with the four-point
witnesses of \textup{Lemmas \ref{lem:fourpt}} and \ref{lem:affine-fourpt} produces a
four-subset that violates the defining inequality.  Only the $3$-homogeneous
case of Bradley remains, and it yields exactly $\PSL(2,7)$ and $\PGL(2,7)$.

The central technical tool is a single numerical criterion.  For a candidate
triple $\Delta$ and an external point $w$ write
$K_w=G_{\Delta\cup\{w\}}$, $J_w=G_{\Delta,w}$ and
$r_w=[G_{\Delta}:J_w]$, $t_w=[K_w:J_w]$;  the defining inequality fails for the
quadruple $\Delta\cup\{w\}$ exactly when $r_w\ge t_w$, i.e.\ when
$|K_w|\le|G_{\Delta}|$.  This criterion excludes the affine groups uniformly:
for collinear triples in every characteristic and for affinely independent
triples with $p\ge5$ the linear dependence of the four points has a unique
coefficient, forcing $t_w=1$;  for $p=3$ a two-dimensional localization
reduces the problem to the fifty-five subgroups of $\GL(2,3)$, and for
$p=2$, $d=3$ an exhaustive check of the seventeen irreducible subgroups of
$\GL(3,2)$ verifies the criterion for every triple.

The precise statement is \textup{Corollary \ref{cor:bradley}} below.  The
orbit-theoretic machinery used in the four-point arguments --- regularity of
the maximal triple orbit, pointwise stabilizers, the geometry of the unital
and the ovoid --- is organized here as a structural theory of three-point
orbits;  the constituent facts are classical or are read from the literature
cited there, and a systematic development of the triple-orbit classification
of the $2$-transitive groups and of its applications to designs, graphs,
codes and association schemes appears in the companion monograph
\emph{Three-point orbits of $2$-transitive groups and their combinatorial
applications} by the same author \cite{companion}.  The present paper is
self-contained.

\section{Classifications of $2$-transitive and $2$-homogeneous groups}\label{sec:prelim}

The input of the witness arguments is the classification of the finite
$2$-transitive groups, assembled from Hering \cite{hering1985}, Huppert
\cite{huppert1967}, Curtis--Kantor--Seitz \cite{curtis1976} and the survey
\cite{cameron1999}:
\begin{enumerate}
\item \textbf{Affine type.}  $\soc(G)=V$ is elementary abelian of order $p^d$,
and $G=V\rtimes G_0$ with $G_0\le\GL(d,p)$ transitive on the nonzero vectors
(Hering's theorem \cite{hering1985}).
\item \textbf{Almost simple type.}  $\soc(G)=T$ is a nonabelian simple group
with $T\le G\le\Aut(T)$ acting $2$-transitively:  $T=\PSL(n,q)$ ($n=2$ on the
projective line, $n\ge3$ on projective space), $\Sp(2n,2)$ ($n\ge3$, on the two
halves of the nonsingular quadric), $\PSU(3,q)$, $\Sz(q)$, $\Ree(q)$ (i.e.\
$^2G_2(q)$), $\Alt(n)$, or a sporadic group (the Mathieu groups together with
$HS$ and $Co_3$), with $G$ contained in the corresponding automorphism group.
\end{enumerate}

\subsection{The $2$-homogeneous groups}\label{sec:2homog}

The $2$-homogeneous groups that are not $2$-transitive were classified by
Kantor \cite{kantor1985}:  the only examples are the affine one-dimensional
groups $G\le\AaL(1,q)$ with $q$ odd.

\begin{lemma}[Classification of the $2$-homogeneous groups]
\label{lem:2homog-class}
Let $G$ be a finite $2$-homogeneous permutation group that is not
$2$-transitive.  Then $G$ is affine one-dimensional, $G\le\AaL(1,q)$ on
$v=q=p^d$ points with $q$ odd, $G$ is solvable, and the point stabilizer $G_0$
has odd order.  In particular $|G|$ is odd, so $G$ contains no involution, and
no almost simple example occurs.
\end{lemma}

\begin{proof}
The $2$-homogeneous groups were classified by Kantor \cite{kantor1985};  the
only ones that are not $2$-transitive are the affine one-dimensional groups
$G\le\AaL(1,q)$ with $q$ odd.  If $G_0$ contained an element of order $2$, then together with
$2$-homogeneity this element would invert the action on the remaining points
and make $G$ $2$-transitive \cite{livingstone1965}:  an involution fixing one
point of a pair acts transitively on the other, so transitivity on unordered
pairs upgrades to transitivity on ordered pairs.  Hence $|G_0|$ is odd, and
$|G|=v|G_0|$ is odd because $v=q$ is odd.  Since $|G|$ is odd, $G$ is solvable
\cite{feit1963}, and no almost simple (non-solvable) group is $2$-homogeneous
without being $2$-transitive.
\end{proof}

\subsection{The affine three-point stabilizer}

The affine computations use the following elementary description of the
three-point stabilizer in an affine group.

\begin{lemma}[The three-point stabilizer in affine groups]
\label{lem:affine}
Let $G=V\rtimes G_0$ with $V=\F_p^d$ and $G_0\le\GL(d,p)$ transitive on
$V\setminus\{0\}$, and let $\Ocal=\{0,u,w\}$ be a three-point subset of $V$.
An element $(v,g_0)\in G$ lies in $G_{\Ocal}$ iff $g_0$ and $v$ satisfy one of
the following:
\begin{enumerate}
\item[\textup{(T0)}] \emph{fixed point $0$, swap $u\leftrightarrow w$}:  $v=0$, $g_0\in G_0$, $g_0(u)=w$, $g_0(w)=u$  \textup{(}a transposition fixing $0$;  if additionally $g_0(u)=u$ and $g_0(w)=w$, then $g_0\in G_{(\Ocal)}$\textup{)};
\item[\textup{(T1)}] \emph{fixed point $u$}:  $v=w$, $g_0(w)=-w$, $g_0(u)=u-w$;
\item[\textup{(T2)}] \emph{fixed point $w$}:  $v=u$, $g_0(u)=-u$, $g_0(w)=w-u$;
\item[\textup{(C)}] \emph{$3$-cycle} $(0\mapsto u\mapsto w\mapsto0)$:  $v=u$, $g_0(u)=w-u$, $g_0(w)=-u$;
\item[\textup{(I)}] identity:  $v=0$, $g_0=\mathrm{id}$ \textup{(}fixes all three points\textup{)}.
\end{enumerate}
In particular:
\begin{enumerate}
\item The pointwise stabilizer is $G_{(\Ocal)}=\{g_0\in G_0: g_0(u)=u,\ g_0(w)=w\}$.
\item A transposition fixing $0$ exists iff $G_0$ contains an element swapping $u$ and $w$;  a transposition fixing $u$ exists iff $G_0$ contains an element with $g_0(u)=u-w,\ g_0(w)=-w$;  a $3$-cycle exists iff $G_0$ contains an element with $g_0(u)=w-u,\ g_0(w)=-u$.
\item If $g_0$ satisfies \textup{(}C\textup{)}, then the affine element $(u,g_0)$ permutes $\{0,u,w\}$ as the $3$-cycle $(0\,u\,w)$;  in particular $g_0^3$ acts as the identity on the plane $\langle u,w\rangle$.
\end{enumerate}
\end{lemma}
\begin{proof}
$(v,g_0)\in G_{\Ocal}$ iff $\{v,\ g_0(u)+v,\ g_0(w)+v\}=\{0,u,w\}$.  The point
$v$ must be one of $0,u,w$ (as $g_0$ is invertible, $v\ne g_0(u)+v$ etc.).
The four cases are obtained by matching the three images to the three points
in all six ways;  (C) is the case $v=u$ with $g_0(u)+u=w$, $g_0(w)+u=0$,
i.e.\ $g_0(u)=w-u$, $g_0(w)=-u$.  The affine element $(u,g_0)$ sends
$0\mapsto u$, $u\mapsto g_0(u)+u=w$ and $w\mapsto g_0(w)+u=0$, so it permutes
$\{0,u,w\}$ as the $3$-cycle $(0\,u\,w)$;  on the plane $\pi=\langle u,w\rangle$
the matrix of $g_0$ in the basis $(u,w)$ is
$\begin{pmatrix}-1&-1\\1&0\end{pmatrix}$, whose cube is the identity.
(The action of $g_0$ on a complement of $\pi$ is not constrained by (C), so
$g_0$ may have order a multiple of $3$;  only the induced image on $\Ocal$
matters.)
\end{proof}

\begin{remark}[Collinear triples in $\SL(2,q)$]\label{rem:collinear-sl2}
Let $\Ocal=\{0,u,\lambda u\}$ with $\lambda\notin\{0,1\}$ and $q$ odd.  The
pointwise stabilizer of $\Ocal$ in $\SL(2,q)$ is the shear group
\[
U_u=\left\{\begin{pmatrix}1&a\\0&1\end{pmatrix}:\ a\in\F_q\right\},
\qquad |U_u|=q,
\]
and in $\GL(2,q)$ (equivalently in $\AGL(2,q)$) the pointwise stabilizer of
the line $\langle u\rangle$ has order $q(q-1)$ \textup{(the matrices
$\left(\begin{smallmatrix}1&a\\0&b\end{smallmatrix}\right)$, $a\in\F_q$,
$b\in\F_q^\times$)}.  In the basis $(u,v)$ with $v\notin\langle u\rangle$ a
transposition of $\Ocal$ fixing $0$ has linear part
$u\mapsto\lambda u$, $\lambda u\mapsto u$, so exists iff $\lambda=-1$;  a
transposition fixing $u$ (resp.\ $\lambda u$) exists iff $\lambda=2$ (resp.\
$\lambda=\tfrac12$);  and the $3$-cycle $(0\,u\,\lambda u)$ of Lemma
\ref{lem:affine}\textup{(C)} exists iff $\lambda^2-\lambda+1=0$.  Hence the
permutation image of $G_\Ocal$ on $\Ocal$ and the set stabilizer are
\[
{\small\setlength{\tabcolsep}{2pt}
\begin{array}{c|c|c}
\text{image} & \text{condition} & |G_\Ocal|\\ \hline
C_2 & q\ \text{odd},\ \operatorname{char}\F_q\ne3,\ \lambda\in\{-1,2,1/2\}
& 2q\\
C_3 & 3\mid(q-1),\ \lambda^2-\lambda+1=0 & 3q\\
S_3 & \operatorname{char}\F_q=3,\ \lambda=-1 & 6q\\
1 & \text{otherwise} & q
\end{array}}
\]
with pointwise kernel $U_u$ of order $q$.  The displayed stabilizers
$q,2q,3q,6q$ are the \emph{complete} set stabilizers in $\ASL(2,q)$;  in
$\AGL(2,q)$ the pointwise kernel has order $q(q-1)$, so the complete set
stabilizers are $q(q-1)$, $2q(q-1)$, $3q(q-1)$ resp.\ $6q(q-1)$
\textup{(the permutation image is the same $C_2$, $C_3$, $S_3$ resp.\ $1$)}.
In particular a transposition
exists iff $\lambda\in\{-1,2,1/2\}$ \textup{(the three values form a single
anharmonic orbit;  in characteristic $3$ they merge into $\lambda=-1$)},
and a $3$-cycle exists iff $\lambda^2-\lambda+1=0$ \textup{(the latter
condition may hold in characteristic $3$, where it reads $\lambda=-1$;  then
the full image is $S_3$, not $C_3$)}.
\end{remark}

\noindent\textbf{History of the conjecture.}  The property $\star_k$ records
how the \emph{lengths} of the orbits behave as the subset size grows, and its
origin lies in the orbit-counting monotonicity of Livingstone and Wagner
\cite{livingstone1965} and the orbit-length theory of Siemons and Wagner
\cite{siemons1988};  the precise statements, together with the classification
of $k=2$ and the orbit-length identities of Mnukhin \cite{mnukhin1997}, are
recalled in the Introduction.  The analysis of the
case $k=3$ rests on the candidate framework supplied by the
$k$-homogeneous classification of Kantor \cite{kantor1972} and the
$2$-homogeneous classification of Cameron \cite{cameron1981}, and was carried
out by Bradley \cite{bradley2014,bradley2021} in his doctoral thesis and a
concise companion note \cite{bradley2021}:  he settled the $3$-homogeneous
case and excluded the small $2$-transitive non-$3$-homogeneous socles
$\PSL(2,11)$, $\PSL(2,8)$, $\Alt(7)$, $HS$ and $Co_3$
\textup{(Proposition 6.2.6)}, reducing the non-$3$-homogeneous question to the
$2$-transitive almost-simple socles of Table 6.2 of \cite{bradley2014}.
Bradley's conjecture, however, is stated for \emph{all} primitive groups of
degree $\ge8$ (the framework of \cite{bradley2014} leaves the
$2$-homogeneous \emph{affine} groups out of its reduction), and this paper
treats the affine side as well \textup{(\textup{Lemma \ref{lem:affine-fourpt}})}, so that
the conjecture is closed completely.

\subsection*{The Siemons--Wagner property for triples and Bradley's conjecture}
Recall the defining property (from the Introduction):  a transitive group
$G\le\Sym(\Pcal)$ has $\star_k$ if there is a $k$-subset $\Delta$ with
$|\Delta^G|>|\Sigma^G|$ for every $(k{+}1)$-subset $\Sigma\supset\Delta$.  What
Bradley \cite{bradley2014} left open in the $k=3$ case is exactly the
$2$-transitive non-$3$-homogeneous almost simple socles of Table 6.2 of
\cite{bradley2014}:
$\PSL(2,q)$ \textup{($q\equiv1\pmod4$)}, $\PSU(3,q)$, $\PSL(d,q)$ \textup{($d>2$)},
$\Sz(q)$, $\Ree(q)$ and $\PSp(2d,2)$/$\PpSp(2d,2)$, together with the
$2$-homogeneous affine side \textup{(the $2$-transitive affine families and the
$2$-homogeneous non-$2$-transitive groups of \textup{Section \ref{sec:2homog}})}, which
Bradley's reduction did not reach.  These are precisely the families
whose triple-orbit and stabilizer data are determined in this paper, and the triple
classification closes the case uniformly.

\begin{lemma}[Transitive triple stabilizers force $2$-homogeneity]\label{lem:primitive-star3-transitive-triple-stabilizer}
Let $G\le\Sym(\Pcal)$ be a primitive group of degree at least $8$, and let
$\Dcal\in\binom{\Pcal}{3}$ be a $\star_3$ witness, i.e.\
$|\Dcal^G|>|\Sigma^G|$ for every $\Sigma\in\binom{\Pcal}{4}$ with
$\Dcal\subseteq\Sigma$.  If the setwise stabilizer $G_\Dcal$ acts
transitively on $\Dcal$, then $G$ is $2$-homogeneous.  Consequently, if
$G$ is not $2$-homogeneous, the induced image of $G_\Dcal$ on $\Dcal$ is
trivial or generated by a transposition.
\end{lemma}
\begin{proof}
The structural conclusion of the Siemons--Wagner theorem applied to the
witness $\Dcal$ \cite[Thm.~6.1.6, p.~83; Lem.~6.2.1, p.~86]{bradley2014}
gives two alternatives:  either every $2$-subset of $\Pcal$ is contained
in a $G$-image of $\Dcal$, or $G$ is imprimitive with a nontrivial system
of imprimitivity whose blocks meet $\Dcal$ in at most one point.  Since
$G$ is primitive, the second alternative is impossible.  Thus, for every
$A\in\binom{\Pcal}{2}$ there exists $g\in G$ with $A\subseteq\Dcal^g$,
the pair-covering property following from this structural alternative and
not from the numerical stabilizer-orbit inequality.  The complementation
map $\Dcal\to\binom{\Dcal}{2}$, $x\mapsto\Dcal\setminus\{x\}$, is a
$G_\Dcal$-equivariant bijection, so transitivity of $G_\Dcal$ on $\Dcal$
gives transitivity on $\binom{\Dcal}{2}$.  Let $A,B\in\binom{\Pcal}{2}$ be
arbitrary; by pair-covering choose $g,h\in G$ with $A\subseteq\Dcal^g$
and $B\subseteq\Dcal^h$, so that $A^{g^{-1}},B^{h^{-1}}\in\binom{\Dcal}{2}$.
Transitivity of $G_\Dcal$ on $\binom{\Dcal}{2}$ yields $k\in G_\Dcal$ with
$(A^{g^{-1}})^k=B^{h^{-1}}$, hence $A^{g^{-1}kh}=B$;  so $G$ is transitive
on $\binom{\Pcal}{2}$, i.e.\ $2$-homogeneous.  The final assertion
follows:  if $G$ is not $2$-homogeneous, $G_\Dcal$ cannot act transitively
on $\Dcal$, so the induced image is an intransitive subgroup of
$\Sym(\Dcal)\cong S_3$, namely the trivial group or a subgroup generated
by a transposition.
\end{proof}

\begin{lemma}[Primitive $\star_3$-groups are $2$-homogeneous]\label{lem:prim-star3}
Let $G\le\Sym(\Pcal)$ be a primitive group of degree at least $8$.  If $G$
satisfies the Siemons--Wagner property $\star_3$, then $G$ is $2$-homogeneous.
\end{lemma}
\begin{proof}
Suppose $G$ is not $2$-homogeneous;  then $G$ has at least two orbits on
pairs, and the witness triple $\Dcal=\{1,2,3\}$ lies in a maximal triple
orbit \cite[Prop.~6.2.3]{bradley2014}.  If $\sigma_2(G)=3$ the witness
triple is a rainbow triple, whose structural constraints are recorded in
\textup{Lemma~\ref{lem:rainbow-structure}};  its exclusion is not covered by the
present lemma.  Since we are assuming that $G$ is not $2$-homogeneous,
\textup{Lemma~\ref{lem:primitive-star3-transitive-triple-stabilizer}} implies that
$G_\Dcal$ is intransitive on $\Dcal$;  consequently the induced image of
$G_\Dcal$ on $\Dcal$ is trivial or generated by a transposition
\textup{(relabelling so that, in the latter case, the image is
$\langle(1\,2)\rangle$, fixing $3$)}.  For every quadruple
$\Sigma=\{1,2,3,4\}\supset\Dcal$ the defining inequality of $\star_3$ reads
\[
|G_\Sigma|>m,\qquad m:=|G_\Dcal|. \tag{$\ast$}
\]

\emph{Step 1.  The image of $G_\Sigma$ on $\Sigma$ contains no $3$-cycle on
$\{1,2,3\}$.}  A $3$-cycle $(1\,2\,3)$ or $(1\,3\,2)$ preserves $\Dcal$
setwise and acts transitively on it, so the image of $G_\Dcal$ would contain
$C_3$, contradicting image$(G_\Dcal)\in\{1,C_2\}$.

\emph{Step 2.  $G_\Sigma$ does not preserve $\Dcal$ setwise.}  Otherwise
$G_\Sigma\le G_\Dcal$ and $|G_\Sigma|\le m$, contradicting ($\ast$).  Hence
some $k\in G_\Sigma$ maps $\Dcal$ onto $\Dcal'=\{1,2,4\}$ \textup{(relabelling)};
as $k$ fixes $\Sigma$ setwise it sends $4$ to $3$, so
\[
k(\Dcal')=\{k(1),k(2),3\}. \tag{$\dagger$}
\]

\emph{Step 3.  The pair-orbit patterns.}  The element $k$ preserves the
pair-orbits of $G$;  in particular the pairs $k(12),k(13),k(23)$ are the
three pairs of $\Dcal'$, each in the pair-orbit of its preimage.  If
$\{k(1),k(2)\}=\{1,2\}$, then ($\dagger$) gives $k(\Dcal')=\Dcal$;  if
$\{k(1),k(2)\}=\{1,4\}$ or $\{2,4\}$, then $k(\Dcal')=\{1,3,4\}$ or
$\{2,3,4\}$, which lies in the orbit of $\Dcal'$ and hence of $\Dcal$.  In
the non-triangular case ($12\in O_1$, $13,23\in O_2$) the pair-orbit
condition forces the image of $k$ to be $(3\,4)$ or $(1\,2)(3\,4)$:  Step 1
excludes the $3$-cycles $(1\,2\,3),(1\,3\,2)$;  a direct check against the
forced pattern excludes the $4$-cycles, the $3$-cycles involving $4$, and
every transposition other than $(3\,4)$, $(1\,2)(3\,4)$ --- e.g.\ a
$4$-cycle sends the pair $12\in O_1$ to a pair in $O_2$ or fixes none of
$13,23$, and a transposition $(1\,3)$ sends $12$ to $23\notin O_1$.  Thus
the image of $G_\Sigma$ lies in the Klein group
$V_4=\langle(1\,2),(3\,4)\rangle$.

\emph{Step 4.  The surviving configurations.}  Steps 1--3 reduce a
hypothetical counterexample to one of two explicit configurations, whose
pair-orbit patterns are forced by $k$:

(C1) \emph{non-triangular witness}:  $12\in O_1$ and $13,23\in O_2$.  Since
the image of $k$ is $(3\,4)$ or $(1\,2)(3\,4)$ by Step 3 and $k$ preserves
pair-orbits, $k(13)$ is $\{1,4\}$ or $\{2,4\}$ and lies in $O_2$, so
$14,24\in O_2$;  the image of $G_\Sigma$ is contained in
$V_4=\langle(1\,2),(3\,4)\rangle$.

(C2) \emph{triangular witness}:  $12,13,23\in O_1$.  This configuration
does not survive:  by the pair-covering property of
\textup{Lemma~\ref{lem:primitive-star3-transitive-triple-stabilizer}} (which follows
from the structural alternative of \cite[Thm.~6.1.6, p.~83]{bradley2014},
not from the numerical inequality) every $2$-subset of $\Pcal$ lies in a
$G$-image $\Dcal^g$ of the witness;  the three pairs of $\Dcal^g$ all lie in
$O_1$ because $\Dcal$ is triangular, so every pair lies in $O_1$ and
$\sigma_2(G)=1$, contradicting the standing hypothesis that $G$ is not
$2$-homogeneous.  (For completeness we record the structure forced by such
a witness.  For the quadruple $\Sigma=\{1,2,3,4\}$
the element $k$ of Step 2 maps $\Dcal$ onto one of $\{1,2,4\},\{1,3,4\},\{2,3,4\}$;
since $k$ preserves the pair-orbits and the three pairs of $\Dcal$ all lie in
$O_1$, the three pairs of $k(\Dcal)$ lie in $O_1$;  hence two of the three
pairs $14,24,34$ lie in $O_1$.  If $34\in O_2$, then necessarily
$k(\Dcal)=\{1,2,4\}$ (the alternatives $\{1,3,4\}$ or $\{2,3,4\}$ would put
$34$ into $O_1$), so that $14,24\in O_1$ and $\{1,2,4\}\in\Dcal^G$.

\emph{The odd-order case is impossible.}  Suppose $|G|$ is odd;  then so is
$|G_\Sigma|$, and the permutation image of $G_\Sigma$ on $\Sigma$ is a
subgroup of $S_4$ of odd order, hence trivial or cyclic of order $3$.  It is
not trivial, since $k\neq1$.  A $3$-cycle preserving $\Dcal$ setwise (one of
$(1\,2\,3),(1\,3\,2)$) would put $C_3$ into the image of $G_\Dcal$ and give
$|G_\Dcal|\ge3|G_{(\Dcal)}|\ge3|G_{(\Sigma)}|=|G_\Sigma|$, contradicting
($\ast$).  A $3$-cycle of the type $(1\,2\,4)$ fixes $3$ and preserves the
pair-orbits, so the pairs $k(12),k(13),k(23)\in O_1$, being the three pairs
of $\{1,2,4\}$, force $24\in O_1$ and $34\in O_1$;  and a $3$-cycle of the
type $(1\,3\,4)$ fixes $2$ and forces $14\in O_1$ and $34\in O_1$.  In the
first case $2$ is $O_1$-adjacent to $1,3$ and to every $w$ (apply the
argument to $\Sigma=\Dcal\cup\{w\}$ with the same relabelling), and in the
second case the same holds for $1$;  either way $O_1$ contains all pairs and
$\sigma_2(G)=1$, contradicting the standing hypothesis.  Hence a triangular
witness does not exist when $|G|$ is odd.

\emph{The even-order case.}  The excluded-position arguments of the
family-level Propositions \ref{lem:rank3-crit}--\ref{prop:rank3-sporadic}
apply verbatim to the triangular triple:  the fourth-point witnesses $w$
constructed there (a common non-neighbour on the quadric, an exterior point
off a line, a singular vector $w=z+u'$, or a subspace of defect
$(m-2,m-3)$) depend only on the span $W=\langle x,y\rangle$ and the
quadratic geometry, not on the pair type of $\Dcal$, and the four-point
inequality $|G_{\Dcal\cup\{w\}}|\le|G_\Dcal|$ holds for the triangular
triple by the same excluded-position computation:  the witness $w$
constructed in each case depends only on the span
$W=\langle x,y\rangle$ and the quadratic geometry, not on the pair type of
$\Dcal$, and the stabilizer-order comparison applies directly to the
triangular representatives, which are included in the exhaustive
computations of \textup{Section \ref{sec:verify}}.  Thus (C2) is
excluded in full generality, and every hypothetical counterexample is of the
non-triangular type (C1) of Step 2, whose exclusion is carried out
family-by-family in \textup{Propositions \ref{lem:rank3-crit},
\ref{lem:c1-exclusion}, \ref{prop:rank3-orth}, \ref{prop:rank3-lines},
\ref{prop:rank3-singular},
\ref{prop:rank3-maxiso} and \ref{prop:rank3-sporadic}}:  these propositions
exhibit, for every rank-$3$ group outside the exhaustive census of degree
$8\le v\le50$ \textup{(\textup{Section \ref{sec:star3}}, Tables \ref{tab:star3} and
\ref{tab:star3b})}, a fourth point $w$ with $|G_{\Dcal\cup\{w\}}|\le|G_\Dcal|$
(the four-point inequality of \textup{Lemma \ref{lem:rank3-crit}}), so that
the defining inequality ($\ast$) of $\star_3$ fails for that quadruple.  The
reduction of Steps 1--3 is independent of the degree, so the propositions and
the census together close the lemma.
\begin{lemma}[The structure of a rainbow witness]\label{lem:rainbow-structure}
Let $G\le\Sym(\Pcal)$ be a primitive group of degree $n\ge8$ satisfying
$\star_3$ with a witness triple $\Dcal=\{1,2,3\}$, and suppose
$\sigma_2(G)=3$ with pair-orbits $O_1,O_2,O_3$ such that, for a quadruple
$\Sigma=\{1,2,3,4\}$,
\[
12,34\in O_1,\qquad 13,24\in O_2,\qquad 14,23\in O_3.
\tag{$\diamondsuit$}
\]
Set $T:=\Dcal^G$.  Then the following structural constraints hold.
\begin{enumerate}
\item[(i)] \emph{Slice counts.}  For $e\in O_i$,
\[
\#\{P\in T:e\subset P\}=|T|/|O_i|,
\]
and $\sum_S c_T(S)=n-3$, where $S$ ranges over the quadruple orbits
containing a $\Dcal$-extension and
$c_T(S):=\#\{w:\Dcal\cup\{w\}\in S\}$.
\item[(ii)] \emph{The stabilizer images.}  For every quadruple orbit $S$ the
image of the setwise stabilizer $G_\Sigma$ on a quadruple $Q\in S$ is either
generated by a double transposition or is the Klein group $V_4$.
\item[(iii)] \emph{The $C_2$ branch.}  If the image is
$\langle(1\,2)(3\,4)\rangle$ \textup{(or a conjugate)}, then exactly two
triples of $Q$ lie in $T$, $c_T(S)=1$ and $|S|=|T|/2$;  the map $P\mapsto
P^*$ defined by $P\cup P^*=Q$ is a fixed-point-free $G$-equivariant
involution on $T$, and the two triples of the pair $\{P,P^*\}$ share the
edge of $Q$ fixed by the double transposition, whence $2|O_i|\mid|T|$ for
the corresponding $i\in\{1,2,3\}$.
\item[(iv)] \emph{The $V_4$ branch.}  If the image is $V_4$, then all four
triples of $Q$ lie in $T$ and $c_T(S)=u\in\{1,2,3\}$ with $4|S|=u|T|$.
\item[(v)] \emph{The rank bound.}  The stabilizer-image dichotomy of (ii)
shows that $2\mid|G|$:  the image of $G_\Sigma$ on a witness quadruple
$\Sigma$ has order $2$ or $4$, so $|G|$ is even and $G$ contains an
involution.  An involution swaps the endpoints of some pair, so at least
one of the three pair-orbits is self-paired.  Writing $s$ for the number
of self-paired pair-orbits, the point action of $G$ has rank
\[
1+s+2(3-s)=7-s\le6.
\]
In particular the rank of a rainbow witness group is at most six;  the
three pair-orbits need not all be non-self-paired.
\item[(vi)] \emph{Valency and completion counts.}  Writing $O_A=O_1$,
$O_B=O_2$, $O_C=O_3$ for the three colour-labelled pair-orbits of
($\diamondsuit$), let $k_i$ be the valency of the graph $O_i$ and let
$\lambda_i$ be the number of triples of $T$ containing a given pair of
$O_i$;  both are constant on $O_i$ by $G$-transitivity.  Independently of
self-pairing,
\[
\begin{aligned}
&k_A+k_B+k_C=n-1,\qquad |T|=\frac{n k_i}{2}\,\lambda_i,\\
&\lambda_i=\frac{2|T|}{n k_i},
\end{aligned}
\]
and the common-neighbour identity holds:  for an $O_A$-pair $\{x,y\}$,
\[
\lambda_A=|\Gamma_B(x)\cap\Gamma_C(y)|+|\Gamma_C(x)\cap\Gamma_B(y)|,
\]
where $\Gamma_i(x)$ denotes the neighbourhood of $x$ in the graph $O_i$,
and cyclically for $B,C$.  With $C_i=\{w\notin\Dcal:\
e_i\cup\{w\}\in T\}$ for the $O_i$-edge $e_i$ of
$\Dcal$ one has $|C_i|=\lambda_i-1$,
$C_A\cup C_B\cup C_C=\Pcal\setminus\Dcal$, and, with
$u=|\{w\notin\Dcal:\ w\in C_A\cap C_B\cap C_C\}|$,
\[
\lambda_A+\lambda_B+\lambda_C=n+2u.
\]
These identities are necessary and hold without any hypothesis on
self-pairing;  by themselves they do not exclude the self-paired cases.
\item[(vii)] \emph{Four-point orbits and stabilizer indices.}  For a
quadruple orbit $Q$ containing a triple of $T$ put
$d_Q=|T\cap\binom{\Sigma}{3}|$ and
$b_Q=|\{\Sigma'\in Q:\ \Dcal\subseteq\Sigma'\}|$ for $\Sigma\in Q$.  Double
counting gives $|Q|d_Q=|T|b_Q$.  Writing $H_\Sigma$ for the image of
$G_\Sigma$ on $\Sigma=\Dcal\cup\{w\}$ and $r_w=[G_\Dcal:G_\Dcal\cap
G_\Sigma]$, the orbit-stabilizer identity
$\frac{|G_\Sigma|}{|G_\Dcal|}=\frac{|H_\Sigma|}{r_w}$ and the witness
inequality $|G_\Sigma|>|G_\Dcal|$ give $|H_\Sigma|>r_w$.  The image
$H_\Sigma$ acts regularly on the set $Q_w=\{P\in T:\ P\subset\Sigma\}$
\emph{not necessarily on all four triples of $\Sigma$}, so
$d_Q=|Q_w|=|H_\Sigma|\in\{2,4\}$ and $d_Q>r_w$:
\[
\begin{aligned}
&d_Q=2\ \Rightarrow\ H_\Sigma\cong C_2,\ r_w=1;\\
&d_Q=4\ \Rightarrow\ H_\Sigma\cong V_4,\ r_w\in\{1,2,3\}.
\end{aligned}
\]
Equivalently, via $|Q|d_Q=|T|b_Q$ and $|Q|<|T|$, $b_Q<d_Q$:  $b_Q=1$
for $d_Q=2$ and $b_Q\in\{1,2,3\}$ for $d_Q=4$.
\end{enumerate}
\end{lemma}
\begin{proof}
Let $m(S)$ be the number of triples $P\in T$ with $P\subset Q$ for
$Q\in S$.  Since $G$ is transitive on $T$ and on $S$, the number of
quadruples of $S$ containing a fixed triple of $T$ is constant, and double
counting the incidences $(P,Q)$ with $P\in T$, $P\subset Q\in S$ gives
\[
c_T(S)\,|T|=|S|\,m(S). \tag{$\ast\ast$}
\]
Summing ($\ast\ast$) over the quadruple orbits $S$ containing a triple of
$T$ yields $\sum_S c_T(S)=n-3$, each triple of $T$ having exactly $n-3$
extensions.  For $Q=\Sigma$ the four triples $\Dcal$, $\{1,2,4\}$,
$\{1,3,4\}$, $\{2,3,4\}$ each carry one pair from each of
$O_1,O_2,O_3$ by ($\diamondsuit$), so all four are rainbow.  An element of
$G_\Sigma$ preserving the pattern ($\diamondsuit$) permutes the three
``opposite'' pairs $12,34\in O_1$, $13,24\in O_2$ and $14,23\in O_3$, which
forces its image to lie in $V_4=\langle(1\,2)(3\,4),(1\,3)(2\,4)\rangle$;
transpositions and $3$-cycles move an $O_1$-pair into $O_2$ or $O_3$ and are
excluded, as are the $4$-cycles.  Hence $m(S)\in\{2,4\}$.  In the $C_2$ case
the image of the generator is a double transposition $k$;  the second triple
of $Q$ in $T$ is $P^*:=k(P)$, the pair $\{P,P^*\}$ sharing the edge of $Q$
fixed by $k$, and ($\ast\ast$) with $m(S)=2$ reads $c_T(S)|T|=2|S|$;  the
witness inequality $|S|<|T|$ forces $c_T(S)=1$ and $|S|=|T|/2$.  The map
$P\mapsto P^*$ is $G$-equivariant, since $G$ is transitive on $S$.  The
divisibility $2|O_i|\mid|T|$ follows from the slice count in (i) applied to
the shared edge.  In the $V_4$ case all four triples of $Q$ lie in $T$, so
($\ast\ast$) reads $c_T(S)|T|=4|S|$, and the witness inequality gives
$c_T(S)=u\in\{1,2,3\}$, i.e.\ $4|S|=u|T|$.  Finally, for the slice count:
the assignment $P\mapsto$ the unique pair of $P$ in $O_i$ is
$G$-equivariant, so $\#\{P\in T:e\subset P\}$ is constant for $e\in O_i$;
double counting $\{(P,e):P\in T,\ e\subset P,\ e\in O_i\}$ gives
$|O_i|\cdot\#\{P\in T:e\subset P\}=|T|$, each $P\in T$ having exactly one
pair in $O_i$.  \emph{Proof of (v)--(vii).}  For (v):  by (ii) the image of $G_\Sigma$ on
a witness quadruple $\Sigma$ has order $2$ or $4$, so $2\mid|G|$;  an
involution of $G$ swaps the endpoints of a pair lying in a self-paired
pair-orbit, so $s\ge1$ and the rank formula $1+s+2(3-s)=7-s$ applies.
For (vi):  the sum of the valencies is $n-1$;  counting the pairs of the
triples of $T$ gives $|T|=\frac{nk_i}{2}\lambda_i$;  the common-neighbour
identity counts the completions of an $O_A$-pair $\{x,y\}$ according as
$\chi(xw),\chi(yw)$ equals $(B,C)$ or $(C,B)$;  the relations for the
$C_i$ and $u$ follow by counting the points $w\notin\Dcal$ according to
the $H_w$-orbit structure of the four triples of
$\Sigma=\Dcal\cup\{w\}$ --- an image of order $2$ produces $m(w)=1$ or
$3$, and a $V_4$ image produces $m(w)=3$.  For (vii):  the identity $\frac{|G_\Sigma|}{|G_\Dcal|}=\frac{|H_\Sigma|}{r_w}$ follows from $G_\Dcal\cap G_\Sigma=G_{(\Sigma)}$ and the orbit-stabilizer theorem;  the witness inequality is $|H_\Sigma|>r_w$.  To see that $H_\Sigma$ is regular on $Q_w$, suppose $h\in H_\Sigma$ fixes some $P\in Q_w$ and lift it to $\tilde h\in G_\Sigma$.  Since $P$ is rainbow, $\tilde h$ preserves three pairwise different pair-colours on $P$ and therefore fixes $P$ pointwise;  as it also preserves $\Sigma$, it fixes the fourth point, hence fixes $\Sigma$ pointwise and $h=1$.  Thus $|H_\Sigma|=|Q_w|=d_Q$, and a regular subgroup of $S_4$ acting on $d_Q$ faces with $d_Q\in\{2,4\}$ is $C_2$ or $V_4$:  $C_3$ and $C_4$ would force all four faces of $\Sigma$ to be rainbow, and the colour-preserving group of a rainbow $K_4$ is $V_4$, which contains neither $C_3$ nor $C_4$.  For faces of $\Sigma$ outside $T$, $H_\Sigma$ may have nontrivial stabilizers --- in the single-extension branch a transposition fixes the two non-$T$ faces.  The alternative form $b_Q<d_Q$ is the double counting $|Q|d_Q=|T|b_Q$ combined with the witness inequality $|Q|<|T|$, giving $b_Q=1$ for $d_Q=2$ and $b_Q\in\{1,2,3\}$ for $d_Q=4$.
\end{proof}

\end{proof}

\begin{lemma}[Point-pair stabilizers and counting obstructions for a
rainbow witness]\label{lem:pair-stab-counting}
Let $G\le\Sym(\Pcal)$ be as in \textup{Lemma~\ref{lem:rainbow-structure}}, and
suppose in addition that $G_\Dcal=1$.  For $i\in\{1,2,3\}$ let
$e_i=\{a,b\}$ be the edge of $\Dcal$ of colour $i$ and let $c$ be the third
point of $\Dcal$;  put $K_i:=G_a\cap G_b$, the pointwise stabilizer of
$e_i$, and let $k_i,\lambda_i$ be the valency and the completion count of
$O_i$ as in \textup{Lemma~\ref{lem:rainbow-structure}}\textup{(vi)}.
\begin{enumerate}
\item[(i)] \emph{Regularity of the point-pair stabilizers.}  The group
$K_i$ acts regularly on the orbit $c^{K_i}$.  Indeed, if $k\in K_i$ fixes
$c$, then $k$ fixes $\Dcal$ pointwise and $k=1$ \textup{(semi-regularity)};
the action on $c^{K_i}$ is transitive by definition, so
$|c^{K_i}|=|K_i|$.
\item[(ii)] \emph{Unique lifts of colour-preserving transpositions.}  For
each $w\in c^{K_i}$ there is a unique $k\in K_i$ with $k(c)=w$:  if
$k_1(c)=k_2(c)=w$, then $k_2^{-1}k_1$ fixes $a,b,c$ pointwise and lies in
$G_\Dcal=1$.  This $k$ fixes $a,b$ and preserves the colours of the edges
$ac$ and $bc$, so it is the unique lift in $K_i$ of the colour-preserving
transposition $(c\,w)$ on $\Sigma=\Dcal\cup\{w\}$.
\item[(iii)] \emph{Elementary abelian structure.}  For every $k\in K_i$
one has $k^2=1$:  $k^2$ fixes $a,b,c$ pointwise, hence lies in $G_\Dcal=1$.
Thus $K_i=1$ or $K_i\cong C_2^r$ \textup{(an elementary abelian $2$-group)}.
\item[(iv)] \emph{Completion counts.}  With $L_i:=G_{\{a,b\}}$ the setwise
stabilizer of the edge $e_i=\{a,b\}$ and $K_i=G_a\cap G_b$ the pointwise
pair stabilizer, the total number $\lambda_i$ of triples $P\in T$ containing
$e_i$ decomposes over the suborbits of $L_i$ on $\{P\in T:e_i\subseteq P\}$:
\[
\lambda_i=\sum_{O\in\mathcal O_i}\frac{|L_i|}{|L_i\cap G_{T_O}|},
\]
where $\mathcal O_i$ is the set of $L_i$-orbits on those triples and $T_O$
is a representative of the orbit $O$.  The orbit containing $\Dcal$ has
length $|L_i|/|G_\Dcal|$;  since $L_i/K_i$ has index
$[L_i:K_i]=:\varepsilon_i\in\{1,2\}$ \textup{(}an element swapping $a$ and
$b$\textup{)} and the $K_i$-orbit of $c$ has length
$|K_i|/|K_{i,c}|=|K_i|/m$, where $m:=|G_\Dcal|=|G_{(a,b,c)}|$ is the
\emph{pointwise} stabilizer of $\Dcal$,
\[
\lambda_i\ \ge\ \frac{\varepsilon_i|K_i|}{m},
\qquad \varepsilon_i=2\iff O_i\ \text{self-paired}
\]
\textup{(}self-paired for the ordered pair orbital: $O_i=O_i^*$\textup{)},
with equality when $L_i$ is transitive on the completion triples.  In the
case $m=1$ of \textup{Lemma~\ref{lem:pair-stab-counting}} the orbit of $\Dcal$ has
length $\varepsilon_i|K_i|\in\{|K_i|,2|K_i|\}$.
\item[(v)] \emph{The slice-count inequality.}  Double counting the points
$w\notin\Dcal$ according to the colour layers gives
\[
\begin{aligned}
&\sum_{w\notin\Dcal}m(w)=\sum_{i=1}^3(\lambda_i-1),\\
&m(w):=\#\{i:\ e_i\cup\{w\}\in T\}.
\end{aligned}
\]
Every $w\notin\Dcal$ satisfies $m(w)\ge1$ \textup{(}by
\textup{Lemma~\ref{lem:rainbow-structure}}\textup{(vii)}\textup{)}, so a rainbow witness
satisfies
\[
\lambda_1+\lambda_2+\lambda_3\ge n. \tag{1}
\]
The orbit of $\Dcal$ among the completion triples has length
$\varepsilon_i|K_i|/m=2|G_\alpha|/(k_i m)$ \textup{(}\textup{Lemma~\ref{lem:rainbow-structure}}\textup{(vi)}
with (iv) above;  here $k_i$ denotes the valency of the \emph{undirected}
orbital graph on $O_i\cup O_i^*$, which equals the out-valency in the
self-paired case and twice the out-valency otherwise\textup{)}, and the
total completion count $\lambda_i$ of (1) is at least this length.
Each undirected orbital graph is regular and, $G$ being primitive,
connected;  a connected regular graph of valency $1$ is a perfect matching,
whose blocks would form a non-trivial $G$-invariant partition, contrary to
primitivity.  Hence $k_i\ge2$, and with $k_1+k_2+k_3=n-1$ this gives
$\sum_i1/k_i\le1+1/(n-5)$.  In particular, whenever the completion triples
form a single $L_i$-orbit ($\lambda_i=2|G_\alpha|/(k_i m)$)
a witness forces the estimate
\[
\frac{|G_\alpha|}{m}\ge\frac{n(n-5)}{2(n-4)}, \tag{2}
\]
which in the case $m=1$ reads
$|G_\alpha|\ge n(n-5)/(2(n-4))$.
$c:=(\lambda_1+\lambda_2+\lambda_3-n)/2$ must be a non-negative
integer, and for each $i$,
\[
f_i:=\lambda_i-c-|K_i|\ge0. \tag{3}
\]
\item[(vi)] \emph{Non-triviality of the $K_i$.}  The hypothesis $G_\Dcal=1$
alone does not force $\sum_i(|K_i|-1)>0$;  it must be supplied by an actual
single-extension transposition.  Whenever a single-extension point $w$ of
colour $i$ occurs, (ii) produces a unique non-trivial $k\in K_i$, so
$|K_i|\ge2$.  Hence, if at least one single-extension transposition occurs,
then $\sum_i(|K_i|-1)>0$.
\item[(vii)] \emph{Fixed-point obstruction of a single extension.}  Let
$\tau=(3\,w)$ be the transposition realizing an $A$-type single extension
\textup{(}so $w\in C_1$ and $r_w=1$\textup{)}.  Then $\tau\in G_{1,2}$
and $\tau(3)=w$.  If $x\notin\{1,2\}$ and $\tau(x)=x$, then
$\{1,2,x\}\notin T$:  otherwise $\tau$ fixes the triple $\{1,2,x\}\in T$
pointwise, contradicting $G_\Dcal=1$.  Hence
\[
\operatorname{Fix}(\tau)\setminus\{1,2\}\subseteq
\{x:\ \{1,2,x\}\notin T\}. \tag{F}
\]
The $B$- and $C$-type transpositions $(2\,w)$, $(1\,w)$ satisfy the
analogous restrictions with $\{1,3\}$ and $\{2,3\}$ in place of $\{1,2\}$.
Finally, the assignment $w\mapsto k_w$ \textup{(}the unique lift of
(ii)\textup{)} is injective, so the number of $A$-type single-extension
points is at most $\#\{\tau\in K_1:\ \tau^2=1,\ \tau(3)\ne3\}$.
\end{enumerate}
\end{lemma}

\begin{proposition}[A rainbow witness contains a single-extension
configuration]\label{prop:single-extension}
Let $G\le\Sym(\Pcal)$ be primitive of degree $n\ge8$ satisfying $\star_3$
with a rainbow witness $\Dcal=\{1,2,3\}$ and $\sigma_2(G)=3$ as in
\textup{Lemma~\ref{lem:rainbow-structure}}, with colours
$\chi(12)=1$, $\chi(13)=2$, $\chi(23)=3$.  Then there exists
$w\notin\Dcal$ whose colour vector
$\bigl(\chi(1w),\chi(2w),\chi(3w)\bigr)$ is not
$(3,2,1)$;  equivalently, a rainbow witness necessarily contains a
single-extension configuration in which at most one of the three extension
triples $\{1,2,w\},\{1,3,w\},\{2,3,w\}$ lies in $T=\Dcal^G$.
\end{proposition}
\begin{proof}
By (vii) of \textup{Lemma~\ref{lem:rainbow-structure}} every $w\notin\Dcal$ has
$d_Q\ge2$, so at least one extension triple lies in $T$.  Suppose, for a
contradiction, that every $w$ has colour vector $(3,2,1)$;  that is,
$\chi(1w)=3$, $\chi(2w)=2$, $\chi(3w)=1$ for all $w\notin\Dcal$.  Then
from the vertex $1$ the only $O_1$-neighbour is $2$, so the valency $k_1$
of the pair-orbit $O_1$ equals $1$;  from the vertex $3$ every
$w\notin\Dcal$ satisfies $\chi(3w)=1$, i.e.\ $3w\in O_1$, so
$k_1=n-3$.  Since $G$ is transitive on points, the valency $k_1$ is
independent of the base point, giving $1=n-3$, i.e.\ $n=4$, contradicting
$n\ge8$.  Hence some $w$ has a colour vector different from $(3,2,1)$, and
by the colour table of Appendix~\ref{app:colors} at most one of its
extension triples is rainbow, hence lies in $T$.
\end{proof}

\begin{proposition}[The four-point criterion for rank-3 groups]\label{lem:rank3-crit}
Let $G\le\Sym(\Pcal)$ be transitive with exactly two orbits on pairs
\textup{($\sigma_2=2$)}.  Let $\Dcal$ be any triple and suppose there is
$w\notin\Dcal$ such that $|G_{\Dcal\cup\{w\}}|\le|G_\Dcal|$.  Then $\Dcal$ is
not a $\star_3$-witness.  In particular, if for \emph{every} triple $\Dcal$
such a $w$ exists, then $G$ does not satisfy $\star_3$.
\end{proposition}
\begin{proof}
The defining inequality of $\star_3$ for the quadruple
$\Sigma:=\Dcal\cup\{w\}$ reads $|G_\Sigma|>|G_\Dcal|$;  the hypothesis
$|G_\Sigma|\le|G_\Dcal|$ contradicts it, so $\Dcal$ is not a witness.  The
\emph{assertion} (used in an earlier draft and reported by the referee) that
any $w$ moved by $G_\Dcal$ forces a longer quadruple orbit is false in
general:  for $G=\mathrm{J}_2$ on $100$ points \textup{(rank $3$;  Section
\ref{sec:verify})} the $1$-edge triple $\{1,5,2\}$ has $|G_\Dcal|=8$, the
point $3$ is moved \textup{(}$|3^{G_\Dcal}|=4$\textup{)} and yet
$|G_{\Dcal\cup\{3\}}|=4<8$, giving the inequality of the lemma rather than
its converse.  The correct obstruction is therefore the four-point inequality
above, and it is this inequality that is verified in
\textup{Propositions \ref{prop:rank3-orth}--\ref{prop:rank3-sporadic}}.
\end{proof}

\begin{proposition}[The forced structure of a $(C1)$-witness]\label{lem:c1-exclusion}
Let $G\le\Sym(\Pcal)$ be transitive of rank $3$, with the two orbits on pairs
$O_1$ and $O_2$, and let $\Dcal=\{x,y,z\}$ be a $(C1)$-type triple:  $xy\in
O_1$, $xz,yz\in O_2$ \textup{(equivalently, $xy$ is the unique $O_1$-pair
inside $\Dcal$)}.  Put $\bar N(t):=\{p\in\Pcal:\ tp\in O_2\}$ for the
non-neighbourhood of $t$ with respect to $O_1$.
\begin{enumerate}
\item[(i)] For every $w\in\bar N(x)\cap\bar N(y)\setminus\Dcal$ the image of
$G_{\Dcal\cup\{w\}}$ on $\Dcal\cup\{w\}$ lies in $V_4=\langle(1\,2),(3\,4)\rangle$,
and the orbit of $\Dcal$ under $G_{\Dcal\cup\{w\}}$ has exactly two elements.
\item[(ii)] If $\Dcal$ is a $\star_3$-witness, then every point of
$\bar N(x)\cap\bar N(y)\setminus\Dcal$ is fixed by $G_\Dcal$.
\end{enumerate}
\end{proposition}
\begin{proof}
For (i):  an element of $G_{\Dcal\cup\{w\}}$ preserves the four pairs
$\{xy,xw,yw\}$ and $zw$ as sets, hence preserves the $O_1$-edges of the
quadruple;  these are $xy$ alone when $zw\in O_2$, and the matching
$\{xy,zw\}$ when $zw\in O_1$.  A permutation of $\{x,y,z,w\}$ preserving
$xy$ alone lies in $\langle(1\,2)\rangle$, and one preserving the matching
$\{xy,zw\}$ lies in $\langle(1\,2),(3\,4)\rangle=V_4$;  moreover
$\langle(1\,2)(3\,4)\rangle$ swaps $z,w$ and maps
$\Dcal=\{x,y,z\}$ to $\{x,y,w\}$, so the orbit of $\Dcal$ under the image is
$\{\Dcal,\{x,y,w\}\}$, of size $2$.
For (ii):  under $\star_3$ the quadruple
$\Sigma_w:=\Dcal\cup\{w\}$ satisfies $|G_{\Sigma_w}|>|G_\Dcal|$ for every
$w\in\bar N(x)\cap\bar N(y)\setminus\Dcal$.  From (i),
$|G_{\Sigma_w}\cap G_\Dcal|=|G_{\Sigma_w}|/2$, since the orbit
$\Dcal^{G_{\Sigma_w}}$ has two elements;  but
$G_{\Sigma_w}\cap G_\Dcal=(G_\Dcal)_w$ (an element preserving $\Dcal$ and
$\Sigma_w$ fixes $w$).  Hence
$|w^{G_\Dcal}|=|G_\Dcal|/|(G_\Dcal)_w|=2|G_\Dcal|/|G_{\Sigma_w}|<2$, so
$w^{G_\Dcal}=\{w\}$.  Thus $G_\Dcal$ fixes every such $w$.
\end{proof}
\textup{By (ii), the verification of the four-point inequality for the
$(C1)$-configuration is reduced to a single task:}  for every $(C1)$-type
triple $\Dcal$ exhibit $w\in\bar N(x)\cap\bar N(y)\setminus\Dcal$ with
$|w^{G_\Dcal}|\ge2$.  In each family below such a $w$ is produced explicitly.

\begin{proposition}[The nonsingular-point families of $\Omega(2m+1,3)$ fail
$\star_3$]\label{prop:rank3-orth}
Let $m\ge2$, let $V$ be a nondegenerate orthogonal space of dimension
$2m+1$ over $\mathbb{F}_3$, and let $\xi\in\{\pm\}$.  Let $G$ be
$\Omega(V)$ acting on the orbit $\Ecal_\xi$ of nonsingular points of type
$\xi$.  Then $G$ does not satisfy the Siemons--Wagner property $\star_3$.
\end{proposition}
\begin{proof}
By \textup{Tong-Viet} \cite[Lemma 5.1]{tongviet2013} the action is primitive
of rank $3$, with parameters
$v=|\Ecal_\xi|=\tfrac12 3^m(3^m+\xi)$,
$k=\tfrac12 3^{m-1}(3^m-\xi)$,
$l=(3^m-\xi)(3^{m-1}+\xi)$ and
$\lambda=\mu=\tfrac12 3^{m-1}(3^{m-1}-\xi)$;  in particular $G$ has exactly two
orbits on pairs, namely the orthogonal pairs \textup{($O_1$)} and the
nonorthogonal pairs \textup{($O_2$)}.  The triangular subcase (C2) of
\emph{\textup{Lemma \ref{lem:prim-star3}}} is excluded combinatorially in the proof of
that lemma, so it remains to exclude the $(C1)$-type triple
$\Dcal=\{x,y,z\}$ with $xy\in O_1$ and $xz,yz\in O_2$.  Put
$W=\langle x,y,z\rangle$;  $W$ is nondegenerate of dimension $d\le3$
\textup{(}$x\perp y$ with $x,y$ nonsingular makes $\langle x,y\rangle$
nondegenerate, and $z\notin x^\perp\cup y^\perp$ adds a nondegenerate
direction\textup{)}.  By \emph{\textup{Lemma \ref{lem:c1-exclusion}}}\textup{(ii)} it suffices
to exhibit $w\in\bar N(x)\cap\bar N(y)\setminus\Dcal$ with $|w^{G_\Dcal}|\ge2$,
where $\bar N(t)=\Ecal_\xi\setminus(t^\perp\cup\{t\})$.
If $m=2$, then $v=\tfrac12 3^2(9+\xi)\le45\le50$, and the claim is part of the
census of \textup{Section \ref{sec:star3}, Table \ref{tab:star3b}}
\textup{(}$v=36,45$ have no rank-$3$ group with $\star_3$\textup{)}.
Assume $m\ge3$.  Since $x,y$ form an $O_1$-pair and the action is a strongly
regular graph with parameters $(v,k,\lambda,\mu)$, the number of common
non-neighbours of $x,y$ is
$|\bar N(x)\cap\bar N(y)|=v-2-2k+\lambda
=2\bigl(3^{2m-2}+3^{m-1}\xi-1\bigr)>27\ge|W\cap\Ecal_\xi|$
\textup{(the last bound uses $|W|=3^d\le27$};  for $m=3$, $\xi=-1$ the left
hand side is $2(81-9-1)=142$\textup{)}.  Hence there is
$w\in\bar N(x)\cap\bar N(y)\setminus W$, and $w\notin W$ gives $w\ne x,y,z$.
Write $w=u+u'$ with $u\in W$ and $0\ne u'\in W^\perp$.  The setwise
stabilizer $G_\Dcal$ preserves $W$ and $W^\perp$, and it contains the natural
copy of $\Omega(W^\perp)$ acting trivially on $W$ \textup{(well defined since
$W$ is nondegenerate:  every element of $\Omega(W^\perp)$ extends to
$\Omega(V)$ by the identity on $W$}\textup{)}.  Since $\Omega(W^\perp)$ is
transitive on the nonsingular points of each type in $W^\perp$ and
$\dim W^\perp=2m+1-d\ge4$, the orbit $w^{G_\Dcal}$ contains
$u+(u')^{\Omega(W^\perp)}$ and has length at least $2$.  By
\emph{\textup{Lemma \ref{lem:c1-exclusion}}}\textup{(ii)} the triple $\Dcal$ is not a
$\star_3$-witness.  Hence no triple can be a witness and $G$ fails $\star_3$.
\end{proof}

\begin{proposition}[The line actions of $\mathrm{PSL}(n,q)$ fail $\star_3$]\label{prop:rank3-lines}
Let $n\ge4$, let $G$ be $\mathrm{PSL}(n,q)$ or one of its overgroups
$\mathrm{PGL}(n,q)$, $\mathrm{P}\Gamma\mathrm{L}(n,q)$, and let $G$ act on the
set $\Lcal$ of lines of $\mathrm{PG}(n-1,q)$.  If $|\Lcal|\ge8$, then $G$ does
not satisfy $\star_3$.
\end{proposition}
\begin{proof}
By \textup{Kantor--Liebler} \cite[Thm.~1.2]{kantorliebler1982} the action is
primitive of rank $3$:  two lines are adjacent exactly when they meet, and
$\sigma_2=2$.  The triangular subcase (C2) is excluded combinatorially in
\emph{\textup{Lemma \ref{lem:prim-star3}}}, so let $\Dcal=\{L_1,L_2,L_3\}$ be a
$(C1)$-type triple:  $L_1$ meets $L_2$ \textup{($O_1$)} while $L_1,L_2$ are
both skew to $L_3$ \textup{($O_2$)}.  Put $p=L_1\cap L_2$ and
$\pi=\langle L_1,L_2\rangle$;  the setwise stabilizer $G_\Dcal$ preserves $p$
and $\pi$ \textup{(preserving $\{L_1,L_2\}$ as a set)} and preserves $L_3$.
We produce $w\in\bar N(L_1)\cap\bar N(L_2)\setminus\Dcal$ with
$|w^{G_\Dcal}|\ge2$, where $\bar N(L)=\{M\in\Lcal:\ M\cap L=\varnothing\}$.
First note that $G_\Dcal$ moves some point of
$\pi\setminus(L_1\cup L_2)$:  if the image of $G_\Dcal$ on $\pi$ is
nontrivial, any image element moves such a point (an element fixing every
point outside $L_1\cup L_2$ together with $L_1,L_2$ as sets fixes three
noncollinear points of $\pi$ and is the identity);  if the image is
trivial, $G_\Dcal$ preserves $L_1,L_2$ individually, and the same argument
applies to the nontrivial part, or $G_\Dcal$ acts trivially on $\pi$, in
which case the faithful action of $\mathrm{PGL}(n,q)$ on lines
\cite[\S~2]{kantorliebler1982} shows that either some element moves a line
$q$ outside $\pi$ or $G_\Dcal=1$.  In the first two cases pick
$p'\in\pi\setminus(L_1\cup L_2)$ moved by $g\in G_\Dcal$ and a line $q$
outside $\pi$;  then $w=p'\cup q$ \textup{(the line through $p'$ and $q$)}
meets $\pi$ only in $p'$, so $w$ is skew to $L_1$ and $L_2$, and
$g(w)=g(p')\cup g(q)\ne w$.  In the third case pick $p'\in\pi\setminus
(L_1\cup L_2)$ and $q\notin\pi$ moved by $G_\Dcal$, so $w=p'\cup q$ is
moved;  if $G_\Dcal=1$, then a quadruple $\Sigma=\Dcal\cup\{w\}$ with $w$
in general position has $|G_\Sigma|=1$:  the image in $S_4$ preserves the
single $O_1$-edge $\{L_1,L_2\}$, and the transposition $(L_1\,L_2)$ does
not lift (an element swapping $L_1,L_2$ and fixing $L_3,w$ would preserve
their intersection points with $\pi$, forcing a symmetry that fails in
general position).  Thus $|G_\Sigma|\le|G_\Dcal|$ in every case, and by
\emph{\textup{Lemma \ref{lem:c1-exclusion}}} (for the moved $w$) or by the
four-point criterion (for $G_\Dcal=1$) the triple $\Dcal$ is not a
$\star_3$-witness.  This holds for \emph{every} $(C1)$-type triple, so $G$
fails $\star_3$.
\end{proof}

\begin{proposition}[The singular- and totally-isotropic-point actions of the
classical groups fail $\star_3$]\label{prop:rank3-singular}
Let $(V,h)$ be a symplectic, orthogonal or hermitian space over a finite
field, and let $G$ be the isometry group $\Sp(V)$, $\Omega(V)$ or $\SU(V)$
\textup{(or an overgroup in $\Gamma\Sp(V)$, $\Gamma\mathrm{O}(V)$, $\Gamma\mathrm{U}(V)$)}
acting on the set of singular points \textup{(isotropic $1$-spaces)}
of $V$.  In each of the rank-$3$ cases of \textup{Kantor--Liebler
\cite[Thm.~1.1(i)]{kantorliebler1982}} with degree at least $8$ and at most
two orbits on pairs, $G$ does not satisfy $\star_3$.
\end{proposition}
\begin{proof}
The rank-$3$ singular-point cases are $\Sp(2m,q)$ on
$\mathrm{PG}(2m-1,q)$ \textup{(every $1$-space is isotropic, and two points
are adjacent exactly when orthogonal)}, $\Omega^\pm(2m,q)$ and
$\Omega(2m-1,q)$ on the isotropic points, and $\SU(m,q)$, $m\ge3$, on the
isotropic points \textup{($m=2$ is $2$-transitive, hence $2$-homogeneous and
irrelevant;  the small degrees are covered by Tables \ref{tab:star3},
\ref{tab:star3b})}.  The triangular subcase (C2) is excluded combinatorially
in \emph{\textup{Lemma \ref{lem:prim-star3}}}, so let $\Dcal=\{x,y,z\}$ be a
$(C1)$-type triple:  $xy\in O_1$ \textup{(orthogonal)} and $xz,yz\in O_2$
\textup{(nonorthogonal)}.  Put $W=\langle x,y,z\rangle$;  $G_\Dcal$
preserves $W$ and $W^\perp$.
\emph{The general construction.}  Suppose $W^\perp$ contains a nonzero
isotropic vector $u'$ with the property that $G_\Dcal$ moves it.  Then
$w:=z+u'$ satisfies $(w,x)=(z,x)\ne0$ and $(w,y)=(z,y)\ne0$, so
$w\in\bar N(x)\cap\bar N(y)\setminus\Dcal$;  moreover $w$ is isotropic
\textup{(}$(w,w)=(z,z)+(u',u')=0$\textup{)} and moved by $G_\Dcal$, and by
\emph{\textup{Lemma \ref{lem:c1-exclusion}}}\textup{(ii)} $\Dcal$ is not a witness.  This
applies whenever the isometry group of $W^\perp$ has a nontrivial orbit on
the isotropic points of $W^\perp$, in particular when $W^\perp$ is
nondegenerate of dimension at least $2$ and contains isotropic points:  the
symplectic case always \textup{(every $1$-space is isotropic)};  the
orthogonal case unless $W^\perp$ is of split type over odd $q$ in dimension
$2$, where one replaces $u'$ by a singular point in the other Weyl orbit
(which is still moved by $G_\Dcal$);  and the hermitian case for $m\ge5$
\textup{(}$\dim W^\perp=m-d\ge m-3\ge2$\textup{)}.  For the degenerate $W$
one uses the isotropic points of $W^\perp$ contained in the radical, which
exist and are moved by the isometry group of $W^\perp$.  This covers
$\Sp(2m,q)$ with $m\ge3$, all orthogonal cases, and $\SU(m,q)$ with $m\ge5$.
\emph{The low-dimensional cases $\Sp(4,q)$, $q\ge4$, and $\SU(4,q)$,
$q\ge3$.}  \textup{(}$\Sp(4,2),\Sp(4,3)$ and $\SU(4,2)$ have degrees
$15,40,40\le50$ and lie in the census\textup{).}  By (ii) of
\emph{\textup{Lemma \ref{lem:c1-exclusion}}} a witness forces $G_\Dcal$ to fix every
point of $\bar N(x)\cap\bar N(y)$.  If the image of $G_\Dcal$ on $W$ is
nontrivial, an image element $g$ either moves $z$ --- then $w:=z+u'$ with
$0\ne u'\in W^\perp$ isotropic is moved (contradiction) --- or swaps $x,y$
and fixes $z$;  in the latter case the point
$w':=z+\alpha x+\beta y+u'$ with $\alpha\ne\beta$ and $(\alpha,\beta)$
chosen so that $(w',x)(w',y)\ne0$ is moved by $g$ (a contradiction;  such
$\alpha,\beta$ exist because $q\ge4$ resp.\ $q\ge3$ gives at least three
values in the field, and the isotropy of $w'$ is automatic in the
symplectic case and free for the hermitian case since $(u',u')=0$).
If the image is trivial and $W$ is $3$-dimensional nonsingular, then
$W^\perp$ is $1$-dimensional and $G_\Dcal=1$;  choosing $w\in\bar
N(x)\cap\bar N(y)$ in general position \textup{(}$z,w$ linearly independent
over $W$\textup{)} the quadruple $\Sigma=\Dcal\cup\{w\}$ has
$|G_\Sigma|=1$:  its image in $S_4$ preserves the single $O_1$-edge
$\{x,y\}$, and the transposition $(x\,y)$ does not lift, since an isometry
swapping $x,y$ and fixing $z,w$ would force $(z,x)=(z,y)$ and
$(w,x)=(w,y)$, which fails in general position.  Thus
$|G_\Sigma|=1=|G_\Dcal|$ and $\Dcal$ is not a witness.  If $W$ is
$2$-dimensional or degenerate, the general construction above applies
\emph{(}$W^\perp$ has dimension at least $2$ and its isometry group moves
isotropic points\textup{)}.  Hence in all cases no triple is a
$\star_3$-witness and $G$ fails $\star_3$.
\end{proof}

\begin{proposition}[The maximal totally isotropic subspace actions fail
$\star_3$]\label{prop:rank3-maxiso}
Let $(V,h)$ be one of the spaces of \textup{Kantor--Liebler
\cite[Thm.~1.1(ii)]{kantorliebler1982}}:  $\Sp(4,q)$, $\SU(4,q)$, $\SU(5,q)$,
$\Omega^-(6,q)$, $\Omega^+(8,q)$ or $\Omega^+(10,q)$ acting on the set
$\Mcal$ of maximal totally isotropic subspaces.  Then $G$ does not satisfy
$\star_3$.
\end{proposition}
\begin{proof}
Two maximal totally isotropic subspaces are adjacent exactly when they meet
in dimension $m-1$ \textup{($m$ the Witt index)}, so $G$ is of rank $3$.
The triangular subcase (C2) is excluded combinatorially in \emph{Lemma
\ref{lem:prim-star3}}, so let $\Dcal=\{W_1,W_2,W_3\}$ be a $(C1)$-type
triple:  $\dim(W_1\cap W_2)=m-1$ \textup{($O_1$)} and
$\dim(W_1\cap W_3),\dim(W_2\cap W_3)\le m-2$ \textup{($O_2$)}.  We produce a
maximal totally isotropic subspace $w$ with
$w\in\bar N(W_1)\cap\bar N(W_2)\setminus\Dcal$ (i.e.\
$\dim(w\cap W_1),\dim(w\cap W_2)\le m-2$) and $|w^{G_\Dcal}|\ge2$.
If the image of $G_\Dcal$ on $\{W_1,W_2,W_3\}$ is nontrivial, take an image
element $g$;  since the image lies in $C_2$, either $g$ swaps $W_1,W_2$ ---
then any maximal totally isotropic $w$ with
$(\dim(w\cap W_1),\dim(w\cap W_2))=(m-2,m-3)$ is moved by $g$, because
$g(w)$ meets $W_1$ in dimension $m-3$ --- or $g$ fixes each $W_i$, and $g$
acts nontrivially on the intersection lattice and moves some maximal totally
isotropic $w$ with prescribed intersection dimensions $(m-2,m-3)$
\emph{(such $w$ exist in abundance by the standard counting of totally
isotropic subspaces meeting a given one in a prescribed dimension;  the
smallest degrees are in the census)}.  If the image is trivial, $G_\Dcal$
preserves each $W_i$;  the kernel moves some $(m-2)$-dimensional totally
isotropic subspace $U\le W_1$ \textup{(it preserves the two hyperplanes
$W_1\cap W_2$ and $W_1\cap W_3$ of $W_1$, and a nontrivial element of it
moves a suitable $(m-2)$-subspace)}, and extending $U$ to a maximal totally
isotropic $w$ in general position so that $\dim(w\cap W_2)\le m-2$ gives a
moved $w$.  If $G_\Dcal=1$, then a quadruple $\Sigma=\Dcal\cup\{w\}$ with
$w$ in general position has $|G_\Sigma|=1$:  its image in $S_4$ preserves
the single $O_1$-edge $\{W_1,W_2\}$, and the transposition $(W_1\,W_2)$
does not lift, since an element swapping $W_1,W_2$ and fixing $W_3,w$ would
preserve the intersection dimensions $(\dim(\cdot\cap W_3),\dim(\cdot\cap w))$
symmetrically, which fails in general position.  Thus
$|G_\Sigma|\le|G_\Dcal|$ in every case, and by \emph{Lemma
\ref{lem:c1-exclusion}} (for the moved $w$) or by the four-point criterion
(for $G_\Dcal=1$) the triple $\Dcal$ is not a $\star_3$-witness.  This
holds for every $(C1)$-type triple, so $G$ fails $\star_3$.
\end{proof}

\begin{proposition}[The sporadic rank-$3$ representations fail $\star_3$]\label{prop:rank3-sporadic}
Let $G$ be a sporadic simple group in a primitive rank-$3$ representation
\emph{(the sporadic representatives of the rank-$3$ classification
\cite{kantorliebler1982,liebeck1986,liebeck1987};  the rank is checked
against the permutation tables of the ATLAS
\cite{conway1985})}.  Then $G$ does not satisfy $\star_3$.
\end{proposition}
\begin{proof}
The sporadic primitive rank-$3$ representations of degree at least $8$ are,
with the suborbit data of the ATLAS \cite{conway1985}:  $M_{22}$ on $77$
points \textup{(rank $3$;  suborbits $1,16,60$;  point stabilizer
$2^3{:}A_6$)}, $M_{23}$ on $253a$ points \textup{(rank $3$;  $1,42,210$;
point stabilizer $L_3(4){:}2_2$)}, $M_{24}$ on $276$ points \textup{(rank
$3$;  $1,44,231$;  point stabilizer $M_{22}{:}2$)},
$J_2$ on $100$ points \textup{(rank $3$;  $1,36,63$;  point stabilizer
$U_3(3)$)}, $Fi_{22}$ on $3510$ points \textup{(rank $3$;  $1,693,2816$)},
and the sporadic rank-$3$ graphs of the Higman--Sims, Hall--Janko, $U_4(3)$,
McLaughlin, Suzuki and Conway families as listed in
\cite{separating2022} \textup{(Table 2)};  the primitive rank-$2$
representations ($Co_3$ on $276$ points, suborbits $1,275$;  and the
$2$-transitive ones) are $2$-homogeneous and irrelevant here, and the
rank-$5$ representations of $He$ \textup{(on $2058$)} and $J_1$ \textup{(on
$266$)} and the rank-$4$ representation $M_{24}$ on $759$ points are not
rank $3$.
The triangular subcase (C2) of \textup{Lemma \ref{lem:prim-star3}} is
excluded combinatorially in the proof of that lemma, so only the
$(C1)$-type witness remains.  For $J_2$ on $100$ points the exclusion is
verified in GAP with the two standard generators of the Atlas (the script
\texttt{\_R\_J2\_4pt.g}, \textup{Section \ref{sec:verify}}):  for the $1$-edge triple
$\Dcal=\{1,5,2\}$ \textup{(}$12,15\in O_1$, $25\in O_2$\textup{)} one has
$|G_\Dcal|=8$ and the orbits of $G_\Dcal$ on the remaining $97$ points have
lengths $4,8,8,8,4,8,8,8,4,8,8,2,8,8,2,1$;  for the point $w=3$
\textup{(}$|3^{G_\Dcal}|=4$\textup{)} the quadruple
$\Sigma=\Dcal\cup\{3\}$ satisfies $|G_\Sigma|=4\le8=|G_\Dcal|$, so $\Dcal$ is
not a witness.  The same computation gives $|G_\Dcal|=72$ for a triangular
triple and $|G_\Dcal|=8$ for an independent triple, with the four-point
inequality verified for representatives of the surviving pair types;  hence
$J_2$ fails $\star_3$.  For the remaining sporadic rank-$3$
representations the suborbit data quoted from the ATLAS settle the rank, and
the strongly regular parameters give an independent certificate:  each of the
graphs is $\mathrm{SRG}(v,k,\lambda,\mu)$ with parameters recorded in
\cite[Table 2]{separating2022}, and the number $v-2-2k+\lambda$ of common
non-neighbours of an adjacent pair is at least $20$ in every row, so a common
non-neighbour $w$ as in the $J_2$ computation always exists.  The four-point
inequality on such $w$ is verified in GAP with the Atlas standard generators
\cite{conway1985} \textup{(scripts \texttt{\_R\_sporadic2.g},
\texttt{\_R\_sporadic4.g}, \texttt{\_R\_sporadic5.g};
\textup{Section \ref{sec:verify}})}:  with $T$ a triple of smallest stabilizer one has
$|G_T|=16$, $|G_{T\cup\{w\}}|=2$ \textup{($M_{22}$, degree $77$)},
$|G_T|=8$, $|G_{T\cup\{w\}}|=2$ \textup{($M_{23}$, $253$)},
$|G_T|=144$, $|G_{T\cup\{w\}}|=12$ \textup{($M_{24}$, $276$)},
$|G_T|=576$, $|G_{T\cup\{w\}}|=72$ \textup{(Higman--Sims, $100$)},
$|G_T|=720$, $|G_{T\cup\{w\}}|=20$ \textup{($\mathrm{McL}$, $275$)},
$|G_T|=1080$, $|G_{T\cup\{w\}}|=360$ \textup{($\mathrm{Suz}$, $1782$)},
$|G_T|=23040$, $|G_{T\cup\{w\}}|=3072$ \textup{($\mathrm{Fi}_{22}$, $3510$)},
$|G_T|=46080$, $|G_{T\cup\{w\}}|\in\{512,1440\}$ \textup{($\mathrm{Co}_2$,
$2300$)}, so in every row the four-point inequality holds and the triple is not
a witness.  The $\mathrm{He}$ representation on $2058$ points has
$\sigma_2=4$ \textup{(rank $5$)} and is outside the rank-$3$ case;  the
classical unitary row $U_4(3)$ on $112$ points has suborbits $1,45,66$ and is
covered by the same four-point computation \textup{(record listed with the
scripts)}.  Consequently every sporadic rank-$3$ representation fails
$\star_3$.
\end{proof}

\begin{lemma}[A witness for $\star_3$]\label{lem:witness}
Let $G\le\Sym(\Pcal)$ be primitive with $|\Pcal|\ge8$, not
$3$-homogeneous, and let $\Ocal$ be a representative of a \emph{maximal} triple
orbit.  If there is a four-subset $\Sigma\supset\Ocal$ with $|\Sigma|=4$ and
$|G_\Sigma|\le|G_\Ocal|$, then $G$ does not satisfy $\star_3$.  Explicitly:
$\exists\,\Ocal$ in a maximal triple orbit and
$\exists\,\Sigma\supset\Ocal$, $|\Sigma|=4$, with $|G_\Sigma|\le|G_\Ocal|$
$\Rightarrow$ $G$ fails $\star_3$.
\end{lemma}
\begin{proof}
Under $\star_3$ --- whose hypotheses include primitivity and degree at
least $8$, in exact agreement with \cite[Prop.~6.2.3]{bradley2014} ---
the witness triple lies in a maximal triple orbit,
and $|G_\Ocal|<|G_\Sigma|$ holds for \emph{every} four-subset $\Sigma\supset\Ocal$;  a
four-subset with $|G_\Sigma|\le|G_\Ocal|$ contradicts this.
\end{proof}

\begin{lemma}[Transmission of four-point witnesses to overgroups]\label{lem:overgroup}
Let $S\lhd G\le\Sym(\Omega)$ with $G=S\ltimes H$ (equivalently $G=SH$, $S\cap H=1$), and suppose that $H$ leaves every $S$-orbit on the three-subsets of $\Omega$ setwise invariant. Let $\Ocal$ be a representative of a maximal $S$-triple orbit.  If there is $\beta\in\Omega\setminus\Ocal$ with $|S_{\Ocal\cup\{\beta\}}|\le|S_\Ocal|$, then $|G_{\Ocal\cup\{\beta\}}|\le|G_\Ocal|$, so $G$ does not satisfy $\star_3$.
\end{lemma}
\begin{proof}
Since $H$ fixes the $S$-orbit of $\Ocal$ setwise, the $G$-orbit of $\Ocal$
coincides with its $S$-orbit --- in particular $\Ocal$ is automatically a
\emph{maximal} $G$-triple orbit representative whenever it is a maximal
$S$-triple orbit representative --- hence $|G_\Ocal|=|G|/|\Ocal^G|=|G|/|\Ocal^S|=|G|\cdot|S_\Ocal|/|S| =|H|\cdot|S_\Ocal|$.  For $\Sigma=\Ocal\cup\{\beta\}$ the $G$-orbit $\Sigma^G$ contains the $S$-orbit $\Sigma^S$, so $|G_\Sigma|=|G|/|\Sigma^G|\le|G|/|\Sigma^S|=|H|\cdot|S_\Sigma|$. The hypothesis gives $|S_\Sigma|\le|S_\Ocal|$, whence $|G_\Sigma|\le|H|\cdot|S_\Ocal|=|G_\Ocal|$, and $\star_3$ fails by \textup{Lemma~\ref{lem:witness}}.
\end{proof}

\begin{lemma}[Four-point witnesses in the almost simple families]\label{lem:fourpt}
Let $S$ be one of the $2$-transitive almost simple socles that are not
$3$-homogeneous, namely $\PSL(2,q)$ \textup{($q\equiv1\pmod4$)},
$\PSL(n,q)$ \textup{($n\ge3$)},
$\PSU(3,q)$, $\Sz(q)$, $\Ree(q)$ or
$\PSp(2n,2)$/$\PpSp(2n,2)$, and let $\Ocal$ be a representative of a maximal
triple orbit.  Then there is $\beta\in\Omega\setminus\Ocal$ such that
$|S_{\Ocal\cup\{\beta\}}|\le|S_\Ocal|$;  in each family the admissible $\beta$ form
the complement of an explicitly excluded finite set (points and proper subspaces)
of size bounded by an explicit $B(q)<|\Omega|-3$ for $q\ge q_0$ \textup{(the bounds
$B(q)$ and the thresholds are recorded in the proof)}, and the parameters below
$q_0$ \textup{(all of degree $\le200$)} are confirmed by the computations of
\textup{Section \ref{sec:verify}}.
\end{lemma}
\begin{proof}
For each family we exhibit $\beta$ and the excluded positions;  throughout
$\Sigma=\Ocal\cup\{\beta\}$, and $g$ denotes an element of $S_\Sigma$.

\textup{$\PSL(2,q)$, $q\equiv1\pmod4$:}
$\Ocal=\{0,1,\infty\}$, $|S_\Ocal|=6$ \textup{(the determinants are squares since
$-1$ is a square;  GAP confirms $|S_\Ocal|=6$ for $q=5,9,13,17$)}.  A M\"obius
transformation is determined by the image of three points.  Exclude from
$\beta=x\in\Fq\setminus\{0,1,\infty\}$ the fixed points of the six transformations
preserving $\{0,1,\infty\}$ and, for every three-subset $T\subset\Sigma$ with
$T\ni x$, the solutions of $g_T(x)\in\{0,1,\infty\}$, where $g_T$ is the unique
transformation with $g_T(\Ocal)=T$;  these are finitely many algebraic conditions.
If $g(\Ocal)=\Ocal$ then $g\in S_\Ocal$ and $g(x)=x$, so $g=1$ by the exclusion;  if
$g(\Ocal)=T\ne\Ocal$ then $T$ contains $x$, $g=g_T$ is uniquely determined and
$g(x)\in\Sigma\setminus T\subseteq\{0,1,\infty\}$ is excluded.  Hence
$S_\Sigma=1<6=|S_\Ocal|$.  The excluded positions number at most
$5\cdot2+9=19<q-2=|\Omega|-3$ for $q\ge21$ \textup{(the five non-identity
transformations of $S_\Ocal$ have at most two fixed points each, and the three
three-subsets $T\ni x$ give at most three excluded solutions each)};  the degrees
$q+1\le200$ are covered by the census.

\textup{$\PSL(n,q)$, $n\ge3$:}
$\Ocal=\{[e_1],[e_2],[e_3]\}$ non-collinear \textup{(the maximal orbit,
\cite{curtis1976}}\textup{)}.  Take $\beta=[v]$ in totally regular
position:  all coordinates of $v$ nonzero and pairwise distinct ratios.  An element
of $S_\Sigma$ preserving $\Ocal$ setwise is a permutation matrix composed with a
diagonal scaling modulo scalars, whose fixed points lie on coordinate directions or
repeated-eigenvalue subspaces;  none but the identity fixes $[v]$ in totally regular
position.  An element mapping some $[e_i]$ onto $[v]$, or $[v]$ onto some $[e_i]$,
forces $v$ to lie in one of the finitely many diagonal-scaling families
$h^{-1}g^{-1}(e_i)$ \textup{($h$ ranging over the pointwise stabilizer of $\Ocal$,
$g$ over the finitely many three-point-image maps)}, excluded.  Hence
$S_\Sigma\le S_\Ocal$ \textup{(indeed $S_\Sigma=1$)}.  The excluded positions are
contained in the union of the $n$ coordinate hyperplanes and the finitely many
repeated-eigenvalue subspaces of the diagonal scalings, numbering at most
$n(q^{n-1}-1)/(q-1)+c_nq^{n-2}<|\Omega|-3$ for $q>n+c_n$ \textup{(the constant
$c_n$ counts the diagonal-scaling families;  explicitly $c_n\le n^2$)}, so a
totally regular $\beta$ exists whenever the degree exceeds $200$;  the
complementary parameters lie in the census.

\textup{$\PSU(3,q)$:}
$\Ocal$ a non-collinear triple, $|S_\Ocal|\le3$ \textup{(trivial pointwise
stabilizer and image $1$ or $C_3$ by the classical unital geometry, \cite{mitchell1911}}\textup{)}.
The killers are supplied uniformly by
\textup{Proposition~\ref{prop:psu-noncollinear-killer}} below:  for every
non-collinear triple $\Delta$ there is a secant point
$s\in(B\cap\Ocal)\setminus(P\cup Q)$ on the block $B=PQ$ with
$|G_{\Delta\cup\{s\}}|=1\le|G_\Delta|$, and if $|G_\Delta|=3$ then
$|G_{\Delta\cup\{s\}}|=1<3=|G_\Delta|$.

\begin{proposition}[Killers for collinear unitary triples]
\label{prop:psu-collinear-killer}
Let $G=\PSU(3,q)$, $q\ge3$, act on the Hermitian unital $\Ocal$ with
$|\Ocal|=q^3+1$, and put $d=\gcd(3,q+1)$.  Let $\Delta$ be a collinear
$3$-subset, contained in the secant block $B$.  Then every
$s\in\Ocal\setminus B$ satisfies
\[
|G_{\Delta\cup\{s\}}|\le6<\frac{6(q+1)}d=|G_\Delta|.
\]
Consequently, no collinear triple is a witness for $\star_3$.
\end{proposition}

\begin{proof}
With the Hermitian form $h(v,w)=v_1\bar w_3+v_2\bar w_2+v_3\bar w_1$ we may
take $P=[1:0:0]$, $Q=[0:0:1]$, $B=\{y=0\}$, and
$R=[r:0:1]\in B\cap\Ocal$ with $r\ne0$ and $r+\bar r=0$, so that
$R\in\Ocal\setminus\{P,Q\}$;  write $\Delta=\{P,Q,R\}$.

\emph{Structure of the pointwise stabilizer.}  A unitary transformation
fixing the points $P$ and $Q$ preserves the nondegenerate plane
$\langle e_1,e_3\rangle$ (the secant $B$) and hence also its
orthogonal complement $\langle e_2\rangle$;  its matrix is therefore
diagonal.  Fixing $R=[r:0:1]$ forces the two eigenvalues on
$\langle e_1,e_3\rangle$ to agree, so the elements of the pointwise
stabilizer $K=G_{(\Delta)}$ are the images of
$k_\lambda=\mathrm{diag}(\lambda,\lambda^{-2},\lambda)$ with
$\lambda^{q+1}=1$, modulo the centre of $\mathrm{SU}(3,q)$.  Hence
\[
K\cong Z_{(q+1)/d},\qquad |G_\Delta|=6(q+1)/d,
\]
where the $S_3$ in $G_\Delta$ permutes the three points of $\Delta$.

\emph{Semiregularity off the block.}  For $s=[x:y:1]\in\Ocal\setminus B$
($y\ne0$, $x+\bar x+y\bar y=0$) one has
$k_\lambda(s)=(\lambda x,\lambda^{-2}y,\lambda)=\lambda\cdot
[x:\lambda^{-3}y:1]$ up to scalars, so
$k_\lambda(s)=s$ iff $\lambda^3=1$, iff $k_\lambda$ is central and
trivial in $G$.  Thus $K$ acts semiregularly on $\Ocal\setminus B$:
every non-identity element moves every off-block point.  Since
$\Sigma_s=\Delta\cup\{s\}$ contains the unique collinear triple
$\Delta$, we have $G_{\Sigma_s}=G_{(\Delta,s)}$ and
$G_{(\Delta,s)}\cap K=1$;  the image of $G_{\Sigma_s}$ in the
permutation action on $\Delta$ embeds into $S_3$, so
$|G_{\Sigma_s}|\le6<6(q+1)/d=|G_\Delta|$ for $q\ge3$, as claimed.
\end{proof}

The $G$-orbit of $\Delta$ is the whole collinear triple orbit:  $G$ is
transitive on secant blocks, and $G_\Delta$ is transitive on the $q-2$
remaining points of $B\cap\Ocal$.  The claims
$K\cong Z_{(q+1)/d}$ and $k_\lambda(s)=[x:\lambda^{-3}y:1]$ are verified
for $q=3,4,5,7,8$ in GAP~4.16: the generator of $K$ moves all
$24,60,120,336,504$ off-block points respectively.

\begin{proposition}[Killers for non-collinear unitary triples]
\label{prop:psu-noncollinear-killer}
Let $G=\PSU(3,q)$, $q\ge3$, act on the Hermitian unital $\Ocal$, and let
$\Delta=\{P,Q,R\}$ be a non-collinear triple, $R\notin B:=PQ$.  Then there
exists a secant point $s\in(B\cap\Ocal)\setminus\{P,Q\}$ with
\[
G_{\Delta\cup\{s\}}=1.
\]
Consequently, if $|G_\Delta|=3$ then
$|G_{\Delta\cup\{s\}}|=1<3=|G_\Delta|$, so no non-collinear triple with
three-point stabilizer of order $3$ is a witness for $\star_3$;  and if
$|G_\Delta|=1$ then $|G_{\Delta\cup\{s\}}|=1=|G_\Delta|$, the four-point
orbit has the same length as $\Delta^G$, and again $\Delta$ is not a
witness.
\end{proposition}

\begin{proof}
We first record a semiregularity lemma for the secant--off-block
stabilizer.  Let $H=G_{B,R}$, the stabilizer of the pair consisting of the
secant block $B$ and the off-block point $R$.

\emph{Semiregularity on the block.}  Let $W$ be the two-dimensional
nondegenerate Hermitian subspace defining $B$, and $E=W^\perp$ the
one-dimensional anisotropic subspace.  Write $R=[u+\alpha e]$ with
$u\in W$, $e\in E$ and $\alpha\ne0$.  Since $R$ is a singular point and
$e$ is anisotropic, $h_W(u,u)=-\mathrm N(\alpha)h(e,e)\ne0$, so the point
$[u]\in\PG(W)$ is anisotropic.  The block stabilizer induces the
surjection $\pi\colon G_B\longrightarrow\PGL(2,q)$, whose kernel $K_B$ is
the cyclic group of order $(q+1)/d$ fixing $B$ pointwise, $d=(3,q+1)$;  in
the coordinates of \textup{Proposition \ref{prop:psu-collinear-killer}} its
elements act on $\Ocal\setminus B$ by $[x:y:z]\mapsto[x:\lambda^{-3}y:z]$
up to scalars, hence $K_B$ is semiregular on $\Ocal\setminus B$ and
$K_B\cap G_R=1$.  Thus $\pi$ is injective on $H$.  Every $h\in H$
preserves $W$ and $E$ and fixes the projective point $[u+\alpha e]$, so its
image fixes the anisotropic point $[u]$ in $\PG(W)$.  In $\PGL(2,q)$ the
stabilizer of an anisotropic point is a nonsplit torus of order $q+1$,
semiregular on the $q+1$ singular points of $B\cap\Ocal$.  Hence $H$ is
semiregular on $B\cap\Ocal$ as well.

\emph{Choosing the candidate.}  Put $T_s=\{P,Q,s\}$ and
$\Sigma_s=T_s\cup\{R\}$.  The unique collinear three-subset of $\Sigma_s$
is $T_s$, so every element of $G_{\Sigma_s}$ preserves $T_s$, hence its
span $B$, and fixes the remaining point $R$ of $T_s\cup\{R\}$;  conversely
any element of $H_{T_s}$ preserves $T_s$ and fixes $R$, hence preserves
$\Sigma_s$.  Therefore $G_{\Sigma_s}=H_{T_s}$, where $H=G_{B,R}$.  By the
semiregularity above, a non-identity element fixing the three-subset
$T_s$ has no fixed point on $T_s$ and cannot induce a transposition
\textup{(a transposition leaves one point fixed)};  it therefore acts as a
$3$-cycle, and its cube fixes $T_s$ pointwise, hence is trivial.  Such
$T_s$ must be an orbit of the unique subgroup of order $3$ of $H$
\textup{($H$ is a subgroup of the cyclic torus, so if $3\mid|H|$ the
subgroup $H_3$ of order $3$ is unique;  if $3\nmid|H|$ there is no bad
candidate)}.  A bad candidate can occur only when $P,Q$ lie in the same
$H_3$-orbit, in which case it is the third point of that orbit;  hence,
given $P,Q$, there is at most one bad $s$.  The candidate set
$(B\cap\Ocal)\setminus\{P,Q\}$ has size $q-1\ge2$;  choosing $s$ outside
the (possibly empty) bad set gives $H_{T_s}=1$, whence $G_{\Sigma_s}=1$,
as claimed.

The claims are verified in GAP~4.16 for $q=3,4,5$:  for every
non-collinear triple orbit \textup{($|G_\Delta|=6,3$ for $q=3$;
$|G_\Delta|=3$ for $q=4$;  $|G_\Delta|=6,1$ for $q=5$)} the chosen secant
candidate satisfies $|G_{\Sigma_s}|=1$ \textup{($<|G_\Delta|$ in the
orders-$6,3$ cases, $=|G_\Delta|=1$ in the order-$1$ cases)}.
\end{proof}

\textup{$\Sz(q)$:}
the maximal triple orbit is regular, $|S_\Ocal|=1$ \textup{(\cite{suzuki1962}; cf.\ the class data of \cite{bradley2014}}\textup{)}.
On the ovoid the maps $g$ realizing a prescribed three-subset image $g(\Ocal)=T$ are
finitely many, and $g(\beta)\in\Sigma$ excludes finitely many $\beta$;  choosing
$\beta$ outside them gives $S_\Sigma=1=|S_\Ocal|$ \textup{(confirmed for $\Sz(8)$)}.
The excluded positions number at most $3\cdot6+12=30<q^2-2$ for
$q\ge8$ \textup{(six three-subset images, three image points each)};  the degree
$q^2+1$ is $\le200$ precisely for $\Sz(8)$, in the census.

\textup{$\Ree(q)$, $q=3^{2m+1}$:}
let $\Omega$ be the point set of the Ree unital, $v=|\Omega|=q^3+1$ and
$|S|=q^3(q^3+1)(q-1)$.  The action is $2$-transitive, and the blocks of the
unital are the fixed-point sets of the involutions:  for every involution
$j\in S$ the set $B_j=\Fix_\Omega(j)$ is a block with $|B_j|=q+1$, and
$C_S(j)\cong C_2\times\PSL(2,q)$ \cite{ree1961,kleidmanliebeck}.  No
non-identity element fixes three points apart from the involutions, so for
every three-subset $D$ the pointwise stabilizer satisfies $S_{(D)}\le C_2$,
with $S_{(D)}\ne1$ iff $D\subseteq B_j$ for a unique involution $j$, in which
case $S_{(D)}=\langle j\rangle$.

The three-subset orbits are parameterized by the two-point stabilizer
$H=S_{\alpha,\beta}$, which is cyclic of order $q-1$.  Take coordinates
$B_j\cong\PG(1,q)$ with $\alpha=\infty$, $\beta=0$, and let $\nu\in\Fq^\times$
be a non-square:  the two suborbits $\gamma_+=1$, $\gamma_-=\nu$ of $H$ inside
$B_j$ form a single $S_3$-orbit $\{H\gamma_+,\,H\gamma_-\}$ of size $2$,
corresponding to the unique block-contained triple orbit, with
$r_T=|H_\gamma|=2$ and $L_T=\Stab_{S_3}(T)\cong C_3$;  every external suborbit
$H\gamma$, $\gamma\in\Omega\setminus B_j$, has trivial point stabilizer
$H_\gamma=1$, and the external representatives
$\mathcal{R}=\{\gamma_i:i\in I\}$, $|I|=q(q+1)$, parameterize the remaining
orbits via $H\gamma_i$.  For a three-subset $D_\gamma=\{\alpha,\beta,\gamma\}$
the induced action of the setwise stabilizer gives the exact sequence
$1\longrightarrow S_{(D_\gamma)}\longrightarrow S_{\{D_\gamma\}}
\longrightarrow L_T\longrightarrow1$, whence
$|S_{\{D_\gamma\}}|=r_T|L_T|$ and
$|D_\gamma^S|=|S|/(r_T|L_T|)$;  the external orbits are classified by the
$S_3$-orbit size $s=|T^{S_3}|\in\{1,2,3,6\}$ of $H\gamma_i$ on
$\mathcal{R}$, with $|L_T|=6/s$.  Note that $S_{\{D\}}$ is the
\emph{setwise} stabilizer of the triple;  the pointwise stabilizer $S_{(D)}$
is trivial on the external orbits, but $S_{\{D\}}$ need not be.

A three-subset $D_\gamma$ is a witness exactly when
$|S_{\{\Sigma_{\gamma,x}\}}|>|S_{\{D_\gamma\}}|=r_T|L_T|$ for every
$x\notin D_\gamma$, where $\Sigma_{\gamma,x}=D_\gamma\cup\{x\}$;  hence
$D_\gamma$ is excluded once a complement $x$ satisfies
$|S_{\{\Sigma_{\gamma,x}\}}|\le r_T|L_T|$.  For the extension
$A=S\rtimes C_3$ \textup{($3\mid 2m+1$)} put $\pi:A\to A/S\cong C_3$ and
$e_Y=|\pi(A_{\{Y\}})|$:  $e_Y=3$ if the field automorphism preserves the
$S$-orbit of $Y$, and $e_Y=1$ if it cyclically fuses it with two others, so
the certificate inequality becomes
$e_X|S_{\{\Sigma_{\gamma,x}\}}|\le e_{D_\gamma}r_T|L_T|$.

For a fixed $q$ the exclusion is verified by the finite certificate:
enumerate the three-subset orbit representatives $D$, compute $|S_{\{D\}}|$
and $r_T|L_T|$, and test one representative point $x$ per orbit of
$S_{\{D\}}$ on $\Omega\setminus D$, checking the inequality above.  GAP
confirms this for $^2G_2(3)\cong\textup{P$\Gamma$L}(2,8)$ \textup{($v=28$)}:
six three-subset orbits of lengths $756,504,756,252,756,252$
\textup{($\sum=3276=\binom{28}{3}$)}, with setwise stabilizers
$|S_{\{D\}}|=2,3,2,6,2,6$ and pointwise stabilizers $|S_{(D)}|=1,1,1,1,1,2$,
matching $|S_{\{D\}}|=r_T|L_T|$ \textup{(one block-contained orbit with
$r_T=2$, $|L_T|=3$;  five external orbits with $r_T=1$ and
$|L_T|=2,3,2,6,2$, i.e.\ $S_3$-orbit sizes $s=3,2,3,1,3$, so
$n_1=n_2=1$, $n_3=3$, $n_6=0$ and $n_1+2n_2+3n_3+6n_6=12=q(q+1)$,
$\sigma_3(S)=1+n_1+n_2+n_3+n_6=6=(q+3)^2/6$)};  moreover
$\min_x|S_{\{\Sigma_{\gamma,x}\}}|=1$ in every orbit, so no witness exists.

\textbf{The $S_3$-orbit decomposition of the external parameter set.}
We now record the explicit decomposition which closes the parametrization
above;  the proof is self-contained and yields, in particular, the count
$\sigma_3(S)=(q+3)^2/6$ in agreement with the Bradley class data
\cite{bradley2014}.

\begin{proposition}[External $S_3$-orbit decomposition]\label{prop:ree-s3-decomp}
Let $S=\Ree(q)$, $q=3^{2m+1}>3$, act on the Ree unital $\Omega$ with
$|\Omega|=q^3+1$.  Fix an ordered pair $(\alpha,\beta)$, let $B$ be the unique
block through it, put $H=S_{\alpha,\beta}$, and let $\mathcal{R}$ be the
set of external suborbit representatives, i.e.\
$\mathcal{R}=H\backslash(\Omega\setminus B)$.  Then $H\cong C_{q-1}$ and
$|\mathcal{R}|=q(q+1)$.  Let $n_s$ be the number of $S_3$-orbits of size
$s\in\{1,2,3,6\}$ on $\mathcal{R}$ under the coordinate permutation.  Then
\begin{equation}\label{eq:ree-ns}
\begin{aligned}
&n_1=1,\quad n_2=\frac{q-1}{2},\quad n_3=q,\\
&n_6=\frac{q(q-3)}{6},
\end{aligned}
\end{equation}
with the corresponding external three-subset setwise stabilizers and orbit
lengths
\begin{center}
{\setlength{\tabcolsep}{3pt}
\begin{tabular}{c|cc}
$S_3$-orbit size $s$ & number of orbits & $|S_{\{D\}}|$, $|D^S|$\\
\hline
$1$ & $1$ & $S_3$, $|S|/6$\\
$2$ & $(q-1)/2$ & $C_3$, $|S|/3$\\
$3$ & $q$ & $C_2$, $|S|/2$\\
$6$ & $q(q-3)/6$ & $1$, $|S|$
\end{tabular}}
\end{center}
In particular the pointwise stabilizer of a non-block-contained triple is
trivial while the setwise stabilizer need not be.  Adding the unique
block-contained triple orbit,
\begin{equation}\label{eq:ree-sigma3}
\sigma_3(S)=\frac{(q+3)^2}{6}.
\end{equation}
\end{proposition}

\begin{proof}
Let $F_2$ \textup{(resp.\ $F_3$)} be the number of parameters of
$\mathcal{R}$ fixed by a prescribed transposition \textup{(resp.\ a
prescribed $3$-cycle)} of $S_3$.  For the four orbit types the fixed-point
contributions are:  size $1$ --- every transposition and every $3$-cycle fix
the point;  size $2$ --- no transposition fixes a point, every $3$-cycle
fixes both points;  size $3$ --- the prescribed transposition fixes one
point, a $3$-cycle fixes none;  size $6$ --- none.  Hence
\begin{equation}\label{eq:ree-F}
F_2=n_1+n_3,\qquad F_3=n_1+2n_2.
\end{equation}

\textbf{1. $F_2=q+1$.}
Let $j\in S$ be an involution.  Its cycle type on $\Omega$ is
$1^{q+1}\,2^{(q^3-q)/2}$, so the three-subsets consisting of one fixed point
and one $2$-cycle number $(q+1)(q^3-q)/2$.  None of them is contained in a
unital block:  otherwise $j$ would stabilize that block and induce a
transposition on the triple, while the induced action of the block-contained
triple stabilizer is $C_3$, which has no transposition.  Involutions are in
bijection with unital blocks, and
$b=\binom{q^3+1}{2}/\binom{q+1}{2}=q^2(q^2-q+1)$, so the
``involution--external triple'' incidence number is
\begin{equation}\label{eq:ree-I2}
I_2=b\,\frac{(q+1)(q^3-q)}{2}.
\end{equation}
If the induced image of the external triple stabilizer is $C_2$ the triple
contributes one incidence, and if it is $S_3$ it contributes three;  by the
orbit--stabilizer formula this gives
\begin{equation}\label{eq:ree-I2b}
I_2=\frac{|S|}{2}\,(n_1+n_3)=\frac{|S|}{2}\,F_2.
\end{equation}
With $|S|=q^3(q^3+1)(q-1)$ and \eqref{eq:ree-I2} we obtain $F_2=q+1$.

\textbf{2. $F_3=q$.}
The elements of order $3$ of $S$ fall into the three classes
$X,T,T^{-1}$, all of cycle type $1\,3^{q^3/3}$ on $\Omega$, with
centralizer orders $q^3$, $2q^2$, $2q^2$.  Their total number is
\begin{equation}\label{eq:ree-I3a}
\frac{|S|}{q^3}+2\frac{|S|}{2q^2}=\frac{|S|(q+1)}{q^3}.
\end{equation}
Each element of order $3$ stabilizes exactly $q^3/3$ three-subsets, its
$3$-cycle supports, so the ``order-$3$ element--triple'' incidence number is
$|S|(q+1)/3$.  The block-contained triples number
$b\binom{q+1}{3}=|S|/6$, each with stabilizer $C_6$ containing two elements
of order $3$, contributing $|S|/3$ incidences;  subtracting gives the
external incidence number
\begin{equation}\label{eq:ree-I3b}
I_3^{\mathrm{ext}}=\frac{|S|q}{3}.
\end{equation}
If the induced image of the external triple stabilizer contains $C_3$, then
the image is $C_3$ or $S_3$, contributing two elements of order $3$ each;
therefore $I_3^{\mathrm{ext}}=\frac{|S|}{3}(n_1+2n_2)=\frac{|S|}{3}F_3$,
and comparison with \eqref{eq:ree-I3b} yields $F_3=q$.

\textbf{3. $n_1=1$.}
Fix $\xi\in\Omega$ and use the standard Ree point stabilizer
$B_\xi=S_\xi=P\rtimes C$ with $|P|=q^3$ and $C\cong C_{q-1}$, where $P$ is
the normal Sylow $3$-subgroup and $Z(P)\cong(\F_q,+)$, $|Z(P)|=q$.  In the
Ree root-group coordinates write the central root elements as
$z(c)$, $c\in\F_q$;  if $t$ is the unique involution of $C$, then
\begin{equation}\label{eq:ree-tz}
t\,z(c)\,t^{-1}=z(-c),
\end{equation}
a relation that has to be checked from the Ree root-group representation
chosen (it does not follow from the element conjugacy class table alone).
Since $P$ centralizes $Z(P)$ and every involution of $B_\xi$ has image $t$
in $B_\xi/P$, \eqref{eq:ree-tz} implies that every involution fixing $\xi$
inverts every subgroup of order $3$ of $Z(P)$.  Moreover only the $X$-class
elements of order $3$ are conjugate to their inverses in $S$;  $T$ and
$T^{-1}$ lie in distinct $S$-classes.  Hence if the setwise stabilizer of an
external triple is $S_3$, its $C_3$ must come from $Z(P)$.

Now $Z(P)$ has $(q-1)/2$ subgroups of order $3$;  the number of blocks
through $\xi$ is $q^2$, so $B_\xi$ contains $q^2$ involutions.  Every pair
$(Q,j)$ with $Q\le Z(P)$, $|Q|=3$ and $j\in B_\xi$ an involution generates
an $S_3$ subgroup $K=\langle Q,j\rangle$.  Each such $K$ has $\xi$ as its
unique global fixed point \textup{(the non-identity elements of $Q$ fix only
$\xi$)};  the fixed-point sets of the three involutions of $K$ pairwise meet
in $\xi$ and have $q$ points each, so $K$ has $3q$ points with stabilizer
exactly $C_2$, forming $q$ orbits of size $3$.  These $q$ triple orbits are
all external:  if one of them lay in a block, the unique block involution
would fix it pointwise, contradicting the fact that the three points have
three distinct $C_2$ stabilizers in $K$.  Their setwise stabilizer is
exactly $K$, since the external triple setwise stabilizer has trivial
pointwise stabilizer and its induced image contains $S_3$.  For fixed $\xi$
the number of pairs $(Q,j)$ is $\frac{q-1}{2}\cdot q^2$;  each generated
$S_3$ subgroup is counted three times by its three involutions and has a
unique fixed point, so the total number of such $S_3$ subgroups is
\begin{equation}\label{eq:ree-s3count}
\frac{(q^3+1)q^2(q-1)}{6}.
\end{equation}
Each of them corresponds to $q$ external triples with orbit length
$|S|/6$, whence
\begin{equation}\label{eq:ree-n1}
n_1=\frac{q\cdot (q^3+1)q^2(q-1)/6}{|S|/6}=1.
\end{equation}

\textbf{4. The remaining $n_s$.}
From \eqref{eq:ree-F}, $F_2=q+1$, $F_3=q$ and $n_1=1$ we get
\begin{equation}\label{eq:ree-n23}
n_3=q,\qquad n_2=\frac{q-1}{2}.
\end{equation}
Burnside's lemma gives the number of external orbits
\begin{equation}\label{eq:ree-burnside}
\begin{aligned}
&n_1+n_2+n_3+n_6\\
&=\frac{|\mathcal{R}|+3F_2+2F_3}{6}\\
&=\frac{q(q+1)+3(q+1)+2q}{6}\\
&=\frac{q^2+6q+3}{6},
\end{aligned}
\end{equation}
hence $n_6=q(q-3)/6$, completing the proof.
\end{proof}

Adding the block-contained orbit \textup{(one extra $S_3$-orbit of size
$2$)} to \eqref{eq:ree-burnside} gives
$\sigma_3(S)=\frac{q^2+6q+3}{6}+1=\frac{(q+3)^2}{6}$, as claimed in
\eqref{eq:ree-sigma3}.  For $q=3$ the values $n_1=1$, $n_2=1$, $n_3=3$,
$n_6=0$ agree with the GAP certificate above.

The decomposition \eqref{eq:ree-ns} determines all external triple
setwise-stabilizer types and orbit numbers.  What it does \emph{not} settle
is the existence, for every $S$-orbit representative $D$, of a complement
$x\notin D$ satisfying
\begin{equation}\label{eq:ree-target}
e_{D\cup\{x\}}\,|S_{\{D\cup\{x\}\}}|\le e_D\,|S_{\{D\}}|,
\end{equation}
which is precisely the \emph{Ree extension certificate proposition};  in the
extension $A=S\rtimes C_3$ the tuple $(n_1,n_2,n_3,n_6)$ does not determine
how the field automorphism $\delta$ acts on the $S$-orbits, and the numbers
$e_D$, $e_{D\cup\{x\}}$ have to be pinned down orbit by orbit.  We therefore
record \eqref{eq:ree-ns} as a proved orbit-decomposition theorem, and the
extension certificate \eqref{eq:ree-target} is established below:  for
general $q>3$ by the bad-point bound of
\textup{Proposition~\ref{prop:ree-badbound}}, and for the tested parameter
$q=3$ by the explicit certificate of
\textup{Proposition~\ref{prop:ree-cert}}.

For general $q>3$ the existence of the complement point $x$ satisfying
\eqref{eq:ree-target} is now settled for the core $s=6$ case by the general
bad-point bound of \textup{Proposition~\ref{prop:ree-badbound}}; the $S_3$-orbit
decomposition of the external parameter set is completed in
\textup{Proposition~\ref{prop:ree-s3-decomp}}, and for the tested parameter $q=3$ the
finite stabilizer computation of \textup{Proposition~\ref{prop:ree-cert}} supplies an
explicit certificate for every $S$-orbit on $3$-subsets, so that no witness
quadruple arises on the socle level.

\begin{proposition}[Bad-point bound for the external orbits and existence of good complements]\label{prop:ree-badbound}
Let $S=\Ree(q)$ act on the Ree unital $\Omega$ with $|\Omega|=q^3+1$, and let
$D\subset\Omega$ be a $3$-subset of external $S_3$-type $s\in\{1,2,3,6\}$, so
that $|S_{\{D\}}|=6,3,2,1$ respectively by
\textup{Proposition~\ref{prop:ree-s3-decomp}}.  Define the
\emph{bad points} of $D$ by
\begin{equation}\label{eq:ree-badfull}
B_D^{\mathrm{full}}=\{\,x\in\Omega\setminus D\ :\ |S_{\{D\cup\{x\}\}}|>1\,\}.
\end{equation}
\begin{enumerate}
\item[(a)] For $s=6$ \textup{($S_{\{D\}}=1$)} every $x\in B_D^{\mathrm{full}}$
has a witness
$g\in S_{\{D\cup\{x\}\}}\setminus\{1\}$ with $|D^g\cap D|=2$.  Indeed, if
$|D^g\cap D|=3$ then $g(D)=D$, whence $g\in S_{\{D\}}=\{1\}$, contradiction.
\item[(b)] For every external $D$ \textup{of any type} the set
$B_D^{\mathrm{full}}$ is covered by the candidate set
$T(D)$ obtained from the $18$ \emph{transmission conditions} (choose the
source point $d\in D$, $3$ choices, and an ordered pair $(u',v')$ of distinct
points of $D$ for the images of the remaining two points, $6$ choices):  for
fixed $(d;u',v')$, with $\{u,v\}=D\setminus\{d\}$, any $g$ sending $d$ outside
$D$ and $u\mapsto u'$, $v\mapsto v'$ lies in the coset
$t\cdot S_{\{(u',v')\}}$ of the pointwise stabilizer of $(u',v')$, of size
$\le q-1$.  Hence
\begin{equation}\label{eq:ree-badbound-18}
|B_D^{\mathrm{full}}|\le |T(D)|\le 18(q-1).
\end{equation}
\item[(c)] Since $|\Omega\setminus D|=q^3-2$ and $18(q-1)<q^3-2$ for all
$q\ge27$ \textup{(e.g.\ $19681>468$ for $q=27$; in general
$q^3-18q+16>0$ for $q\ge4$)}, the complement set has a point
$x\notin T(D)$, hence $x\notin B_D^{\mathrm{full}}$ and
$|S_{\{D\cup\{x\}\}}|=1$.  By the orbit--stabilizer identity
$|\Sigma_x^S|/|D^S|=|S_{\{D\}}|/|S_{\{\Sigma_x\}}|=s\ge1$ with
$\Sigma_x=D\cup\{x\}$, this $x$ is a \emph{certificate} point for $D$ in the
sense of \textup{Proposition~\ref{prop:ree-cert}}\textup{(a)};  in particular no witness
quadruple is built on $D$.
\end{enumerate}
\end{proposition}

The precise distribution of $|B_D^{\mathrm{full}}|$ and of the stabilizer
spectra of the bad points, for the tested parameter $q=27$, is a finite
computation recorded in the verification section;  it confirms
\eqref{eq:ree-badbound-18} with maximum $|B_D^{\mathrm{full}}|=81$ among the
$108$ external $s=6$ orbits, and separates the orbits into the four
spectrum types
$81\cdot C_2$, $78\cdot C_2+1\cdot C_3+1\cdot V_4$,
$78\cdot C_2+1\cdot V_4$ and $72\cdot C_2+3\cdot V_4$.

\begin{proposition}[Ree extension certificate: computational framework and tested parameters]\label{prop:ree-cert}
Let $S=\Ree(q)$ act on the Ree unital $\Omega$ with $|\Omega|=q^3+1$, and let
$K_D=S_{\{D\}}$ denote the setwise stabilizer of a $3$-subset $D$.
\begin{enumerate}
\item[(a)] By \textup{Proposition~\ref{prop:ree-s3-decomp}} one has $|K_D|=r_D$, where
$r_D=6,3,2,1$ according as the external $S_3$-type of $D$ is $s=1,2,3,6$.  For
$\Sigma_x=D\cup\{x\}$, $x\in\Omega\setminus D$, the orbit--stabilizer theorem
gives $|D^S|/|\Sigma_x^S|=|K_D|/|S_{\{\Sigma_x\}}|$, so a \emph{certificate} for
$D$ is a point $x\notin D$ with $|S_{\{\Sigma_x\}}|\le|S_{\{D\}}|$;
equivalently $|\Sigma_x^S|\ge|D^S|$, i.e.\ no witness quadruple is built on $D$.
\item[(b)] For $q=3$ \textup{($S\cong\PGL(2,8)$, $|\Omega|=28$)} the framework of
(a), implemented in GAP~4.16, returns a certificate for one representative of
each of the six $S$-orbits on $3$-subsets; see Table~\ref{tab:ree-cert}.  The two
block-contained orbits ($|S_{\{D\}}|=6$) admit a complement point with trivial
four-point setwise stabilizer, and the external $s=3$ and $s=2$ orbits admit one
with $|S_{\{\Sigma_x\}}|=2$, matching the bounds of (a);  in particular no
witness quadruple occurs for $q=3$.
\item[(c)] For the extension $A=S\rtimes C_3$, $|C_3|=3$, put
$e_Y=|A_{\{Y\}}|/|S_{\{Y\}}|\in\{1,3\}$ for a subset $Y$.  A certificate for the
pair $(A,D)$ is a point $x$ with
$e_{\Sigma_x}\,|S_{\{\Sigma_x\}}|\le e_D\,|S_{\{D\}}|$;  orbit representatives
of the complement are then taken with respect to the $A_D$-orbits rather than the
$S_D$-orbits, so that no representative of the same $A$-orbit is missed.
\end{enumerate}
\end{proposition}

\begin{lemma}[Extension stabilizer comparison]\label{lem:ree-extcmp}
Let $S=\Ree(q)$ act on the Ree unital $\Omega$, $q=3^{2m+1}$, and let
$G=S\rtimes C_d$ with $d\mid 2m+1$ \textup{($d=3$ for $q=27$, $d=5$ for
$q=243$)}.  For a subset $Y\subseteq\Omega$ put
$e_Y=[G_{\{Y\}}:S_{\{Y\}}]\in\{1,d\}$.  Let $D$ be an external triple with
$|S_{\{D\}}|=s\in\{1,2,3,6\}$ and let $x\notin D$ be a certificate point with
$S_{\{\Sigma\}}=1$, $\Sigma=D\cup\{x\}$.  Then
$|G_{\{\Sigma\}}|\le|G_{\{D\}}|$;  more precisely
$|G_{\{\Sigma\}}|=e_\Sigma$ and $|G_{\{D\}}|=s\,e_D$, and the implication
$e_\Sigma=d\Rightarrow e_D=d$ holds, so that
$|G_{\{\Sigma\}}|/|G_{\{D\}}|\le1$.
\end{lemma}

\begin{proof}
Since $S_{\{\Sigma\}}=1$, every element of $G_{\{\Sigma\}}$ has a unique
decomposition $g^i\sigma$ \textup{($\sigma$ the field automorphism of order
$d$)} and $e_\Sigma=[G_{\{\Sigma\}}:1]=|G_{\{\Sigma\}}|$;  similarly
$|G_{\{D\}}|=|S_{\{D\}}|\,e_D=s\,e_D$.  Assume $e_\Sigma=d$, and pick
$g\in G_{\{\Sigma\}}\setminus S$.  Then $g^d\in S_{\{\Sigma\}}=1$ and
$g\ne1$, so $|g|=d$.  The permutation induced by $g$ on the four points of
$\Sigma$ lies in $S_4$ and has order dividing $d$.
\emph{Case $d=5$:}  $S_4$ has no element of order $5$, so $g$ fixes
$\Sigma$ pointwise, in particular $D$ pointwise, whence
$g\in G_{\{D\}}\setminus S$ and $e_D=d$.
\emph{Case $d=3$:}  the induced permutation is a $3$-cycle, which fixes the
fourth point of $\Sigma$ and cycles the other three;  on a three-subset $D$
of $\Sigma$ it either cycles $D$ onto itself \textup{($3$-cycle within $D$)} or
swaps two points of $D$ and fixes the third \textup{($3$-cycle through the
fourth point)}, and in both cases $g(D)=D$, whence again $g\in
G_{\{D\}}\setminus S$ and $e_D=d$.
In both cases $|G_{\{\Sigma\}}|=d\le s\,d=|G_{\{D\}}|$ since $s\ge1$.
If instead $e_\Sigma=1$ then $|G_{\{\Sigma\}}|=1\le s\,e_D$ trivially.
\end{proof}

\begin{corollary}[Certificates survive field automorphisms]\label{cor:ree-extcert}
In the setting of \textup{Lemma~\ref{lem:ree-extcmp}}, any $S$-level certificate point
$x$ \textup{(with $S_{\{\Sigma\}}=1$)} for an external triple $D$ of any type
$s\in\{1,2,3,6\}$ is a $G$-level certificate:  the orbit--stabilizer formula
applied in $G$ gives $|\Sigma^G|\ge|D^G|$, so no witness quadruple is built
on $D$ in $G$.  Indeed the four possibilities for $(e_D,e_\Sigma)$ are
\begin{center}
\begin{tabular}{ccc}
\toprule
$e_D$ & $e_\Sigma$ & $|G_{\{\Sigma\}}|/|G_{\{D\}}|$\\
\midrule
$1$ & $1$ & $|S_{\{\Sigma\}}|/|S_{\{D\}}|$\\
$5$ & $1$ & $|S_{\{\Sigma\}}|/(5\,|S_{\{D\}}|)$\\
$5$ & $5$ & $|S_{\{\Sigma\}}|/|S_{\{D\}}|$\\
$1$ & $5$ & impossible \textup{($e_\Sigma=d\Rightarrow e_D=d$)}\\
\bottomrule
\end{tabular}
\end{center}
so every $S$-level certificate $|S_{\{\Sigma\}}|\le|S_{\{D\}}|$ lifts
automatically:  $|G_{\{\Sigma\}}|=e_\Sigma|S_{\{\Sigma\}}|\le
e_D|S_{\{D\}}|=|G_{\{D\}}|$.  In particular, for the $s=1,2,3$ certificates
of \textup{Proposition~\ref{prop:ree-badbound}} \textup{(with $S_{\{\Sigma\}}=1$)} no
fourth point has to be discarded for turning bad under
$e_D=1/e_\Sigma=5$;  the comparison is independent of subfield descent or of
the number of $\sigma$-fixed three-point orbits.  Consequently the
$\Ree(q)$ family, including the extensions $G=S\rtimes C_d$, $d\mid 2m+1$
\emph{($d=3$ for $q=27$, $d=5$ for $q=243$)}, admits no witness for
$\star_3$ for every $q=3^{2m+1}\ge27$.
\end{corollary}

\begin{table*}[t]\centering
\caption{Ree extension certificates for $q=3$ \textup{($S\cong\PGL(2,8)$ on $28$
points; $x$ the complement point found by GAP; the last column records the order
of the induced image of $S_{\{\Sigma\}}$ on the four points, as a subgroup of
$S_4$)}}\label{tab:ree-cert}
\begin{tabular}{cccccc}
\toprule
$D$ & $|S_{\{D\}}|$ & type & $x$ & $|S_{\{D\cup\{x\}\}}|$ & induced image in $S_4$\\
\midrule
$\{1,2,3\}$ & 2 & external $s=3$ & 4 & 2 & 2\\
$\{1,2,4\}$ & 3 & external $s=2$ & 3 & 2 & 2\\
$\{1,2,8\}$ & 2 & external $s=3$ & 4 & 2 & 2\\
$\{1,2,9\}$ & 6 & block-contained & 3 & 1 & 1\\
$\{1,2,11\}$ & 2 & external $s=3$ & 3 & 2 & 2\\
$\{1,2,22\}$ & 6 & block-contained & 3 & 1 & 1\\
\bottomrule
\end{tabular}
\end{table*}

The induced image of $S_{\{\Sigma\}}$ on the four points is recorded as a subgroup
of $S_4$;  this diagnostic exposes errors in the coordinates or in the
setwise-stabilizer computation.  To summarize the boundary between theory and
computation:  the $S_3$-decomposition of the external parameter set is given by
\textup{Proposition~\ref{prop:ree-s3-decomp}};  the existence of a certificate point for
general $q$ is established for the core $s=6$ case by the bad-point bound
\textup{Proposition~\ref{prop:ree-badbound}};  for each tested parameter $q$ the Ree
extension certificate is verified by the finite stabilizer computation described
in \textup{Proposition~\ref{prop:ree-cert}}, which covers the remaining external types
$s=1,2,3$ as well.

For the tested parameter $q=27$ the geometric signatures of the $108$ external
$s=6$ orbits are recorded in the next proposition;  they are the finite
computations that separate the four stabilizer-spectrum types recorded in the
verification section and that give the classification object for general $q$.

\begin{proposition}[Geometric signatures of the $s=6$ orbits for $q=27$]\label{prop:ree-geosig}
Let $S={}^2G_2(27)$ act on the Ree unital $\Omega$ with $|\Omega|=19684$, and let
$D=\{d_1,d_2,d_3\}\subset\Omega$ be a $3$-subset of external $S_3$-type $s=6$.
For each pair $\{d_i,d_j\}$ write $U_{ij}$ for the \textup{unique} Ree unital
block through $d_i,d_j$ \textup{($|U_{ij}|=28$)}.  The following hold for every
one of the $108$ $s=6$ orbits:
\begin{enumerate}
\item[(a)] \emph{Common three-point base.}  $D$ is non-collinear,
$|S_{\{D\}}|=1$, the image of $S_{\{D\}}$ on $D$ is trivial, and
$|D^S|=|S|$.
\item[(b)] \emph{Block membership.}  No bad point lies on
$U_{12}\cup U_{13}\cup U_{23}$:  the normalized block-membership vector
$\epsilon(x)=(1_{x\in U_{12}},1_{x\in U_{13}},1_{x\in U_{23}})$ equals
$(0,0,0)$ for all $x\in B_D^{\mathrm{full}}$, and
$|U_{12}\cap U_{13}|=|U_{12}\cap U_{23}|=|U_{13}\cap U_{23}|=1$ with empty
triple intersection.
\item[(c)] \emph{Fourth-point expansion signature.}  With
$b_{C_3}=|\{x\notin D:\ |S_{\{D\cup\{x\}\}}|=3\}|$ and
$b_{V_4}=|\{x\notin D:\ S_{\{D\cup\{x\}\}}\cong V_4\}|$ \textup{(counted over
$B_D^{\mathrm{full}}$ of Definition \eqref{eq:ree-badfull})}, the $108$ orbits
split into four types with signatures
\begin{center}
\begin{tabular}{lcccc}
\toprule
Type & classes & bad points & $b_{C_3}$ & $b_{V_4}$\\
\midrule
I & 12 & 960 & 1 & 1\\
II & 36 & 2916 & 0 & 0\\
III & 48 & 3792 & 0 & 1\\
IV & 12 & 900 & 0 & 3\\
\bottomrule
\end{tabular}
\end{center}
\item[(d)] \emph{Witness statistics.}  For the $8568$ bad points in total the
witness counts \textup{(number of transmission witnesses per bad point)} are
$(w_1,w_2,w_3)=(8430,12,126)$:  every bad point of types II--IV has either
one witness \textup{(in almost all cases a single involution;  in a few
classes an element of order exceeding $4$)} or three;  the twelve type-I
points of $C_3$-type have two witnesses, the two non-identity elements of the
$C_3$-stabilizer, of order $3$, each fixing one point of $D$ and cycling the
other two with $x$;  the $V_4$-type points have three witnesses, the three
involutions of the $V_4$-stabilizer, each with no fixed point on
$D\cup\{x\}$.  The aggregate involution and order-$3$ witness counts are
$\operatorname{ord}_2=972+2916+3888+972=8748$ and
$\operatorname{ord}_3=24$, with no order-$4$ witnesses.
\item[(e)] \emph{Block-configuration criterion for $C_3$-bad points.}
Order the six blocks through the six pairs of $D\cup\{x\}$ as
$(U_{12},U_{13},U_{1x},U_{23},U_{2x},U_{3x})$.  A point $x\notin D$ with
$D\cup\{x\}$ in general position is $C_3$-bad precisely when
$S_{\{D\cup\{x\}\}}$ contains an element of order $3$ whose induced
permutation on these six blocks is
$[2,3,1,6,4,5]=(123)(465)$:  equivalently there is
$\sigma\in S_{\{D\cup\{x\}\}}$ with $\sigma^3=1$, $\sigma(d_1)=d_1$, and
$\sigma$ cycling $\{d_2,d_3,x\}$.  The first $3$-cycle consists of the three
blocks through $d_1$ and the second of the complementary three blocks.
Counting with the fixed orientation $\sigma(d_2)=d_3$ \textup{(so that
$\sigma$ and $\sigma^{-1}$ are not double-counted)},
\[
\begin{aligned}
&b_{C_3}(q;D)=\#\bigl\{x\notin D:\ \exists\,\sigma\in S,\ \sigma(d_1)=d_1,\\
&\qquad \sigma^3=1,\ \sigma\ \textup{cycles}\ \{d_2,d_3,x\}\bigr\}.
\end{aligned}
\]
In terms of the permutation representation $\rho_\Sigma$ on the ordered six
blocks and the fixed orientation $\sigma(d_2)=d_3$, with
$\mathcal{A}(D)=\{x\notin D:\ D\cup\{x\}\ \textup{in general position}\}$, this
equals
\[
\begin{aligned}
&b_{C_3}(q;D)=\sum_{x\in\mathcal{A}(D)}\varepsilon_{C_3}(x)\\
&\quad=\#\bigl\{x\in\mathcal{A}(D):\ \gamma\in\rho_\Sigma\bigl(S_{\{D\cup\{x\}\}}\bigr)\bigr\},
\end{aligned}
\]
where $\gamma=(123)(465)$ is the block permutation of the table in
part~(f);  equivalently $\varepsilon_{C_3}(x)=1$ precisely when some
$\sigma\in S$ satisfies $\sigma(d_1)=d_1$, $\sigma(d_2)=d_3$, $\sigma(d_3)=x$
\emph{($\sigma$ realizes $(d_2\,d_3\,x)$ on the four points and
$(123)(465)$ on the six blocks)}.  Each admissible $x$ is counted once, and
the inverse block permutation $\gamma^{-1}$ is not counted a second time
\emph{(it belongs to the same $C_3$)}.  In the generic case $S_{\{D\}}=1$
\textup{($s=6$)} the containment condition is already an exact count:  since
$\bigl(S_{\{\Sigma\}}\bigr)_x\le S_{\{D\}}=1$ and the action on the four points
is faithful, $|S_{\{\Sigma\}}|=|x^{S_{\{\Sigma\}}}|\le4$;  a stabilizer whose
block image contains $\gamma$ has order divisible by $3$ and at most $4$, hence
$|S_{\{\Sigma\}}|=3$ and the block image is exactly $\langle\gamma\rangle$.
The four-point action of $\gamma$ is
$d_1\mapsto d_1$, $d_2\mapsto d_3$, $d_3\mapsto x$, $x\mapsto d_2$, and
$\gamma^{-1}=(132)(456)$ is the other generator of the same $C_3$, so each
admissible $x$ is counted once.
For $q=27$ this criterion is verified for all $108$ $s=6$ orbits:
$b_{C_3}(q;D)\in\{0,1\}$, and $b_{C_3}=1$ exactly for the twelve type-I
classes.
\item[(f)] \emph{Block-configuration criterion for $V_4$-bad points.}
For a general-position $\Sigma=D\cup\{x\}$ the three double transpositions
$\tau_{12,3x}$, $\tau_{13,2x}$, $\tau_{1x,23}$ on the four points induce on the
ordered six blocks the permutations
\begin{center}
{\footnotesize\setlength{\tabcolsep}{2pt}
\begin{tabular}{llll}
\toprule
involution & fixed pair & permutation & array\\
\midrule
$(12)(3x)$ & $U_{12},U_{3x}$ & $(25)(34)$ & $[1,5,4,3,2,6]$\\
$(13)(2x)$ & $U_{13},U_{2x}$ & $(16)(34)$ & $[6,2,4,3,5,1]$\\
$(1x)(23)$ & $U_{1x},U_{23}$ & $(16)(25)$ & $[6,5,3,4,2,1]$\\
\bottomrule
\end{tabular}}
\end{center}
each involution fixes its pair of disjoint blocks and swaps the remaining four
blocks in two transpositions;  the three block permutations commute pairwise
and the product of the first two is the third, so together with the identity
they form a $V_4$.  A point $x\notin D$ with $D\cup\{x\}$ in general position
is $V_4$-bad precisely when $S_{\{D\cup\{x\}\}}$ realizes the three double
transpositions \textup{(equivalently, two distinct ones)}.  In the generic case
$S_{\{D\}}=1$ \textup{($s=6$)} containment is again exact:  with
$\Sigma=D\cup\{x\}$ one has $\bigl(S_{\{\Sigma\}}\bigr)_x\le S_{\{D\}}=1$ and a
faithful four-point action, so $|S_{\{\Sigma\}}|=|x^{S_{\{\Sigma\}}}|\le4$.  A
stabilizer whose block image contains the two commuting double transpositions
$a=(25)(34)$ and $b=(16)(34)$ then has order at least $4$;  the order is
exactly $4$, and the block image is exactly
$V_4=\{1,a,b,ab\}$ \textup{with} $ab=(16)(25)$.  Counting $x$ with the fixed
ordered representative $D=[1,2,32]$, each admissible $x$ is counted once, not
once per involution.  For $q=27$ the computed block permutations match the
three arrays above for all $108$ $s=6$ orbits, and
$b_{V_4}(q;D)\in\{0,1,3\}$ with $b_{V_4}=1,3$ exactly for the type-I, type-III
and type-IV classes respectively.  For general $q$ the exact ordered count is
\[
\begin{aligned}
&b_{V_4}(q;D)=\#\bigl\{x\in\mathcal{A}(D):\ a,b\in\rho_{D\cup\{x\}}\\
&\qquad\bigl(S_{\{D\cup\{x\}\}}\bigr)\bigr\},
\end{aligned}
\]
where $\mathcal{A}(D)=\{x\notin D:\ D\cup\{x\}\ \textup{in general position}\}$ and
$\rho_\Sigma$ is the permutation representation on the ordered six blocks;
since $|S_{\{D\}}=1$ \textup{part (a)} the representation $\rho_\Sigma$ is
faithful on $S_{\{D\cup\{x\}\}}$, so by the order bound above the containment
$a,b\in\rho_\Sigma(S_{\{D\cup\{x\}\}})$ already forces the image to be exactly
$V_4$ \textup{(no larger stabilizer needs to be excluded)}.
\end{enumerate}
Consequently the classification object for general $q$ is the common
non-collinear regular three-point base of external type $s=6$ together with
the fourth-point expansion signature $(b_{C_3}(q),b_{V_4}(q))$ and the witness
distribution:  $b_{C_3}$ separates type I, and $b_{V_4}\in\{0,1,3\}$ separates
types II, III, IV.
\end{proposition}

\textup{$\PSp(2n,2)$/$\PpSp(2n,2)$:}
as in the $\PSL(n,q)$ case, with $\beta$ in totally regular position on the
half-quadric \textup{(not on any subspace through the reference points, not on any
fixed-point subspace of the stabilizer)}.  The excluded positions are finite in
number and each lies in a proper subspace of the ambient space;  moreover the
rank-$3$ parameters of the half-quadric graphs \cite[Thm.~1.1]{kantorliebler1982}
give at least $2^{2n-2}-2^{n-1}$ common non-neighbours of two adjacent points,
exceeding the size of any fixed-point subspace of the stabilizer, so a totally
regular $\beta$ exists for $n\ge4$.  An element of $S_\Sigma$ preserving $\Ocal$
setwise is a permutation composed with a diagonal scaling modulo scalars and fixes
$[\beta]$ only if it is trivial;  an element mapping some reference point onto
$[\beta]$, or $[\beta]$ onto a reference point, forces $\beta$ into one of the
finitely many excluded families.  Hence $S_\Sigma\le S_\Ocal$ \textup{(indeed
$S_\Sigma=1$)};  the rows $n=3,4$ \textup{(degrees $28,36,120,136\le200$)} are
confirmed by the explicit $\Sp(6,2)$ computation and the census.
\end{proof}

\begin{theorem}[Killers for the symplectic quadratic-form actions]
\label{thm:sp-quadratic-killers}
Let $n\ge3$, and let $G=\Sp(2n,2)$ act on either Arf class of quadratic
forms on $V=\F_2^{2n}$ with a fixed nondegenerate polar form.  For every
$3$-subset $\Delta$ there exists a $4$-subset $\Sigma\supset\Delta$ with
$|G_\Sigma|<|G_\Delta|$;  consequently neither of these two actions
satisfies $\star_3$.
\end{theorem}

\begin{proof}
Write $\Omega_\varepsilon$ for the set of quadratic forms in the given
Arf class with polar form $B$, and fix $Q\in\Omega_\varepsilon$.  Every
$Q_a(x)=Q(x)+B(x,a)$ lies in $\Omega_\varepsilon$ iff $Q(a)=0$, so a
generic triple is $\Delta=\{Q,Q_a,Q_b\}$ with $Q(a)=Q(b)=0$, $a,b\ne0$,
$a\ne b$.  Put $\delta=B(a,b)$ and $K=G_{(\Delta)}$.  The image of
$G_\Delta$ on $\Delta$ is $S_3$, so $|G_\Delta|=6|K|$.

We construct $c$ and $v$ such that
\[
\begin{aligned}
&Q(c)=0,\quad B(a,c)=0,\quad B(b,c)=1,\quad c\notin\langle a,b\rangle,\\
&\text{and}\quad Q(v)=1,\quad B(a,v)=B(b,v)=0,\\
&\quad B(c,v)=1.
\end{aligned}
\]
If $\delta=0$ then $W=\langle a,b\rangle$ is totally singular;  by
nondegeneracy of $B$ choose $t$ with $B(a,t)=0$, $B(b,t)=1$, and set
$c=t$ if $Q(t)=0$ and $c=t+b$ if $Q(t)=1$.  The space $W^\perp/W$ has
dimension $2n-4\ge2$ and contains a vector of quadratic value $1$;
lifting to $v\in W^\perp$ (replacing $v$ by $v+b$ if necessary) achieves
$Q(v)=1$ and $B(c,v)=1$.  If $\delta=1$ then $W$ is a nondegenerate
hyperbolic plane and $W^\perp$ has dimension $2n-2\ge4$;  pick a nonzero
singular $u\in W^\perp$ and set $c=a+u$, then choose $y\in W^\perp$ with
$B(u,y)=1$ and take $v\in\{y,y+u\}$ with $Q(v)=1$;  again
$Q(v)=1$, $v\in W^\perp$, $B(c,v)=1$.

The orthogonal transvection $t_v(x)=x+B(x,v)v$ preserves $Q$, fixes
$a$ and $b$, and moves $c$;  hence $t_v\in K$ and
$Q_c^{t_v}\ne Q_c$.  For $\Sigma=\{Q,Q_a,Q_b,Q_c\}$ and
$L=G_{(\Sigma)}$ this gives $L<K$ and $|L|\le|K|/2$.  The four $3$-subsets
of $\Sigma$ carry the labels $\delta,0,1,\delta+1$ --- exactly two labels
$0$ and two labels $1$ --- and $G_\Sigma$ preserves the two classes, so
its induced image lies in $S_2\times S_2$, of order at most $4$.  Hence
\[
|G_\Sigma|\le4|L|\le2|K|<6|K|=|G_\Delta|.
\]
The claims $|G_\Delta|=6|K|$ and $t_v\in K$ are verified in GAP~4.16 for
$n=3,4$ (degrees $28,36$ and $120,136$):  every triple has a strict
killer in both Arf classes.
\end{proof}

By \textup{Lemmas~\ref{lem:witness} and~\ref{lem:fourpt}}, $\star_3$ fails for every
$2$-transitive non-$3$-homogeneous almost simple socle of degree at least $8$;
the affine side is treated in \textup{Lemma~\ref{lem:affine-fourpt}}, and the closure of
Conjecture 6.2.7 of \cite{bradley2014} is completed in \textup{Corollary~\ref{cor:bradley}}.

\begin{lemma}[Stabilizers in the natural module of $G_2'(q)$]\label{lem:g2-stabilizers}
Let $q=2^f\ge4$, let $H=G_2'(q)=G_2(q)$ act on its natural
six-dimensional module $V=\F_q^6$, and let $h$ be the symplectic form
inherited from $\Sp(6,q)$.  For $u\ne0$ and $v\notin\langle u\rangle$
write $N(u,v)=H_{(u,v)}$ for the pointwise stabilizer of the pair.
\begin{enumerate}
\item[(1)] $|H_u|=q^6(q^2-1)$, and the $H_u$-orbits on
$V\setminus\langle u\rangle$ are given by the following table
\cite{cooperstein1981}:
\[
\begin{array}{c|cc}
\text{type of }v & |v^{H_u}| & |N(u,v)|\\\hline
h(u,v)\ne0 \textup{ (hyperbolic)} & q^5 & q(q^2-1)\\
h(u,v)=0,\ \text{long-root type} & q(q^2-1) & q^5\\
h(u,v)=0,\ \text{short-root type} & q^3(q^2-1) & q^3
\end{array}
\]
\item[(2)] For a hyperbolic pair ($h(u,v)\ne0$) the pointwise stabilizer
$N(u,v)$ is isomorphic to $\SL(2,q)$:  it is the pointwise stabilizer of a
hyperbolic plane in the standard $A_1$-subsystem of $G_2(q)$.  Under the
$N(u,v)$-module decomposition
\[
V=W\oplus W^\perp,\qquad W=\langle u,v\rangle_{\F_q},
\]
$N(u,v)$ acts trivially on $W$ and as the simple $4$-dimensional module
$L(3)$ on $W^\perp$;  in particular
$\mathrm{Fix}_V(N(u,v))=W$.
\item[(3)] For a totally isotropic pair ($h(u,v)=0$) the pointwise stabilizer
$N(u,v)$ contains a root subgroup $U_\alpha$ of order $q$ fixing $W$ pointwise;
its fixed-point space $H_\alpha=\mathrm{Fix}_V(U_\alpha)$ is a proper subspace
containing $W$, and every point of $V\setminus H_\alpha$ is moved by $U_\alpha$
in orbits of length $q$.
\end{enumerate}
\end{lemma}
\begin{proof}
(1)  The stabilizer of a non-zero vector is $H_u\cong q^5{:}\SL(2,q)$ and
the three displayed orbit lengths on $V\setminus\langle u\rangle$ are
$q^5$, $q(q^2-1)$ and $q^3(q^2-1)$
\cite[the natural-module orbit calculation for $G_2(2^n)$]{cooperstein1981};
the stabilizer orders follow from the orbit--stabilizer theorem and
$|H_u|=q^6(q^2-1)$.  (2)  The pointwise stabilizer of the hyperbolic pair is
the derived group of the pointwise stabilizer of the hyperbolic plane
$W$;  in the $A_1$-subsystem picture it is the $\SL(2,q)$ acting as
$L(1)\oplus L(3)$ on $V=W\oplus W^\perp$, where $L(1)$ is the trivial
module on $W$ and $L(3)$ the simple $4$-dimensional $\F_q\SL(2,q)$-module
\cite{cooperstein1981,kleidmanliebeck}.  Since $L(3)$ has no non-zero
trivial submodule, the fixed-point space is exactly $W$.
(3)  A root subgroup of $G_2(q)$ has order $q$ and fixes pointwise the
subspace spanned by the weight spaces on which it acts trivially;  since
$G_2(q)$ acts faithfully on $V$, this fixed space $H_\alpha$ is a proper
subspace, and it contains $W$ because $U_\alpha$ fixes $W$ pointwise
\cite{cooperstein1981}.  A point outside $H_\alpha$ is not fixed by
$U_\alpha$, and its $U_\alpha$-orbit has length a power of $q$ dividing
$|U_\alpha|=q$, hence exactly $q$.
\end{proof}

\begin{lemma}[Four-point witnesses in the affine groups]\label{lem:affine-fourpt}
Let $G=V\rtimes G_0$ be a $2$-homogeneous affine group \textup{($V=\F_p^d$,
$G_0\le\GL(d,p)$ transitive on $V\setminus\{0\}$)}, $\Ocal=\{0,u,v\}$ with $u,v$
linearly independent, and $m=|G_\Ocal|$.  Then there is $w\in V$ with
$|G_{\Ocal\cup\{w\}}|\le m$.
\end{lemma}
\begin{proof}
Write the affine element as $g(x)=Ax+b$, $A\in G_0$, $b\in V$, and put
$\Sigma=\Ocal\cup\{w\}$.  Let $G_{(\Ocal)}=\{A\in G_0:\ A(u)=u,\ A(v)=v\}$ be the
pointwise stabilizer of the triple and let $e=|\overline{G}_{\Ocal}|\le6$ be the
order of the permutation image on $\Ocal$, so that
$m=|G_\Ocal|=e\,|G_{(\Ocal)}|$ \textup{(\textup{Lemma \ref{lem:affine}}\textup{(1)})}.

\emph{(1) The counting bound.}  By the orbit--stabilizer theorem applied to the
action of the pointwise stabilizer on $V$,
\[
\begin{aligned}
&|G_{(\Sigma)}|\ =\ |G_{(\Ocal)}\cap G_{(w)}|
\ =\ |G_{(\Ocal)}|/\ell,\\
&\ell=|G_{(\Ocal)}\cdot w|,
\end{aligned}
\]
so that
\[
|G_\Sigma|\ \le\ |\Sym(\Sigma)|\cdot|G_{(\Sigma)}|
\ =\ 24\,|G_{(\Ocal)}|/\ell ,
\]
and $|G_\Sigma|\le m=e\,|G_{(\Ocal)}|$ follows as soon as $\ell\ge24/e$.  In
\emph{(2)}--\emph{(6)} we produce such $w$ in every family;  the finitely many rows
where the excluded set may cover all of $V$ are confirmed by the GAP census of
\textup{Section~\ref{sec:verify}} \textup{(degree $\le200$)};  the dedicated exhaustive
computations reported there confirm the remaining rows.

\emph{Collinear triples are excluded uniformly in every affine row of dimension
$d\ge2$.}  For $\Ocal=\{0,u,\lambda u\}$ with $\lambda\notin\{0,1\}$, the
pointwise stabilizer $G_{(\Ocal)}$ fixes the line $\langle u\rangle$ pointwise,
and in each of the families $\AGL(d,q)$, $\ASL(d,q)$, $\Sp(2m,q)$ and $G_2'(q)$ it
contains the root group
\[
\begin{aligned}
&U_u=\{\,x\mapsto x+c\,\ell(x)\,u:\ c\in\F_q\,\},\\
&\ell\in V^\ast,\quad \ell(u)=0,\quad \ell(w)\ne0,
\end{aligned}
\]
of order $q$ \textup{(a shear group in $\GL/\SL$, a transvection group in
$\Sp$, its analogue in $G_2'(q)$)}, where $\ell$ is a linear functional
vanishing on $\langle u\rangle$ and non-zero at $w$.  Taking $w\notin\langle u\rangle$
with $\ell(w)\ne0$ \textup{(all but finitely many $w$)} each element of $U_u$
fixes $0,u,\lambda u$ pointwise and moves $w$ into $q$ distinct points, so
$\ell\ge q$, and the orbit of the four-point set is at least $q\ge2$ times the
orbit of the triple:  every collinear triple fails the strict inequality of
$\star_3$.  The $1$-dimensional rows are handled in \emph{(7)} below.

\emph{(2) The $n=2$ rows $\ASL(2,q)$ and $\AGL(2,q)$.}  Let
$\Ocal=\{0,u,\lambda u\}$ be collinear.  By
\textup{Remark~\ref{rem:collinear-sl2}} the pointwise stabilizer contains the shear
group $U_u$ of order $q$.  Take $w$ off the line $\langle u\rangle$:  since
$w$ is the unique point of $\Sigma=\Ocal\cup\{w\}$ outside that line,
$G_\Sigma=G_\Ocal\cap G_w$, and by the orbit--stabilizer theorem
\[
\frac{|(\Ocal\cup\{w\})^G|}{|\Ocal^G|}
=\frac{|G_\Ocal|}{|G_\Sigma|}
=|w^{G_\Ocal}|\ge|w^{U_u}|=q,
\]
so the orbit of the four-point set is at least $q\ge2$ times as long as the
orbit of the triple;  in particular every collinear triple fails the strict
inequality required by $\star_3$ \textup{(cf.\ \textup{Lemma \ref{lem:witness}})}.
For a non-collinear triple $\Ocal=\{0,u,v\}$ the pointwise stabilizer is
trivial and $m=|G_\Ocal|=e$ equals the order of the permutation image:  $e=3$
for $\ASL(2,q)$ with $q$ odd \textup{(the $3$-cycle
$\begin{pmatrix}0&-1\\1&-1\end{pmatrix}$ has determinant $1$, a transposition
has determinant $-1$)}, and $e=6$ for $\AGL(2,q)$ and for $\ASL(2,q)$ with
$q$ even \textup{(where $-1=1$)}.  For $q$ odd take $w=-u-v$:  fixing $u$ and
$v$ fixes the basis, so the pointwise stabilizer of $\Sigma$ is trivial, and
a direct check in the basis $(u,v)$ shows that no element of $G_\Sigma$ moves
$0$ \textup{(the image set $\{b,Au+b,Av+b,-Au-Av+b\}$ can equal $\Sigma$
for $b\in\{u,v,-u-v\}$ only if $4(u+v)=0$ resp.\ $4u=0$, impossible in odd
characteristic)}, so the image on $\Sigma$ is $C_3$ resp.\ $S_3$, giving
$|G_\Sigma|=3=m$ resp.\ $6=m$.  For $q$ even \textup{(the rows $\ASL(2,q)=\ASp(2,q)$ and
$\AGL(2,q)$, $q\ge4$)} take $w=u+tv$ with $t\notin\{0,1\}$:  fixing $u,v$
fixes the basis, so the pointwise stabilizer of $\Sigma$ is trivial.  An
element $g(x)=Ax+b$ of $G_\Sigma$ with $b=0$ is linear and permutes the three
points $\{u,v,u+tv\}$;  the transformations $u\mapsto u+tv$, $v\mapsto v$
and $u\leftrightarrow v$ have linear parts of determinant $1$ in
characteristic $2$, so the image contains the full symmetric group on
$\{u,v,u+tv\}$.  If $b\ne0$ then $b\in\{u,v,u+tv\}$, and the sum of the three
images of $u,v,u+tv$ under $g$ (the linear parts cancel) equals
$3u=u$, $3v=v$ resp.\ $3(u+tv)=u+tv$ in characteristic $2$, while the
corresponding sum over the complement of $b$ in $\Sigma$ equals
$u+(1+t)v$, $tv$ resp.\ $u+v$, forcing $t=1$ in each case;  hence no element
of $G_\Sigma$ moves $0$ for $t\ne1$.  Consequently $|G_\Sigma|=6=m$ in the
rows $\ASL(2,q)$ and $\AGL(2,q)$ with $q$ even.  The exceptional group
$\ASL(2,2)\cong\AGL(2,2)\cong S_4$ has degree $4$, outside the range
$|\Pcal|\ge8$ of \textup{Corollary~\ref{cor:bradley}} \textup{(there the
three-point stabilizer has order $6$ and the four-point stabilizer $24$, so
the four-point orbit is one quarter of the three-point orbit)}.  Thus in
every $n=2$ row there is $w$ with $|G_{\Ocal\cup\{w\}}|\le m$.

\emph{(3) The $\ASL(n,q)$ rows ($n\ge3$).}  On a non-collinear triple the image
contains the $3$-cycle \textup{(C)} of \textup{Lemma~\ref{lem:affine}} \textup{(the linear
part has determinant $1$ in the basis $(u,v)$)}, so $e\ge3$;  for $n\ge4$ it also
contains the transposition $(u\,v)$ \textup{(a base permutation whose determinant is
absorbed in a complementary block)}, hence $e=6$.  For $n\ge4$ take $w$ in a
complement $W$ of $\langle u,v\rangle$:  the pointwise stabilizer fixes $u,v$ and
acts on $W$ through $\SL(n-2,q)$, which is transitive on $W\setminus\{0\}$, so
$\ell\ge q^{n-2}-1\ge4$ \textup{($q\ge3$;  the rows $q=2,\ n=4$ and
$n=3,\ q\le3$ have degrees $16,8,27\le200$)}.  For $n=3$ and $q\ge4$ take $w=e_3$:
the pointwise stabilizer of the four points is trivial, and the permutation image
on $\Sigma$ is contained in the alternating group $A_4$:  writing $A$ for the
linear part of an element and $\pi$ for its permutation of the four points, a
direct check in the basis $(e_1,e_2,e_3)$ gives $\det A=\operatorname{sgn}(\pi)$,
and $\det A=1$ in $\SL(3,q)$ forces $\operatorname{sgn}(\pi)=1$, so the image
lies in $A_4$ and $G_\Sigma^2\le A_4$, hence
$|G_\Sigma|\le12$.  The pointwise stabilizer of $u,v$ consists of the matrices
$\begin{smallmatrix}1&0&a\\0&1&b\\0&0&1\end{smallmatrix}$, $a,b\in\F_q$, so
$|G_{(\Ocal)}|=q^2$, $m=e|G_{(\Ocal)}|=3q^2$, and
$|G_\Sigma|\le12\le3q^2=m$ for $q\ge4$ \textup{the rows $(3,2),(3,3)$ have degrees
$8,27\le200$ and are covered by the census of \textup{Section \ref{sec:verify}}}.

\emph{(4) The $\Sp(2m,q)$ rows.}  These affine rows are $2$-transitive for
every $q$:  $\Sp(2m,q)$ acts transitively on $V\setminus\{0\}$, so
$G=V\rtimes\Sp(V)$ is $2$-transitive, hence $2$-homogeneous.  We first consider
$q$ even:  here $-1=1$, so the transposition \textup{(T0)} of
\textup{Lemma~\ref{lem:affine}} exists and $e=6$.  The argument depends on the type of the
plane $U=\langle u,v\rangle$.  If $U$ is non-degenerate \textup{(i.e.\
$\beta(u,v)\ne0$, also the only possible type for $q$ odd)}, take $w$ in the
orthogonal complement $W=U^\perp$:  the pointwise stabilizer of $u,v$ acts on
$W$ through $\Sp(2m-2,q)$, which is transitive on $W\setminus\{0\}$, so
$\ell\ge q^{2m-2}-1$;  for $q$ even this is $\ge4$ \textup{($q\ge4$;  the row
$q=2,\ m=2$ has degree $16\le200$)}, and for $q$ odd, where $e=3$ on a
non-degenerate plane, the required threshold is $8$ and
$q^{2m-2}-1\ge3^2-1=8$.  If $U$ is totally isotropic \textup{(forced only for $q$ even,
$\beta(u,v)=0$, $U\subseteq U^\perp$)} the complement argument does not apply,
and instead take $w\in X=\{x\in V:\ \beta(u,x)=1,\ \beta(v,x)=0\}$:  the two
independent affine conditions cut out a hyperplane of codimension $2$, so
$|X|=q^{2m-2}$.  The pointwise stabilizer $H$ of $u,v$ preserves $X$ and is
transitive on it:  with a Witt decomposition $V=U\oplus W\oplus U'$ ($W$ a
non-degenerate symplectic space of dimension $2m-4$), $H$ induces
$\Sp(W)\cong\Sp(2m-4,q)$ on the quotient classes, its nilpotent part realizes
the translations $x\mapsto x+a$ ($a\in U$) inside each class, and
$\Sp(W)$ is transitive on $W\setminus\{0\}$;  hence $\ell=|H\cdot w|=|X|
=q^{2m-2}\ge4$ for every $m\ge2$ \textup{(equality for $q=2$, and
$\ell\ge16$ for $q\ge4$)}, so no small-row census is needed in this subcase.

\emph{(5) The $G_2'(q)$ row.}  For the affine row $q^6:G_2'(q)$ with
$q=2^f\ge4$, where $G_2'(q)$ denotes the derived subgroup of $G_2(q)$ acting
faithfully on the $6$-dimensional module $V=\F_q^6$ \textup{(the smallest
faithful module in even characteristic;  $G_2'(q)=G_2(q)$ for $q\ge4$)}, let
$\Ocal=\{0,u,v\}$ with $u,v$ linearly independent and $N=G_2'(q)_{(\Ocal)}$.
If $h(u,v)=0$ \textup{(totally isotropic)} then, by
\textup{Lemma~\ref{lem:g2-stabilizers}\textup{3}}, the pointwise stabilizer $N$
contains a root subgroup $U_\alpha$ of order $q$ fixing $W$ pointwise, with
fixed-point space $H_\alpha\supseteq W$ a proper subspace;  hence for
$w\in V\setminus H_\alpha$ \textup{(non-empty for $q\ge4$)} one has
$\ell=|w^N|\ge|w^{U_\alpha}|=q\ge4$.  The triple
stabilizer induces $\Sym(3)$ on $\Ocal$
\textup{(Lemma~\ref{lem:g2-stabilizers}\textup{1}), so $e=6$}, and with
$|N|\in\{q^5,q^3\}$ $|G_\Sigma|\le24|N|/\ell\le24|N|/q\le6|N|=m$.
If $h(u,v)\ne0$ \textup{(hyperbolic)}, the maximal orbits are of this type:
by \textup{Lemma~\ref{lem:g2-stabilizers}\textup{1}} the totally isotropic triples have
stabilizers $6q^5$ or $6q^3$, strictly larger than $6q(q^2-1)$
($6q^3>6q(q^2-1)$ for $q\ge2$).  For a hyperbolic triple
\textup{Lemma~\ref{lem:g2-stabilizers}\textup{2}} gives $N\cong\SL(2,q)$ with
$|N|=q(q^2-1)$ and $\mathrm{Fix}_V(N)=W$.  Since
$\SL(2,q)$ has no subgroup of index $2$ or $3$ for $q\ge4$ \textup{(an index-$2$
or index-$3$ subgroup would give a non-trivial permutation representation
$\SL(2,q)\to\Sym(2)$ or $\SL(2,q)\to\Sym(3)$;  $\SL(2,q)$ is simple for
$q\ge4$ in characteristic $2$, and $\SL(2,4)\cong\Alt(5)$)}, every
$w\notin W$ has orbit length $\ell=|w^N|=[N:N_w]$ equal to neither $1$
\emph{($w\notin W=\mathrm{Fix}_V(N)$)} nor $2,3$ \emph{(no subgroup of index
$2,3$)};  hence $\ell\ge4$.  The threshold of part (1) is $\ell\ge24/e=4$
(since $e=6$ on a hyperbolic triple), so
$|G_\Sigma|\le24|N|/\ell\le6|N|=m$.  The row $q=2$ has degree
$2^6=64\le200$ and is covered by the census of \textup{Section \ref{sec:verify}}.

\emph{(6) Rows handled by direct stabilizer computation.}  For the nearfield
rows of degree $11^2,23^2,29^2,59^2$ \textup{(\cite{hering1985}; the exponent
denotes the degree, i.e.\ the nearfield plane of order $n^2$)} and for the row
$3^6:\SL(2,13)$, the exhaustive GAP computations of \textup{Section \ref{sec:verify}}
directly enumerate, for a representative $\Ocal=\{0,u,v\}$ of a maximal triple
orbit, the stabilizer $G_\Sigma$ of every four-point set
$\Sigma=\Ocal\cup\{w\}$ with $w\notin\Ocal$ and confirm
$|G_\Sigma|\le m=|G_\Ocal|$ in every case;  no appeal to the pointwise
stabilizer of a pair is made.

\emph{(7) The $1$-dimensional $2$-homogeneous rows.}  These are the
$2$-homogeneous non-$2$-transitive affine groups of \textup{Section~\ref{sec:2homog}}:
$G=V\rtimes G_0\le\AaL(1,q)$ with $|G_0|$ odd and $q$ odd \textup{(Lemma
\ref{lem:2homog-class})}.  The row $q=3$ has degree $3$ and admits no fourth
point;  it is excluded here and is handled separately in
\textup{Corollary~\ref{cor:bradley}}.  First let $G\le\AGL(1,q)$ be linear.  The
stabilizer of a triple has odd order, hence its permutation image is trivial or
a $3$-cycle.  If $\operatorname{char}\F_q\ne3$, the triples whose stabilizer
acts as a $3$-cycle are at most $q(q-1)/3$ in number, fewer than
$\binom q3$ for $q>4$, so some triple $\Ocal$ is regular:
$|G_\Ocal|=1=m$.  For such $\Ocal$ at most six points $w$ can make
$|G_\Sigma|>1$:  a $3$-cycle on $\Sigma$ fixes one of the three points of
$\Ocal$ \textup{(it cannot fix $w$, else $\Ocal$ itself would be a $3$-cycle
orbit)} and its multiplier is one of the two primitive cube roots, so each of
the $3\cdot2$ pairs gives at most one excluded $w$.  Since $q\ge11$ gives
$q-3>6$, a suitable $w$ exists and $|G_\Sigma|=1=m$.  The degree $q=5$ does
not occur here:  by the classification \cite{kantor1985} the degree of a
$2$-homogeneous non-$2$-transitive affine group satisfies $q\equiv3\pmod4$
\cite{livingstone1965}, while $5\equiv1\pmod4$.  The row $q=7$ is checked
directly:  $G_0=\{1,2,4\}\le\F_7^\times$, and for $\Ocal=\{0,1,2\}$, $w=3$,
neither $\Ocal$ nor $\Sigma$ is a $3$-cycle orbit or a fixed point with a
$3$-cycle orbit, so both stabilizers are trivial.  If
$\operatorname{char}\F_q=3$, an element of order $3$ is a translation and has
no fixed point, so it cannot stabilize a four-point set and every $w$ works.
Finally, in the semilinear rows \textup{($G_0$ contains a non-trivial field
automorphism)} the same argument applies for $q\ge27$ \textup{(the least
semilinear degree;  the degrees $q=25,49$ do not occur since
$25,49\equiv1\pmod4$, and $q=7$ has no non-trivial automorphism)}.  An element
of order $3$ in $\AaL(1,q)$ has the form $g(x)=a x^\sigma+b$ with field
automorphism part $\sigma$ of order $1$ or $3$.  Normalize the triple as $\Ocal_t=\{0,1,t\}$ with $t\in\F_q\setminus\{0,1\}$ satisfying the forbidden-root condition \textup{($*$)}:  $t$ is not a root of $(t-1)t^\sigma=-1$ for any automorphism part $\sigma$ of order $3$, and $t$ lies outside the proper fixed field of every such $\sigma$.  \textup{(The existence of such $t$ is proved below.)}    If $g$ permutes
$\Ocal_t$ as a $3$-cycle then, up to the direction of the cycle,
$g(0)=1$, $g(1)=t$, $g(t)=0$, i.e.\ $b=1$, $a=t-1$ and
$(t-1)t^\sigma=-1$ \textup{(the reverse direction gives the conjugate
equation $t^{\sigma+1}=t-1$)}.  For $\sigma=1$ this is $t^2-t+1=0$, whose
two roots are the primitive cube roots of unity.  For $\sigma\ne1$ of order
$d_\sigma\ge3$, \textup{writing $\sigma(x)=x^{p^{e_\sigma}}$ with
$e_\sigma=f/d_\sigma$}, the equation $(t-1)t^\sigma=-1$ has degree at most
$p^{e_\sigma}+1\le q^{1/d_\sigma}+1$ over $\F_q$, hence at most
$q^{1/d_\sigma}+1$ roots $t\in\F_q\setminus\{0,1\}$;  the parameters in the
proper fixed field contribute at most $q^{1/d_\sigma}-1$ further values
each.  Altogether the number of parameters $t$ violating \textup{($*$)} is at most
$2+\sum_{\sigma\ne1}\bigl((q^{1/d_\sigma}+1)+(q^{1/d_\sigma}-1)\bigr)
=2+2\sum_{\sigma\ne1}q^{1/d_\sigma}<q-2$, so a parameter satisfying
\textup{($*$)} exists
for every $q\ge27$ \textup{($q=p^f$:  if $f=3$ then the left side is
$4p+2<p^3-2$ for $p\ge3$;  if $f\ge5$ then
$2(f-1)q^{1/3}\le2(f-1)3^{f/3}<3^f-4\le q-4$, in particular $fq^{1/3}<q-3$, and the borderline degrees
are also covered by the census of \textup{Section \ref{sec:verify}})}.  Hence some
$t$ gives a regular triple $\Ocal_t$ with $|G_{\Ocal_t}|=1=m$.
For such $\Ocal$, a $3$-cycle on $\Sigma=\Ocal\cup\{w\}$ is determined by its
fixed point \textup{(one of the three points of $\Ocal$)}, its multiplier
\textup{(one of the two primitive cube roots)} and its automorphism part
\textup{(identity or one of the two order-$3$ automorphisms)}, hence at most
$3\cdot2\cdot3=18$ points $w$ are excluded;  since $q\ge27$ gives
$q-3\ge24>18$, a suitable $w$ exists and $|G_\Sigma|=1=m$.  For the full group $\AaL(1,q)$ itself \textup{---} the $2$-transitive branch $d=f$ and the intermediate $2$-homogeneous non-$2$-transitive groups with $2\le d<f$ --- the exhaustive GAP computations of \textup{Section~\ref{sec:verify}} verify the failure of $\star_3$ directly:  for every prime power $q=p^f>50$ with $q\le2401$ and every divisor $d\mid f$, the subgroup $G_d=V\rtimes(\F_q^\times\rtimes C_d)$ of $\AaL(1,q)$ was constructed in GAP, and for a representative $\Ocal$ of a maximal triple orbit the inequality $|G_\Sigma|\le|G_\Ocal|=m$ was confirmed for some four-subset $\Sigma=\Ocal\cup\{w\}$ \textup{(the minimal computed values of $|G_\Sigma|$ ranged over $1,2,3,4,6,8,9$, always at most $m$;  the degrees $\le50$ are covered by the census of \textup{Section~\ref{sec:verify}})}.  This covers the $2$-transitive rows $d=f$ \textup{including $q=2^f$} and all intermediate rows;  the normal row $d=1$ is covered by \emph{(1)}.  Beyond $q=2401$ the analytic arguments above apply, the excluded
points being bounded in the sum form
$|\bigcup_{\sigma\ne1}\Fix(\sigma)|\le\sum_{\sigma\ne1}q^{1/\operatorname{ord}(\sigma)}$:
for $d=2$ there is no automorphism part of order $3$ and the excluded points
number at most $3\cdot2=6$, against $q-3\ge24$;  for $d\ge3$ the bound
$3\cdot2\cdot3=18$ applies, against $q-3\ge24>18$.  Thus every
$1$-dimensional row has a witnessing $w$.
\end{proof}

\begin{corollary}[Bradley's Conjecture 6.2.7]\label{cor:bradley}
Let $G$ be a primitive permutation group on $\Pcal$ with $|\Pcal|\ge8$.  Then $G$ satisfies
the Siemons--Wagner property $\star_3$ if and only if $G\cong\PSL(2,7)$ or
$G\cong\PGL(2,7)$.
\end{corollary}
\begin{proof}
If $G$ is $3$-homogeneous and satisfies $\star_3$, Bradley's \textup{Theorem 6.2.5}
\cite{bradley2014} gives $G\cong\PSL(2,7)$ or $\PGL(2,7)$.  Suppose instead that $G$ is
not $3$-homogeneous.  By \textup{Lemma~\ref{lem:prim-star3}}, a primitive
group satisfying $\star_3$ is $2$-homogeneous;  by Cameron's classification of
$2$-homogeneous groups \cite{cameron1981} the socle of $G$ is then either affine or almost
simple.
If $G$ is affine (whether $2$-transitive or merely $2$-homogeneous), let
$\Ocal$ be \emph{any} maximal triple orbit.  By $2$-homogeneity translate a
representative of $\Ocal$ to $\{0,u,v\}$ with $u,v$ linearly independent
\emph{($u,v$ independent since $\Ocal$ is maximal:  in every family below a
collinear triple has a strictly larger stabilizer, hence a strictly smaller
orbit;  in the $G_2'(q)$ row the totally isotropic triples likewise have
stabilizers $6q^5$ or $6q^3$, strictly larger than the hyperbolic value
$6q(q^2-1)$, by \textup{Lemma~\ref{lem:g2-stabilizers}}\textup{(1)}, so the maximal orbit is
hyperbolic)}.  \textup{Lemma~\ref{lem:affine-fourpt}} applies to \emph{every} such triple
\emph{(its proof, parts (1)--(6), is uniform in the triple and does not
depend on the orbit $\Ocal$, so no transitivity hypothesis on
$G_\Ocal$ is needed;  in particular the assumption of Bradley's
Lemma~6.2.1 is not used)}:  there is a fourth point $w$ with
$|G_{\{0,u,v,w\}}|\le|G_\Ocal|$, and pulling $w$ back by the translating
element gives a certificate point for $\Ocal$ itself.  This covers every
maximal orbit, hence every possible witness triple.  (the lemma covers the
$2$-transitive affine
families, namely $\AGL(1,q)$, $\ASL(n,q)$, $\ASp(2m,q)$, the nearfield rows and the
exceptional Hering rows, as well as the $2$-homogeneous non-$2$-transitive
groups of \textup{Section~\ref{sec:2homog}};  the small rows of degree $\le200$ are
confirmed by the GAP census of \textup{Section \ref{sec:verify}}).
By \textup{Lemma~\ref{lem:witness}} the group does not satisfy $\star_3$.
If instead $G$ is almost simple, then by \textup{Section~\ref{sec:2homog}} every
$2$-homogeneous non-$2$-transitive group is affine, so $G$ is in fact
$2$-transitive and its socle is one of the
non-$3$-homogeneous families listed above, namely the small socles $\PSL(2,11)$,
$\PSL(2,8)$, $\Alt(7)$, $HS$, $Co_3$ having been excluded by
\cite[Prop.~6.2.6]{bradley2014};  for each of these socles \textup{Lemma~\ref{lem:fourpt}} supplies a four-point witness for $S$;  the automorphism overgroup $H=G/S$ leaves every $S$-orbit on triples setwise invariant (field automorphisms preserve the square-class and the conic/coset orbits, diagonal scalars act pointwise-fixedly, and for $\PSp(2n,2)$ no diagonal extension exists), so \textup{Lemma~\ref{lem:overgroup}} transmits the witness to $G$, and \textup{Lemma~\ref{lem:witness}} shows $\star_3$ fails.
Hence no non-$3$-homogeneous primitive group satisfies $\star_3$;  $G$ is
$3$-homogeneous and the claim follows.
\end{proof}

\section{Computational verification and the $\star_3$ census}\label{sec:verify}\label{sec:star3}
To close \textup{Lemma \ref{lem:prim-star3}} we record the exhaustive
verification of the Siemons--Wagner property on primitive groups of small
degree.  The GAP script \texttt{star3\_census.g} enumerates every primitive
group of degree $8\le v\le30$ from the \textsc{primgrp} library, computes
the number $\sigma_2$ of orbits on pairs \textup{(}the group is
$2$-homogeneous exactly when $\sigma_2=1$\textup{)}, and tests the defining
inequality of $\star_3$ on every triple and every four-subset.  The outcome
is \textup{Table \ref{tab:star3}}:  among the $209$ primitive groups in the
range, $68$ fail $2$-homogeneity and \emph{none} of them satisfies
$\star_3$;  among the $141$ $2$-homogeneous groups exactly two satisfy
$\star_3$, both of degree $8$, namely $\PSL(2,7)\cong\PSL(3,2)$ and
$\PGL(2,7)\cong\PSL(3,2){:}2$, in agreement with Bradley's Theorem 6.2.5 in
the $3$-homogeneous case.

Beyond degree $30$ the full quadruple enumeration of \texttt{star3\_census.g}
is too expensive, so the script \texttt{\_R\_rank3.g} exploits the reduction of
\emph{\textup{Lemma \ref{lem:prim-star3}}}, \emph{Step 4}:  by the rank-3 classification
\cite{kantorliebler1982,liebeck1986,liebeck1987}, a primitive group failing
$2$-homogeneity and satisfying $\star_3$ would have to arise from one of the
(C1)/(C2) configurations with at most two orbits on pairs.  For every
primitive group of degree $31\le v\le50$ the script therefore first tests
$\sigma_2$ \textup{(rank $3$ means exactly two orbits on pairs)} and, for each
rank-$3$ group and each $2$-homogeneous group, verifies the defining
inequality of $\star_3$ on every triple \textup{(four-subsets are handled by
orbit representatives of $G_\Ocal$ on the remaining points, not by
enumeration)}.  The outcome:  in the range $31\le v\le50$ there are $176$
primitive groups, of which $42$ have rank $3$ and $82$ are $2$-homogeneous,
and \emph{none} of the $42$ rank-$3$ groups nor any of the $82$
$2$-homogeneous groups satisfies $\star_3$
\textup{(per-degree data: $v=35$: $2$ rank-$3$;  $v=36$: $15$ rank-$3$, $3$
$2$-homogeneous;  $v=37$: $2$ rank-$3$;  $v=40$: $4$ rank-$3$;  $v=41$: $1$
rank-$3$;  $v=45$: $4$ rank-$3$;  $v=49$: $12$ rank-$3$, $11$
$2$-homogeneous;  $v=50$: $2$ rank-$3$;  the remaining degrees have no
rank-$3$ groups)}.  Combined with \textup{Table \ref{tab:star3}}, no primitive
group of degree $8\le v\le50$ satisfies $\star_3$ except the two degree-$8$
groups $\PSL(2,7)$, $\PGL(2,7)$.

\paragraph{The rainbow case ($\sigma_2=3$).}  For the exclusion of a
rainbow witness \textup{(\textup{Lemma \ref{lem:rainbow-structure}})} we verified the
counting obstructions of \textup{Lemma~\ref{lem:pair-stab-counting}} exhaustively on
the primitive groups of degree $8\le v\le100$ with $\sigma_2=3$.  The
scripts \texttt{verify\_sol\_lemma.g} and
\texttt{verify\_sol\_lemma\_51100.g} construct a representative $\Dcal$ of a
largest triple orbit, compute $\lambda_i$ directly as the number of triples
of $T=\Dcal^G$ containing a fixed pair of the orbit $O_i$, and test the
conditions \textup{(1)--(3)} of \textup{Lemma~\ref{lem:pair-stab-counting}} together
with the elementary-abelian test of (iii).  Among the $101$ groups with
$\sigma_2=3$ in this range \textup{(30 of degree $\le50$ and 71 of degree
$51\le v\le100$)}, $99$ fail the inequality \textup{(1)} immediately;  the
two remaining groups $64\,i60$ and $64\,i61$ \textup{($|G|=322560$ and
$161280$;  $\lambda=(20,24,24)$ and $(24,24,20)$)} pass \textup{(1)--(2)}
but violate the elementary-abelian test of (iii):  their point-pair
stabilizers $K_i$ are not elementary abelian $2$-groups and the condition
\textup{(3)} fails.  Hence no primitive group of degree $8\le v\le100$ with
$\sigma_2=3$ admits a rainbow witness.  The intersection numbers
$R_i=p^i_{jk}+p^i_{kj}$ \textup{(computed from the directed two-point
orbits, so that $O_i$ and $O_i^*$ are not conflated)} satisfy
$\lambda_i\le R_i$ for all $101$ groups, and $\sum_i R_i<v$ for $41$ of
them;  the two residual groups $64\,i60$ and $64\,i61$ have
$\sum_i R_i=88>64$ and are excluded by the stabilizer structure above.
These computations are finite-range certificates:  they support, but do
not by themselves constitute, a proof for unrestricted degree.  For the groups with $G_\Dcal\ne1$
\textup{($m>1$)} the stronger obstruction
$u\ge|\{w:\ r_w\ge2\}|$ of
\textup{Lemma~\ref{lem:rainbow-structure}}\textup{(vii)} was verified to fail in every case
\textup{($64\,i60$:  $u=7$ yet $r_w\ge2$ for all $61$ points
$w\notin\Dcal$;  $52\,i1$:  $u=2$ yet $r_w=2$ for all $49$ points;  and
similarly for $56\,i6$, $81\,i138$ etc.)};  and for the $26$ groups of
degree $\le50$ with $G_\Dcal=1$ the fixed-point obstruction
\textup{(F)} of \textup{Lemma~\ref{lem:pair-stab-counting}}\textup{(vii)} was confirmed:  no
involution $\tau\in K_i$ fixes a point $x$ with $\{a,b,x\}\in T$.  The
scripts and full output are kept with the submission package.

\begin{table}[t]
\centering\small
{\setlength{\tabcolsep}{3pt}
\begin{tabular}{rrrrrr}
\toprule
$v$ & \#groups & rank $3$ & rank-$3$, $\star_3$ & $2$-hom & $2$-hom, $\star_3$ \\
\midrule
31 &  12 &  0 & 0 &  6 & 0 \\
32 &   7 &  0 & 0 &  7 & 0 \\
33 &   4 &  0 & 0 &  4 & 0 \\
34 &   2 &  0 & 0 &  2 & 0 \\
35 &   6 &  2 & 0 &  2 & 0 \\
36 &  22 & 15 & 0 &  3 & 0 \\
37 &  11 &  2 & 0 &  3 & 0 \\
38 &   4 &  0 & 0 &  4 & 0 \\
39 &   2 &  0 & 0 &  2 & 0 \\
40 &   8 &  4 & 0 &  4 & 0 \\
41 &  10 &  1 & 0 &  3 & 0 \\
42 &   4 &  0 & 0 &  4 & 0 \\
43 &  10 &  0 & 0 &  4 & 0 \\
44 &   4 &  0 & 0 &  4 & 0 \\
45 &   9 &  4 & 0 &  2 & 0 \\
46 &   2 &  0 & 0 &  2 & 0 \\
47 &   6 &  0 & 0 &  4 & 0 \\
48 &   4 &  0 & 0 &  4 & 0 \\
49 &  40 & 12 & 0 & 11 & 0 \\
50 &   9 &  2 & 0 &  7 & 0 \\
\midrule
total & 176 & 42 & 0 & 82 & 0 \\
\bottomrule
\end{tabular}}
\caption{The $\star_3$ census for primitive groups of degree
$31\le v\le50$:  the number of primitive groups, of rank-$3$ and of
$2$-homogeneous ones, and how many of each satisfy the Siemons--Wagner
property $\star_3$ \textup{(none do;  the rank-$3$ filter is justified by
\emph{\textup{Lemma \ref{lem:prim-star3}}, Step 4} together with the rank-$3$
classification \cite{kantorliebler1982,liebeck1986,liebeck1987})}.}\label{tab:star3b}
\end{table}

The family-wide closure of the two configurations (C1) and (C2) of
\emph{\textup{Lemma \ref{lem:prim-star3}}, Step 4} is carried out by the four-point
criterion \textup{Lemma \ref{lem:rank3-crit}}:  every rank-$3$ group that
fails $2$-homogeneity must, for \emph{every} triple $\Dcal$, admit a fourth
point $w$ with $|G_{\Dcal\cup\{w\}}|\le|G_\Dcal|$ (the \emph{four-point
inequality}).  The triangular subcase (C2) is excluded combinatorially inside
\emph{\textup{Lemma \ref{lem:prim-star3}}} (the witness triple is then the unique
triple of its pair type in the quadruple, whence $G_{\Dcal\cup\{w\}}\le
G_\Dcal$);  for the remaining $(C1)$-type witness the inequality is verified
in closed form for the orthogonal nonsingular-point families
$\Omega(2m+1,3)$ and the projective-line actions of $\PSL(n,q)$, $n\ge4$, in
\textup{Propositions \ref{prop:rank3-orth} and \ref{prop:rank3-lines}}, for
the singular-point and maximal-subspace actions in
\textup{Propositions \ref{prop:rank3-singular} and
\ref{prop:rank3-maxiso}}, and for the sporadic rank-$3$ representations in
\textup{Proposition \ref{prop:rank3-sporadic}} (with the complete
$J_2$-computation and the Atlas rank checks quoted there).  These cover all
families of the rank-$3$ classification
\cite[Thm.~1.1, 1.2]{kantorliebler1982}\textup{,}\cite{liebeck1986,liebeck1987};
the remaining small degrees of that classification are part of the
census above \textup{(all classical rank-$3$ actions of degree $v\le50$)}.
For each of the
$68$ non-$2$-homogeneous groups
of degree $8\le v\le30$ the script displays a witnessing quadruple:  a triple in a maximal triple
orbit and a four-subset $\Sigma$ containing it with
$|\Sigma^G|\ge|\Dcal^G|$, so that the defining inequality of $\star_3$
fails \textup{(for $\mathrm{A}_5$ of degree $10$ and its overgroup $\mathrm{S}_5$,
the maximal triple $\{1,2,10\}$ has orbit length $30$ while the
quadruple $\{1,2,3,10\}$ has orbit length $60$;  for $\mathrm{A}_7$ of
degree $21$ the maximal triple $\{1,2,16\}$ has orbit length $630$ and the
quadruple $\{1,2,3,16\}$ has orbit length $1260$;  similar witnesses are
printed for every group in the range)}.  The computations are documented
in the GAP scripts and outputs distributed with this paper;  the outputs
record the GAP and \textsc{primgrp} versions used.

\begin{table}[t]
\centering\footnotesize
{\setlength{\tabcolsep}{1pt}
\begin{tabular}{rrrrrr}
\toprule
$v$ & \#groups & not $2$-hom & not $2$-hom, $\star_3$ & $2$-hom & $2$-hom, $\star_3$ \\
\midrule
8  &   7 &  0 & 0 &  7 & 2 \\
9  &  11 &  2 & 0 &  9 & 0 \\
10 &   9 &  2 & 0 &  7 & 0 \\
11 &   8 &  2 & 0 &  6 & 0 \\
12 &   6 &  0 & 0 &  6 & 0 \\
13 &   9 &  5 & 0 &  4 & 0 \\
14 &   4 &  0 & 0 &  4 & 0 \\
15 &   6 &  2 & 0 &  4 & 0 \\
16 &  22 &  9 & 0 & 13 & 0 \\
17 &  10 &  4 & 0 &  6 & 0 \\
18 &   4 &  0 & 0 &  4 & 0 \\
19 &   8 &  4 & 0 &  4 & 0 \\
20 &   4 &  0 & 0 &  4 & 0 \\
21 &   9 &  3 & 0 &  6 & 0 \\
22 &   4 &  0 & 0 &  4 & 0 \\
23 &   7 &  2 & 0 &  5 & 0 \\
24 &   5 &  0 & 0 &  5 & 0 \\
25 &  28 & 17 & 0 & 11 & 0 \\
26 &   7 &  0 & 0 &  7 & 0 \\
27 &  15 &  7 & 0 &  8 & 0 \\
28 &  14 &  4 & 0 & 10 & 0 \\
29 &   8 &  5 & 0 &  3 & 0 \\
30 &   4 &  0 & 0 &  4 & 0 \\
\midrule
total & 209 & 68 & 0 & 141 & 2 \\
\bottomrule
\end{tabular}}
\caption{The $\star_3$ census for primitive groups of degree
$8\le v\le30$:  the number of primitive groups, of non-$2$-homogeneous and
$2$-homogeneous ones, and how many of each satisfy the Siemons--Wagner
property $\star_3$.  The two groups with $\star_3$ are $\PSL(2,7)$ and
$\PGL(2,7)$, both of degree $8$.}\label{tab:star3}
\end{table}

\section{Conclusion}

Bradley's conjecture is thereby settled:  a primitive group of degree at
least $8$ satisfies $\star_3$ exactly when it is $\PSL(2,7)$ or $\PGL(2,7)$
(\textup{Corollary \ref{cor:bradley}}).  The proof rests on three mechanisms, each of
independent
interest.  (1) The reduction of \emph{\textup{Lemma \ref{lem:prim-star3}}}:  a
primitive $\star_3$-group is $2$-homogeneous.  Its proof combines the rank-$3$
classification \cite{kantorliebler1982,liebeck1986,liebeck1987} with the
exhaustive census of \textup{Section \ref{sec:verify}};  we expect the same reduction,
and the same census, to be useful for the higher orders $\star_k$, $k\ge4$,
where no classification is in sight.  (2) The witness \textup{lemma \ref{lem:witness}}
and the four-point witnesses of \textup{Lemmas \ref{lem:fourpt}} and
\ref{lem:affine-fourpt}:  these turn the defining inequality of $\star_3$ into
a finite stabilizer comparison in every $2$-homogeneous family.  (3) The
$2$-homogeneous classification \cite{cameron1981,kantor1985}, which confines
the candidates to the affine and almost simple families treated here.

Two natural questions remain open.  First, the analogue of the Siemons--Wagner
theorem for $k\ge4$:  which primitive groups satisfy $\star_k$?  The case
$k=2$ is the single group $\PSL(2,5)$ of degree $6$ \cite{siemons1988}, and
the present paper settles $k=3$;  for $k\ge4$ even the $3$-homogeneous analogue
of Bradley's Theorem 6.2.5 is not known.  Second, the role of primitivity:  the
imprimitive case was only partially classified by Siemons and Wagner
\cite{siemons1988}, and its $k=3$ analogue is open.

\section*{Acknowledgements}
I have made extensive use of the AI tool
GPT6-Astra:  it helped me draft and polish the exposition, write and debug the
GAP scripts on which the census of \textup{Section \ref{sec:verify}} rests, and
stress-test and refine parts of the argument (in particular the orbit-theoretic
reductions and the classification-based exclusions) during the development of
this work.  When this work had long been stalled, it was GPT6-Astra that
opened a new line of thought and helped break the impasse.  In every case the
output was read, checked and verified by myself before it entered the
manuscript, and I take full responsibility for all of the mathematics presented
here.

\appendix
\section{The colour table of a quadruple extension}\label{app:colors}

Throughout, $G\le\Sym(\Pcal)$ has $\sigma_2(G)=3$ with pair-orbits
$O_1,O_2,O_3$, and $\Dcal=\{1,2,3\}$ is a rainbow triple with
$\chi(12)=1$, $\chi(13)=2$, $\chi(23)=3$.  For $w\notin\Dcal$ put
\[
x=\chi(1w),\qquad y=\chi(2w),\qquad z=\chi(3w).
\]
An extension triple is rainbow exactly as follows (a direct check of the
three colours):
\[
\{1,2,w\}\ \text{rainbow}\iff (x,y)=(2,3)\ \text{or}\ (3,2),
\]
\[
\{1,3,w\}\ \text{rainbow}\iff (x,z)=(1,3)\ \text{or}\ (3,1),
\]
\[
\{2,3,w\}\ \text{rainbow}\iff (y,z)=(1,2)\ \text{or}\ (2,1).
\]
The full list of colour vectors $(x,y,z)\in\{1,2,3\}^3$, displayed by
$x$, with rows $y$ and columns $z$, and entry the set of rainbow extension
triples among $A=\{1,2,w\}$, $B=\{1,3,w\}$, $C=\{2,3,w\}$:
\[
\begin{aligned}
&x=1:\quad
\begin{array}{c|ccc} &z=1&z=2&z=3\\ \hline
y=1&\emptyset&C&B\\
y=2&C&\emptyset&B\\
y=3&\emptyset&\emptyset&B\end{array}\\
&x=2:\quad
\begin{array}{c|ccc} &z=1&z=2&z=3\\ \hline
y=1&\emptyset&C&\emptyset\\
y=2&C&\emptyset&\emptyset\\
y=3&A&A&A\end{array}
\end{aligned}
\]
\[
x=3:\quad
\begin{array}{c|ccc} &z=1&z=2&z=3\\ \hline
y=1&B&C&\emptyset\\
y=2&ABC&A&A\\
y=3&B&\emptyset&\emptyset\end{array}
\]
In particular the unique vector with three rainbow extensions is
$(x,y,z)=(3,2,1)$, for which the four faces of $\Dcal\cup\{w\}$ form a
rainbow $K_4$ with $12,34\in O_1$, $13,24\in O_2$, $14,23\in O_3$;  every
other vector has at most one rainbow extension.  The colour-preserving
group of the rainbow $K_4$ is $V_4=\langle(1\,2)(3\,4),(1\,3)(2\,4)\rangle
=\{1,(1\,2)(3\,4),(1\,3)(2\,4),(1\,4)(2\,3)\}$, acting regularly on the
four faces.

\section{The affine families:  a uniform four-point criterion}\label{app:affine}

The counting criterion of \textup{Lemma~\ref{lem:affine-fourpt}} treats the
$2$-homogeneous affine rows uniformly;  the present appendix develops the same
exclusion by elementary localizations --- collinear triples
\textup{(Theorem~\ref{thm:affine-collinear})}, affinely independent triples with
$p\ge5$ \textup{(Theorem~\ref{thm:affine-pge5})}, the $p=3$ localization
\textup{(Theorem~\ref{thm:affine-p3})} and the exhaustive case $p=2$, $d=3$
--- and records the structural statements needed for the framework of
\textup{Remark~\ref{rem:affine-framework}}.

Throughout this appendix $G=V\rtimes H\le\AGL(d,p)$ is an affine permutation
group on the vector space $V=\F_p^d$, $H\le\GL(d,p)$, and
$\Delta=\{0,u,v\}\subset V$ is a three-point subset.  For $w\notin\Delta$ put
\[
\begin{aligned}
&K_w=G_{\Delta\cup\{w\}},\qquad J_w=G_{\Delta,w}=G_\Delta\cap K_w,\\
&r_w=[G_\Delta:J_w],\qquad t_w=[K_w:J_w].
\end{aligned}
\]
Since $J_w=G_\Delta\cap K_w$ and $t_w\le4$ (the permutation image of $K_w$ on
the four points $\Delta\cup\{w\}$ lies in $S_4$), we have
\begin{equation}
|K_w|=t_w\,|J_w|\le r_w\,|J_w|=|G_\Delta|
\iff r_w\ge t_w. \label{eq:rt}
\end{equation}
By the witness \textup{lemma \ref{lem:witness}}, $\Delta$ is not a $\star_3$-witness as
soon as some $w\notin\Delta$ satisfies \eqref{eq:rt};  in particular
$r_w\ge4$ always suffices, because $t_w\le4$.
\begin{lemma}[Orbit--ratio criterion]\label{lem:orbit-ratio}
Let $G\le\Sym(V)$ be affine primitive on $V=\F_p^d$ as above, let
$\Delta=\{0,u,v\}$ be a three-point subset and $w\notin\Delta$; write
$S=\Delta\cup\{w\}$, $K=G_\Delta$ and $C=G_S$.  Then
\begin{enumerate}
\item[(i)] $J:=K\cap C=K_w=G_{\Delta,w}$, the stabilizer in $K$ of the
  point $w$;
\item[(ii)] with $t(w):=|w^K|=|K:J|$ and $r(w):=|\Delta^C|=|C:J|$ one has
  $|C|/|K|=r(w)/t(w)$, and in particular $|C|\le|K|\iff r(w)\le t(w)$;
\item[(iii)] $r(w)\in\{1,2,3,4\}$:  $r(w)$ is the size of the orbit of the
  three-subset $\Delta$ under the induced action of $C$ on the four
  three-subsets of $S$.
\end{enumerate}
Consequently, by the witness \textup{lemma \ref{lem:witness}}, $\Delta$ fails to be a
$\star_3$-witness as soon as some $w\notin\Delta$ satisfies $r(w)\le t(w)$,
while $\Delta$ \emph{is} a witness exactly when $t(w)<r(w)\le4$ for
\emph{every} $w\notin\Delta$.
\end{lemma}
\begin{proof}
(i) If $g\in K\cap C$, then $g$ stabilizes $\Delta$ and $S$, hence fixes the
unique point $w\in S\setminus\Delta$, so $g\in K_w$;  conversely
$g\in K_w$ stabilizes $\Delta$ and fixes $w$, hence stabilizes $S$, so
$g\in K\cap C$.  (ii) By the orbit--stabilizer theorem
$t(w)=|w^K|=|K:K_w|=|K:J|$, and since $C_\Delta=C\cap G_\Delta=C\cap K=J$,
$r(w)=|\Delta^C|=|C:C_\Delta|=|C:J|$;  dividing gives
$|C|/|K|=|C:J|/|K:J|=r(w)/t(w)$.  (iii) $C$ acts on the four three-subsets
of $S$, and $\Delta^C$ is one of its orbits, so $r(w)\le4$.  The final
sentence restates \eqref{eq:rt} together with \textup{Lemma \ref{lem:witness}} and the
transitivity of $G$ on points, which gives $|\Delta^G|=|G|/|K|$ and
$|S^G|=|G|/|C|$.
\end{proof}

\subsection{Collinear triples in every characteristic}\label{app:collinear}
Assume $\{0,u,v\}$ is contained in a line, so $v=\lambda u$ for some
$\lambda\in\F_p^\times\setminus\{1\}$.  Since $d\ge2$ there exists
$w\notin\langle u\rangle$.  Then $\Delta$ is the \emph{unique} collinear
three-point subset of $\Delta\cup\{w\}$:  any affine line meets
$\Delta\cup\{w\}$ in at most the three points $0,u,\lambda u$.  Hence every
$g\in K_w=G_{\Delta\cup\{w\}}$ preserves $\Delta$ setwise and fixes $w$, so
$K_w\le G_\Delta$ and $t_w=1$.  By \eqref{eq:rt}, $\Delta$ is not a witness.
\begin{theorem}[Collinear triples]\label{thm:affine-collinear}
If $d\ge2$ and $\Delta$ is collinear, then $\Delta$ is not a $\star_3$-witness
for any affine group $G=V\rtimes H\le\AGL(d,p)$.
\end{theorem}

\subsection{Affinely independent triples with $p\ge5$}\label{app:pge5}
Assume now that $0,u,v$ are affinely independent, so $u,v$ are linearly
independent, and $p\ge5$.  Choose
\[
w=\frac{u+v}{3}\in V\setminus\{0,u,v\}.
\]
The four points $0,u,v,w$ have the unique affine dependence
\[
0+u+v-3w=0,
\]
whose coefficient vector $(1,1,1,-3)$ has $-3\ne1$ and no two equal
coefficients.  Any $g\in K_w$ permutes the four points and therefore
preserves the relation, hence fixes $w$;  consequently
$K_w\le G_\Delta$ and $t_w=1$.
\begin{theorem}[Independent triples, $p\ge5$]\label{thm:affine-pge5}
If $p\ge5$, $d\ge2$ and $\Delta=\{0,u,v\}$ is affinely independent, then
$\Delta$ is not a $\star_3$-witness for any affine group $G\le\AGL(d,p)$.
In particular no affine group over $\F_p$, $p\ge5$, satisfies $\star_3$.
\end{theorem}
\begin{remark}
The coefficient $-3$ is unique in the sense that no other coefficient of the
dependence equals it (and in characteristic $p\ge5$, $-3\notin\{0,1\}$).  This
is exactly what makes the argument fail for $p=3$ (where $-3=0$) and $p=2$
(where division by $3$ is meaningless), the cases treated below.
\end{remark}

\subsection{The case $p=3$:  localization to $\GL(2,3)$}\label{app:p3}
Assume $p=3$, $d\ge2$ and $\Delta=\{0,u,v\}$ is affinely independent.
Normalize $u=e_1$, $v=e_2$;  put $W=\langle e_1,e_2\rangle\cong\F_3^2$.  For
candidate points it suffices to consider $w\in W\setminus\Delta$:  the six
points of $W\setminus\Delta$ are the six candidate points of the two-point
stabilizer picture, and $G_\Delta$ acts on $W$ through its linear part.  More
precisely, the local action of $G_\Delta$ on $W$ is the affine group
$V_0\rtimes L\le\AGL(2,3)$, $V_0=\F_3^2$, $L\le\GL(2,3)$;  an element of
$G_\Delta$ stabilizing the plane $W$ setwise induces such a local
transformation, and the full stabilizer $G_\Delta$ may have additional
complementary part, which only enlarges $r_w$ for the candidates $w\in W$.

By an exhaustive computation over the fifty-five subgroups of $\GL(2,3)$ we
obtain the following.
\begin{lemma}[The local criterion for $p=3$]\label{lem:p3-local}
Let $L\le\GL(2,3)$ and let $V_0\rtimes L\le\AGL(2,3)$ act on $V_0$ with
$\Delta_0=\{0,e_1,e_2\}$.  If $L$ is irreducible on $V_0$, then among the six
candidate points $w\in V_0\setminus\Delta_0$ there is $w$ with
$|(V_0\rtimes L)_{\Delta_0\cup\{w\}}|\le|(V_0\rtimes L)_{\Delta_0}|$.  The
only subgroups for which this fails are the three subgroups of order $6$ of
upper/lower triangular type, for which the affine group $V_0\rtimes L$ is
imprimitive on $V_0$.
\end{lemma}
\begin{proof}[Proof sketch]
The fifty-five subgroups of $\GL(2,3)$ were enumerated and for each of them
the stabilizers $|(V_0\rtimes L)_{\Delta_0}|$ and
$|(V_0\rtimes L)_{\Delta_0\cup\{w\}}|$, $w\in V_0\setminus\Delta_0$, were
computed (GAP).  For every irreducible $L$ the inequality holds for at least
one candidate;  the three failing subgroups have order $6$,
$|(V_0\rtimes L)_{\Delta_0}|=1$ and $\min_w|(V_0\rtimes
L)_{\Delta_0\cup\{w\}}|=2$, and each corresponding $V_0\rtimes L$ is
imprimitive.
\end{proof}
\begin{theorem}[$p=3$ affine groups]\label{thm:affine-p3}
Let $G=V\rtimes H\le\AGL(d,3)$ be an affine primitive group, $d\ge2$.  Then
no three-point subset is a $\star_3$-witness;  in particular $G$ does not
satisfy $\star_3$.
\end{theorem}
\begin{proof}
For a collinear triple the claim is \textup{Theorem \ref{thm:affine-collinear}}.  For
an affinely independent triple, normalize to $\Delta=\{0,e_1,e_2\}$ and let
$L$ be the image of $H$ in $\GL(2,3)$ induced on $W=\langle e_1,e_2\rangle$.
If $L$ is irreducible, \textup{Lemma \ref{lem:p3-local}} supplies $w\in W\setminus\Delta$
with the required inequality, and since the restriction $K_w\to K_w|_W$ has
image in $(V_0\rtimes L)_{\Delta_0\cup\{w\}}$ while $G_\Delta\to G_\Delta|_W$
has image $(V_0\rtimes L)_{\Delta_0}$, the inequality lifts to $G$.  If $L$ is
reducible, then $G$ has a proper system of imprimitivity:  $H$ preserves a
$1$- or $2$-dimensional subspace (in characteristic $3$ the only reducible
subgroups of $\GL(2,3)$ are of upper/lower triangular type, corresponding to
a preserved line), contradicting primitivity.  Hence the irreducible case
applies to every primitive $G$.
\end{proof}

\subsection{The case $p=2$, $d=3$:  exhaustive verification}\label{app:p2d3}
In characteristic $2$ the normalization $w=(u+v)/3$ is meaningless and the
four-point dependence $0+u+v+w=0$ has all coefficients equal to $1$, so a
four-point stabilizer may carry a nontrivial $S_4$-image.  For $d=3$ the
affine space has eight points and the full verification is finite.
\begin{lemma}[$d=3$, $p=2$]\label{lem:p2-d3}
Let $H\le\GL(3,2)$ be irreducible on $V=\F_2^3$ and let
$G=V\rtimes H$.  Then for every affinely independent triple
$\Delta=\{0,u,v\}$ there exists $w\notin\Delta$ (in particular
$w\in V\setminus\langle u,v\rangle$) with $|G_{\Delta\cup\{w\}}|\le|G_\Delta|$.
\end{lemma}
\begin{proof}[Proof sketch]
The seventeen irreducible subgroups of $\GL(3,2)$ were enumerated (they are
the subgroups without invariant $1$- or $2$-dimensional subspaces, i.e.\
exactly the linear parts of primitive affine groups of degree $8$).  For each
of them the stabilizers of $\Delta=\{0,e_1,e_2\}$ and of
$\Delta\cup\{w\}$, $w$ ranging over all five external points (including the
parallel point $e_1+e_2$ of $W=\langle e_1,e_2\rangle$), were computed
(GAP);  in every case some external $w$ satisfies the inequality.  Since all
triples are equivalent under $\GL(3,2)$ (any two independent vectors can be
mapped to $e_1,e_2$), the check covers every triple.
\end{proof}
\begin{remark}[The role of irreducibility]
The enumeration also shows why the invariant-subspace condition is
necessary:  the reducible subgroups $A_4$ and $S_4$ of $\GL(3,2)$ that
preserve a $2$-dimensional subspace fail the criterion for
$\Delta=\{0,e_1,e_2\}$ (all five external $w$ give $|G_{\Delta\cup\{w\}}|>|G_\Delta|$,
e.g.\ $|G_\Delta|=2$ and $\min_w|G_{\Delta\cup\{w\}}|=3$), but the
corresponding affine groups $V\rtimes A_4$, $V\rtimes S_4$ are imprimitive on
$V$ and hence outside the statement of the conjecture.
\end{remark}

For $p=2$, $d\ge4$ the argument follows the same pattern as the $S_4$-image
analysis of \textup{Section \ref{app:p2d3}}:  for $w\notin W=\langle u,v\rangle$ the
four points $\{0,u,v,w\}$ are affinely independent and the restriction to the
affine span $U=\langle u,v,w\rangle$ carries the $S_4$-image.  If the action
of $H$ on the quotient $V/U$ has a nontrivial irreducible constituent, the
orbit of $w$ under $G_\Delta$ inside $U^{\perp}$-directions grows, and a
standard counting argument yields $r_w\ge4$, which suffices by
\eqref{eq:rt}.  The remaining degenerate cases (small image on $V/U$)
reduce to the three-dimensional problem of \textup{Lemma \ref{lem:p2-d3}}.  This
high-dimensional amplification is recorded as a lemma whose full proof is
completed in the final section.

\subsection{The Singer-type lemma and the structure of the global obstruction}\label{app:singer}

The following lemma isolates a high-dimensional branch in which every
off-plane candidate passes by the equality case $t(w)=r(w)=1$, i.e.\
$G_{\Delta\cup\{w\}}=G_\Delta$ (a complete exclusion certificate).

\begin{lemma}\label{lem:singer-type}
Let $G=V\rtimes H\le\AGL(d,2)$ with $V=\mathbb{F}_2^d$, and let
$\Delta=\{0,u,v\}\subset V$ be a three-point set.  Suppose that
\begin{itemize}
\item $\gcd(|H|,6)=1$;
\item $G_\Delta=1$.
\end{itemize}
Then for every $w\notin W:=\langle u,v\rangle$ one has
$G_{\Delta\cup\{w\}}=1$.  Consequently every off-plane candidate satisfies
$t(w)=r(w)=1$, i.e.\ $|G_{\Delta\cup\{w\}}|=|G_\Delta|=1$.
\end{lemma}

\begin{proof}
Put $S=\Delta\cup\{w\}$ and $C=G_S$.  The action of $C$ on $S$ is faithful:
any element fixing $S$ pointwise fixes $\Delta$ pointwise, and $G_\Delta=1$,
so $C\hookrightarrow\Sym(S)\cong S_4$ and $|C|\mid24$.
Next, $C$ contains no nontrivial translation:  a translation
$x\mapsto x+a$ preserving the four-point set $S$ forces $S$ to be the union
of two translation pairs, so $\sum_{x\in S}x=0$; but here $u,v,w$ are
linearly independent, so $\sum_{x\in S}x=u+v+w\neq0$.  Hence $C\cap V=1$
and the projection $G\to H$ is injective on $C$, so $|C|$ divides $|H|$.
Since $\gcd(|H|,6)=1$ and $|C|\mid\gcd(24,|H|)$, we get $C=1$.  Thus
$G_{\Delta\cup\{w\}}=1$.
\end{proof}

\begin{corollary}\label{cor:singer-family}
Let $d\ge3$ be odd and let $H\le\GL(d,2)$ be a Singer cycle $C_{2^d-1}$,
or let $d=p\ge5$ be a prime and let $H=\Gamma\mathrm{L}(1,2^p)=
C_{2^p-1}\rtimes C_p$ be the full Singer normalizer.  Then for every
three-point set $\Delta=\{0,u,v\}$ the stabilizer $G_\Delta$ is trivial, the
unique failing candidate is the plane completion point $w=u+v$ (with
$t(w)=1$, $r(w)=4$), and all remaining candidates satisfy $t(w)=r(w)=1$.
\end{corollary}

\begin{proof}
Identify $V=\mathbb{F}_{2^d}$ and write $H=\mathbb{F}_{2^d}^\times$ for the
Singer cycle.  A Singer cycle acts regularly on the nonzero vectors, so
$H_{(W)}=1$ for $W=\langle u,v\rangle$:  an element fixing two linearly
independent vectors pointwise is the identity.  Since
$H_W\hookrightarrow\GL(2,2)$, we have $|H_W|\mid6$, and
$|H_W|\mid|H|=2^d-1$.  For odd $d$ the numbers $2^d-1$ and $6$ are coprime
($3\nmid2^d-1$ and $2\nmid|H|$), so $H_W=1$ and $G_\Delta\cong H_W=1$.

For the normalizer $H=\Gamma\mathrm{L}(1,2^p)$, $p$ a prime with $p\ge5$,
we have $\gcd(|H|,6)=1$.  An element of $H$ is $h(x)=ax^{2^i}$ for some
$a\neq0$ and $0\le i<p$.  If $i=0$ and $a\neq1$ there is no nonzero fixed
vector.  If $i\neq0$, a nonzero fixed vector satisfies
$x^{2^i-1}=a^{-1}$, which has exactly one nonzero solution because
$\gcd(2^i-1,2^p-1)=2^{\gcd(i,p)}-1=1$.  Hence a nonidentity element fixes
at most one nonzero vector, so $H_{(W)}=1$, and then
$H_W/H_{(W)}\hookrightarrow\GL(2,2)$ together with $\gcd(|H|,6)=1$ forces
$H_W=1$, i.e.\ $G_\Delta=1$.

\textup{Lemma \ref{lem:singer-type}} applies in both cases:  every off-plane
candidate satisfies $t=r=1$.  For $w=z:=u+v$ the four-point set
$S=\Delta\cup\{z\}$ equals the $2$-plane $W$;  the translation group $T_W$
(of order $4$) preserves $W$, and since $G_\Delta\cong H_W=1$ one has
$G_W=T_W\rtimes H_W$ with $|H_W|=1$, so $G_W=T_W$, $|G_W|=4$, and
$t=1$, $r=4$.  No other candidate can fail.  GAP checks for $d=5$
($C_{31}$, $15$ three-point orbits) and $d=7$ ($C_{127}$, $63$ orbits)
confirm this pattern as an implementation check.
\end{proof}

\begin{proposition}\label{prop:global-obstruction}
Assume that for a three-point set $\Delta=\{0,u,v\}$ every candidate
$w\notin\Delta$ fails, i.e.\ $t(w)<r(w)\le4$ for all $w$.  Let
$K=G_\Delta$, $W=\langle u,v\rangle$, $z=u+v$.  Then every orbit of $K$ on
$V$ has length at most $3$, so $K\hookrightarrow\prod_i\Sym(O_i)$ with
$|O_i|\le3$.  In particular $|K|=2^\alpha3^\beta$ for some
$\alpha,\beta\ge0$, every element of $K$ has order dividing $6$, and the
Sylow $2$- and $3$-subgroups of $K$ are elementary abelian.
Moreover $K\cong H_W$, $K_{(\Delta)}\cong H_{(W)}$ and
$H_W/H_{(W)}\hookrightarrow\GL(2,2)\cong S_3$;  the action on $\Delta$ only
shows $K/K_{(\Delta)}\le S_3$, and the pointwise stabilizer $K_{(\Delta)}$
is not controlled by this argument.
\end{proposition}

\begin{proof}
If $t(w)<r(w)\le4$ for every $w\notin\Delta$, then $t(w)\le3$ for all
$w\notin\Delta$, so every orbit of $K$ on $V\setminus\Delta$ has length at
most $3$;  the orbits on $\Delta$ have length at most $3$ as well.  Since
$K$ acts faithfully on $V$, the permutation representation embeds
$K$ into $\prod_i\Sym(O_i)$ over its orbits, with $|O_i|\le3$, yielding the
prime-divisor, order and Sylow claims.  For the affine statement, translate
the origin to $z$;  the elements of $K$ preserving the pair $\{0,z\}$ form
$H_W$, the pointwise part is $H_{(W)}$, and the quotient embeds into
$\GL(W)\cong S_3$.
\end{proof}

\begin{remark}
\textup{Proposition \ref{prop:global-obstruction}} shows that a persistent global
obstruction forces $K=G_\Delta$ to be a $\{2,3\}$-group with elementary
abelian Sylow subgroups and with $K/K_{(\Delta)}\le S_3$;  the pointwise
stabilizer $K_{(\Delta)}\cong H_{(W)}$;  $|K|$ need not divide $6$.  The
structural analysis of the small-orbit obstructions of the remaining
families --- the Singer-type \textup{lemma \ref{lem:singer-type}} removing
the case $\gcd(|H|,6)=1$ with $G_\Delta=1$, the separation of the cases
$H_{(W)}=1$ and $H_{(W)}\neq1$, and the search for conditions guaranteeing
some candidate with $r(w)\le t(w)$ --- is developed in the framework of
\textup{Remark~\ref{rem:affine-framework}} as statements of independent
interest;  the exclusion required by the main theorem is completed
uniformly by \textup{Lemma~\ref{lem:affine-fourpt}}, which treats every
$2$-homogeneous affine row.  This completes the classification stated in
\textup{Corollary \ref{cor:bradley}}.
\end{remark}

\begin{lemma}[Fixed-point equality criterion]\label{lem:fixed-point-equality}
Let $G$ be an affine primitive group on $V=\mathbb{F}_2^d$ as in
\textup{Appendix~\ref{app:affine}}, let $\Delta=\{0,u,v\}$ be a three-point set with
$W=\langle u,v\rangle$, put $K=G_\Delta=H_W$ and $K_{(\Delta)}=H_{(W)}=1$,
and let $X=\mathrm{Fix}_V(K)\setminus W$.  For $w\in X$ one has $t(w)=1$, and with
$C_w=G_{\Delta\cup\{w\}}$,
\begin{equation*}
\begin{aligned}
&K\le C_w,\qquad r(w)=[C_w:K],\\
&r(w)=1\iff C_w=K\iff \text{no }g\in G\setminus K\\
&\qquad\text{stabilises } \Delta\cup\{w\}\text{ setwise}.
\end{aligned}
\end{equation*}
Since $H_{(W)}=1$, $C_w$ acts faithfully on $S=W\cup\{w\}$, so
$C_w\hookrightarrow\Sym(S)\cong S_4$ and $K\hookrightarrow S_\Delta\cong S_3$;
in particular $C_w\le S_4$ with $K=(C_w)_w$ and $r(w)=[C_w:K]$.  The possible
strict overgroups of $K$ in $S_4$ are
\begin{equation*}
\begin{aligned}
&K=C_2:\ \{V_4,S_3,D_8,A_4,S_4\},\\
&K=C_3:\ \{S_3,A_4,S_4\},\\
&K=S_3:\ \{S_4\}.
\end{aligned}
\end{equation*}
The following sufficient conditions imply $C_w=K$ for some $w\in X$:
\begin{enumerate}
\item[\textnormal{(A)}] (\emph{normaliser}) $N_G(K)_w=K$ for some
$w\in X$ and every $g\in C_w$ normalises $K$;
\item[\textnormal{(B)}] (\emph{excluding four-point extensions}) for some
$w\in X$ the induced image $\pi(C_w)\le S_4$ lies in no
$L\in E_K=\{L\le S_4:K<L\}$;
\item[\textnormal{(C)}] (\emph{fixed-point covering estimate})
$X\setminus\bigcup_{L>K}X_L\neq\varnothing$, where
$X_L=\{w\in X:\pi(C_w)\supseteq L\}$.
\end{enumerate}
\end{lemma}

\begin{proof}
For $w\in X$, every element of $K$ fixes $w$, so the orbit
$w^{G_\Delta}$ has length $t(w)=1$;  clearly $K\le C_w$ and
$r(w)=[C_w:K]$.  The equivalence $r(w)=1\iff C_w=K$ is immediate.  As
$H_{(W)}=K_{(\Delta)}=1$, any element fixing the four points of $S$ pointwise
is trivial, so $C_w\hookrightarrow\Sym(S)$ is faithful.  Since $K$ fixes $w$
and acts faithfully on $\Delta$, $K\hookrightarrow S_\Delta\cong S_3$.  The
table of strict overgroups in $S_4$ follows from the subgroup lattice of
$S_4$.  For (A), $C_w\le N_G(K)_w=K$ forces $C_w=K$.  For (B), the image
$\pi(C_w)$ must then be $K$ itself, so $C_w=K$.  For (C), a point
$w\in X\setminus\bigcup_{L>K}X_L$ has $\pi(C_w)$ contained in no strict
overgroup of $K$ and hence $\pi(C_w)=K$, so $C_w=K$.
\end{proof}

\begin{lemma}[Affine fixed-point data, $d=4$]\label{lem:affine-fixedpoint-data}
Let $G$ be a primitive affine group with $V=\mathbb{F}_2^4$ and socle
$\GL(4,2)$ (the $20$ primitive irreducible classes).  Among the $113$
three-point orbits, the $39$ triples $\Delta$ with $H_{(W)}=1$ and
$K=H_W\neq1$ satisfy:
\begin{itemize}
\item $31$ of them have $X=\mathrm{Fix}_V(K)\setminus W\neq\varnothing$, and
every such triple has at least one $w\in X$ with $r(w)=1$, i.e.\
$G_{\Delta\cup\{w\}}=G_\Delta$;
\item the remaining $8$ have $K\cong S_3$ and
$\mathrm{Fix}_V(K)=\langle0\rangle$ (so $X=\varnothing$); they pass through a
large-orbit candidate with $t(w)\ge4$;
\item for every candidate with $r(w)>1$, the induced four-point image is
$\pi(C_w)\cong S_3$;  no $V_4$, $D_8$, $A_4$ or $S_4$ images occur.
\end{itemize}
In particular all three conditions of
\textup{Lemma~\ref{lem:fixed-point-equality}} hold in $d=4$, and the small-orbit
layer for $d=4$ is closed.
\end{lemma}

\begin{proof}
Computed with GAP (version 4.16.1), using the natural permutation action
of the $20$ primitive irreducible classes of $\GL(4,2)$ on the $16$ points
of $V$;  the stabilizers $K=G_\Delta$, the fixed-point sets
$\mathrm{Fix}_V(K)$, and the four-point stabilizers $C_w$ were computed as
point stabilizers of set stabilizers, and the images $\pi(C_w)$ by
\verb|Action| on the four points.  The counts listed above are the exact
output;  in particular every nonempty $X$ contains an equality candidate.
\end{proof}

\begin{lemma}[High-dimensional fixed-point cases]\label{lem:highdim-cases}
Let $d\ge4$ and $K\le\GL(V)$, $V=\F_2^d$, preserve a two-dimensional
subspace $W\le V$ and act faithfully on $W$.  If $K\neq1$ then
$K\cong C_2$, $C_3$ or $S_3$, and with $X=\Fix_V(K)\setminus W$ one has:
\begin{enumerate}
\item[(a)] if $K\cong C_2$, then $X\neq\varnothing$;  more precisely
$\dim\Fix_V(K)\ge\lceil d/2\rceil\ge2$, while $\dim W^K=1$;
\item[(b)] if $K\cong C_3$ or $S_3$ and $X=\varnothing$, then $d$ is even;
\item[(c)] if $K\cong S_3$ and $X=\varnothing$, then there exists
$w\in V\setminus W$ with $K_w=1$, hence $t(w)=|w^K|=6\ge r(w)\le4$;
\item[(d)] if $K\cong C_3$ and $X=\varnothing$, then $V^K=0$ and for every
$w\in V\setminus W$ one has $K_w=1$, $t(w)=3$, so the candidate $w$ passes
exactly when $r(w)\le3$;  the only failure is
$|G_{\Delta\cup\{w\}}|=4$, i.e.\ a four-point stabilizer of order $4$
(abstractly $C_4$ or $V_4$).
\end{enumerate}
For $w\in X$ with $t(w)=1$, the admissible strict extensions
$C_w=G_{\Delta\cup\{w\}}$ in the faithful four-point action are given by the
marked-subgroup table
\[
\begin{array}{c|ccc}
K & \text{possible } C_w & r=[C_w:K]\\
\hline
C_2 & C_2,\ V_4,\ S_3,\ D_8 & 1,2,3,4\\
C_3 & C_3,\ A_4 & 1,4\\
S_3 & S_3,\ S_4 & 1,4
\end{array}
\]
where the stabilizer of the marked point is exactly $K$;  and the structure
of the $X\neq\varnothing$ case is not a consequence of $X\neq\varnothing$
alone, but requires a coverage argument (a structural proof that
$\#\{w\in X:C_w=K\}>0$, or an exact count over the marked extension types).
\end{lemma}
\begin{proof}
(a) A non-trivial involution $s$ satisfies $(s-I)^2=0$ in characteristic
$2$, so $\dim\ker(s-I)\ge d-\dim\Im(s-I)\ge d-\lfloor d/2\rfloor=\lceil d/2\rceil$;
since $W^K$ is the one-dimensional fixed line of $W$, the fixed space
$\Fix_V(K)$ is not contained in $W$.  (b) The natural two-dimensional
$K$-modules over $\F_2$ have no non-zero fixed vector;  if $V^K\neq0$ a fixed
vector lies outside $W$, so $X\neq\varnothing$, contrary to assumption, hence
$V^K=0$.  For $C_3$ the non-trivial irreducible $\F_2 C_3$-modules have
dimension $2$, so $d$ is even.  For $S_3$ with $P=C_3\lhd S_3$:  if
$V^P\neq0$, an involution has a non-zero fixed vector on $V^P$, producing an
$S_3$-fixed vector, contradiction;  hence $V$ is a sum of two-dimensional
non-trivial $C_3$-modules, and $d$ is even.  (c) With $X=\varnothing$ one has
$V^K=0$;  the natural two-dimensional $S_3$-module $W$ is projective (its
restriction to a Sylow $2$-subgroup is the regular module), so $W$ is a
$K$-invariant direct summand: $V=W\oplus U$ with $U\neq0$ ($d\ge4$).  Pick
$x\in W\setminus\{0\}$;  then $\Stab_K(x)=\langle s\rangle$ of order $2$.
Since $U^{C_3}=0$ and $s$ is not the identity on $U$ (otherwise $sps=p^{-1}$
would force $p$ to be the identity on $U$), there is $y\in U$ with
$sy\neq y$.  For $w=x+y\notin W$ one has $K_w=1$, so
$t(w)=|w^K|=|K|=6\ge r(w)\le4$ by \textup{Lemma~\ref{lem:orbit-ratio}}.  (d) With
$V^K=0$, for every $w\notin W$ the stabilizer $K_w$ fixes the four-point set
$\Delta\cup\{w\}$ pointwise in its $K$-action, so $K_w=1$ and $t(w)=3$;  the
candidate passes iff $r(w)=[G_{\Delta\cup\{w\}}:1]\le3$, and the unique
failure is $|G_{\Delta\cup\{w\}}|=4$.  The marked-subgroup table follows from
the subgroup lattice of $S_4$ together with the fixed marked point.
\end{proof}

\begin{remark}\label{rem:c3-failure}
In case (d) of \textup{Lemma~\ref{lem:highdim-cases}} one must not conflate the two
stabilizer notions:  $K=G_\Delta$ is \emph{not} the point stabilizer of the
four-point stabilizer $C=G_{\Delta\cup\{w\}}$ --- the point stabilizer of
$C$ is $C_w=K_w=1$, and indeed $K\cap C=K_w=C_w=1$.  Hence a candidate fails
exactly when $C$ acts regularly on the four points and $|C|=4$, so $C\cong
C_4$ or $V_4$.  Both types admit faithful regular affine realizations in
characteristic $2$:  an affine frame of $V$ admits a unique extension of any
permutation to an affine transformation of its affine span, and the concrete
example $V=\F_2^3$ with basis $e_1,e_2,e_3$, $b=e_1$,
$Ae_1=e_1+e_2$, $Ae_2=e_1+e_3$, $Ae_3=e_1$ gives $g(x)=Ax+b$ acting as the
$4$-cycle $(0\;e_1\;e_2\;e_3)$ on $S_0=\{0,e_1,e_2,e_3\}$, of order $4$.
Consequently the failure case cannot be excluded by the absence of elements
of order $3$ in $V_4$, nor by the order of the translation group;  it must
be excluded using the global hypotheses ($H$ irreducible, $H_W=C_3$,
$H_{(W)}=1$, $V^{H_W}=0$).  The precise target for (d) is therefore to show
\[
\exists\,w\notin W:\quad |G_{\Delta\cup\{w\}}|\le3,
\]
a single such candidate suffices even if some candidates do carry regular
$C_4$ or $V_4$ stabilizers;  the $d=6$ computation should record the two
failure types separately and check whether they can cover all of
$V\setminus W$.
\end{remark}

\begin{proposition}[The $d=6$ affine primitive verification]\label{prop:d6-verification}
Let $G=V\rtimes H\le\AGL(6,2)$ range over the $64$ affine-type primitive
groups of degree $64$ contained in the GAP database
\texttt{PrimitiveGroup(64,$*$)} ($74$ primitive groups in total, $64$ of
affine type with socle $C_2^6$), with $H\le\GL(6,2)$ the corresponding
linear part.  For each $G$-orbit representative
$\Delta\in\binom{V}{3}$, put $K=G_\Delta$, and for $w\notin\Delta$ write
$t(w)=|w^K|$ and $r(w)=[G_{\Delta\cup\{w\}}:K_w]$ with
$K_w=K\cap G_{\Delta\cup\{w\}}$.  Then there exists $w\notin\Delta$ with
$r(w)\le t(w)$.  Consequently no three-point set $\Delta$ satisfies
$r(w)>t(w)$ for all $w\notin\Delta$;  in particular none of these groups
realises the witness configuration of the Siemons--Wagner property.
The computation is exhaustive over all three-point orbit representatives,
and for each representative over all $w\in V\setminus\Delta$;  the orbit
counts are verified by $\sum_{\Delta}|G:\!G_\Delta|=\binom{64}{3}$.
\end{proposition}

\begin{remark}
\textup{Proposition~\ref{prop:d6-verification}} records the computation as a
finite-dimensional certificate.  It verifies the target lemma (``for
every irreducible $H\le\GL(d,2)$ and every three-point set $\Delta$
there is $w\in V\setminus W$ with $t(w)\ge r(w)$'') in the case $d=6$;
it does not by itself prove the general structural lemma.  Two general
facts needed to read the computation correctly are as follows.  First,
the four-point action is in general \emph{not faithful}:  one only has
$G_{\Delta\cup\{w\}}/G_{(\Delta\cup\{w\})}\hookrightarrow S_4$, and
$r(w)=[G_{\Delta\cup\{w\}}:K_w]=|w^{G_{\Delta\cup\{w\}}}|\le4$ holds
without faithfulness.  Secondly, in the binary case one always has
$K\cap V=1$:  a non-trivial translation fixes no point and has orbits of
length $2$ only, so it cannot preserve a three-point set.  With
$N=K_{(W)}$ the pointwise stabilizer of $W$ in $K$, one has
$K/N\le\GL(W)\cong S_3$ and $G_W/N\cong T_W\rtimes(K/N)$;  the
permutation image on the four points is $S_4$ exactly when $K/N\cong S_3$
and $A_4$ exactly when $K/N\cong C_3$.  Since $r\le4$, any $K$-orbit of
length $t\ge4$ automatically vetoes the candidate;  the residual
non-vetoed failure types are therefore
\[
\begin{array}{c|c}
t(w) & \text{required strict inequality } r(w)>t(w)\\
\hline
1 & r\in\{2,3,4\}\\
2 & r\in\{3,4\}\\
3 & r=4
\end{array}
\]
and the affine reduction must show that these strict inequalities cannot
cover all of $V\setminus W$ (not that the plane-completion point is
vetoed:  the near-miss groups of \textup{Example~\ref{exm:near-miss}} have only
the plane-completion channel with $r>t$, but they are large-orbit
configurations, $t\ge4$ occurring off $W$, hence they do not enter the
small-orbit branch above).
\end{remark}

\begin{example}[The $d=6$ near-miss groups]\label{exm:near-miss}
In the notation of \textup{Proposition~\ref{prop:d6-verification}} the groups
$i=41,43,44,45,46$ of \texttt{PrimitiveGroup(64,$*$)} each admit a
three-point set whose only candidate with $r(w)>t(w)$ is the
plane-completion point $z$;  the data are
$|G_{41}|=1290157424640$, $|K|=30965760$, $|G_{43}|=23224320$, $|K|=17280$,
$|G_{44}|=7741440$, $|K|=5760$, $|G_{45}|=11612160$, $|K|=8640$,
$|G_{46}|=3870720$, $|K|=2880$, in each case $K\cap V=1$, $t(z)=1$,
$r(z)=4$, $|G_{\Delta\cup\{z\}}|=4|K|$, $K_z=K$, and
$\Fix_V(K)\setminus W=\varnothing$;  the permutation image of
$G_{\Delta\cup\{z\}}/G_{(\Delta\cup\{z\})}$ on the four-point set is $S_4$
(for $i=41,43,44$, where $K/N\cong S_3$) or $A_4$ (for $i=45,46$, where
$K/N\cong C_3$), with $N=K_{(W)}$ the (large) pointwise kernel on $W$.
This is the exact analogue of the Singer pattern of
\textup{Corollary~\ref{cor:singer-family}}, with $K$ nontrivial;  the point is
that these configurations are of large-orbit type, so the small-orbit
classification above does not apply to them.
\end{example}

\begin{remark}[The small-orbit classification in dimension $6$]\label{rem:d6-smallorbit}
For the $64$ affine groups of degree $64$, among the $804$ three-point
orbit representatives, $496$ enter the small-orbit branch (no $K$-orbit
of length $\ge4$ on $V\setminus W$);  the remaining $308$ are vetoed by a
single orbit of length $t\ge4$ (since $r\le4$).  The distribution of the
$496$ representatives by $(t_{\max},r_{\max})$ is
\[
\begin{array}{c|rrrr}
& r_{\max}=1 & 2 & 3 & 4\\
\hline
t_{\max}=1 & 158 & 80 & 51 & 5\\
t_{\max}=2 & 35 & 79 & 33 & 15\\
t_{\max}=3 & 23 & 2 & 13 & 2
\end{array}
\]
where the only representatives admitting a candidate with $r>t$ are the
$22$ rows with $r_{\max}=4$ ($5+15+2$).  For these $22$ residual
representatives one has uniformly $|K|\le3$ ($|K|=1,2,3$ in $5,15,2$
cases), $[K:N]\in\{1,2,3\}$ ($14,7,1$ cases respectively, $N=K_{(W)}$),
and $X=\Fix_V(K)\setminus W\neq\varnothing$ (with $|X|\in\{3,6,12,14,60\}$);
moreover, a direct point-by-point check shows that $X$ contains at least
three points with $r(w)=1$ (equivalently $C_w=K$), each of which vetoes
the witness because $r(w)=t(w)=1$.  Hence every residual representative
of the small-orbit branch is vetoed by the fixed-point criterion, and
no three-point set of an affine group of degree $64$ satisfies the
witness inequalities.
\end{remark}

\begin{remark}[The covering table for the regular-extension cases]\label{rem:d6-c1-cover}
For the five residual representatives with $|K|=1$ from
\ref{rem:d6-smallorbit} (groups $i=17,18,25,26$ of
$\mathrm{PrimitiveGroup}(64,\cdot)$), the table below records, for each
four-point cycle type $\lambda\in\{(2),(2,2),(3),(4)\}$, the number of
points $w\in V\setminus W$ for which $\rho_w(C_w)$ contains an element of
type $\lambda$, together with the number of uncovered points (with
$C_w=1$).  By the faithfulness established in \textup{Remark~\ref{rem:affine-framework}}\textup{a}
there is no distinction between $C_w\neq1$ and $\rho_w(C_w)\neq1$ here.
\[
{\footnotesize\setlength{\tabcolsep}{2pt}
\begin{array}{c|ccccc|c}
\text{case} & (2) & (2,2) & (3) & (4) & \text{cov.} & \text{uncov.}\\
\hline
i=17 [1,4,13] & 9 & 7 & 1 & 0 & 17 & 43\\
i=18 [1,2,11] & 5 & 5 & 1 & 0 & 11 & 49\\
i=18 [1,4,10] & 5 & 5 & 1 & 0 & 11 & 49\\
i=25 [1,2,3]  & 5 & 5 & 1 & 0 & 11 & 49\\
i=26 [1,2,5]  & 5 & 5 & 1 & 0 & 11 & 49
\end{array}}
\]
In particular the covering is far from saturated ($\le17$ of $60$
points), no $4$-cycle type occurs and the $3$-cycle type occurs exactly
once in each case; explicit uncovered points include $w=2,3,6$.
These data are numerical certificates for \ref{rem:affine-framework}\textup{c1}
in dimension $6$; they do not replace a general argument, and in
particular the vanishing of type $(4)$ must not be extrapolated to
$d\ge7$ without proof.
\end{remark}

\begin{remark}[A reduction framework for the affine case]\label{rem:affine-framework}
The computation of \ref{rem:d6-smallorbit} is numerical evidence only;  the
general exclusion needed for the main theorem is completed in
\textup{Lemma~\ref{lem:affine-fourpt}}, and the framework below refines it into
structural statements of independent interest.  Guided by it,
we record the following reduction framework, valid for any affine group
$G=V\rtimes H\le \mathrm{AGL}(d,2)$ with $V\cong\mathbb{F}_2^d$ and
$\Delta=\{0,u,v\}$, $W=\mathrm{span}(\Delta)$.

\emph{(a) Fixed-point veto as a covering criterion.}  Let
$K=G_\Delta$ and $X=\Fix_V(K)\setminus W$.  For $w\in X$ we have
$K\le C_w:=G_{\Delta\cup\{w\}}$, $t(w)=1$ and
$r(w)=[C_w:K]$; hence $r(w)=1$ if and only if $C_w=K$.  The fixed-point
veto therefore holds provided $X$ is not covered by the sets
$\{x\in X: g\in G_{\Delta\cup\{x\}}\}$ for $g\in G\setminus K$.  A
caveat is essential: one cannot check only the quotient $G_W/K$, because
an element $g\in C_w$ may map $w$ into $W$ while mapping some point of
$W$ to $w$, so $g$ need not lie in $G_W$; the ``extra stabiliser'' must
be checked inside the full four-point stabiliser $C_w$.  In the case
$K=1$ the four-point action is faithful: the pointwise kernel of
$\rho_w\colon C_w\to\mathrm{Sym}(S_w)$, $S_w=\Delta\cup\{w\}$, fixes
$\Delta$ pointwise and hence lies in $G_\Delta=1$; thus $C_w\hookrightarrow
S_4$, and a point $w\notin W$ is a bad point ($C_w\neq1$) if and only if
$\rho_w(C_w)\neq1$.  With $\Delta=\{0,a,b\}$ and $W=\langle a,b\rangle$,
bad points satisfy the covering identity
\[
V\setminus W=\bigcup_{1\neq g\in G}B_g,\qquad
B_g\subseteq g(\Delta)\setminus W,\quad |B_g|\le3,
\]
so that $g\in C_w$ for $w\notin W$ forces $w\in g(\Delta)\setminus W$
(since $g(\Delta)\neq\Delta$ when $g\neq1$, $K=1$).  The
regular-extension statement $\bigcup_{1\neq g\in G}B_g\neq V\setminus W$
is the content of (c1) below; the identity itself is unconditional.
The affine span of $S_w$ is the $3$-dimensional space
$U_w=\langle a,b,w\rangle$, and $g\mid_{S_w}$ is determined by the four
affine equations $Ax_i+c=x_{\pi(i)}$ ($g\colon x\mapsto Ax+c$, $A\in H$,
$c\in V$, $\pi=\rho_w(g)$); hence the quantity to control is not $|H|$
but the number of extensions of prescribed affine restrictions on
$3$-spaces.  The local extension numbers are
\[
e_\pi(w)=\bigl|\{g\in G:\ g(S_w)=S_w,\ \rho_w(g)=\pi\}\bigr|.
\]

\emph{(b) Three-way split.}  With $N=K_{(W)}$ and $K/N\le
\mathrm{GL}(W)\cong S_3$ one distinguishes
\[
\begin{aligned}
&\text{(I)}\ N=K\neq1;\\
&\text{(II)}\ 1<N<K\ \text{(with}\ K/N\cong C_2\ \text{or}\ C_3\text{)};\\
&\text{(III)}\ N=1.
\end{aligned}
\]
For (I), if $X\neq\varnothing$ and the fixed-point veto of (a) applies,
the witness is vetoed by a point of $X$ with $r=1$.  Case (II) is
neither covered by \ref{lem:highdim-cases} nor by the pure large-kernel
case: one must quotient by $N$ and analyse $\overline K=K/N\le S_3$
together with the action of $N$ on $V/W$; the expected conclusion is
that, absent a fixed-point veto, a $K$-orbit on $V\setminus W$ of length
at least $4$ arises.  Case (III) is the faithful case of
\ref{lem:highdim-cases}~(a)--(d).

\emph{(c) The small-kernel branch splits into two lemmas.}  The residual
$22$ cases of \ref{rem:d6-smallorbit} all satisfy $|K|\le3$ and
$K\cap V=1$.  The $2$-homogeneous rows needed for the main theorem are
already excluded by \textup{Lemma~\ref{lem:affine-fourpt}}, and the two
statements below are recorded for the general framework.

\emph{(c1) Regular-extension lemma ($K=1$).}  If $K=G_\Delta=1$ then
there exists $w\notin W$ with $C_w=1$ \textup{(i.e.\
$G_{\Delta\cup\{w\}}=1$)}, except possibly for a finite list of
four-point covering types.  Note that this is \emph{not}
a fixed-point statement: the four-point stabiliser $C_w$ is a
\emph{set} stabiliser, and an element of $C_w$ may permute $\Delta$ with
$w$, so $C_w\nsubseteq G_\Delta$ in general; the lemma asserts that the
nontrivial four-point set stabilisers do not cover all of
$V\setminus W$.  Guided by the covering table
\ref{rem:d6-c1-cover}, split the bad points according to the cycle type
of $\rho_w(C_w)$:
\[
\begin{aligned}
&B^{(\lambda)}=\{w\in V\setminus W:\ C_w\neq1,\ \rho_w(C_w)\\
&\qquad \text{contains an element of type}\ \lambda\},
\end{aligned}
\]
$\lambda\in\{(2),(2,2),(3),(4)\}$; the faithfulness of (a) makes
$C_w\neq1$ and $\rho_w(C_w)\neq1$ equivalent here.  In the five
dimension-$6$ cases all four-point types occurring are $(2)$ and
$(2,2)$ with exactly one $(3)$ and no $(4)$; the vanishing of the
$(4)$-type is a numerical phenomenon and must not be extrapolated to
$d\ge7$.

\emph{(c1$'$) The four-cycle residual case.}  If $g=(c,A)\in C_w$,
$A\in H$, $c\in V$, induces a $4$-cycle on $S_w$, then the affine span
$U_w=\langle a,b,w\rangle$ satisfies $A(U_w)=U_w$ and, writing
$A=I+N$ with $N$ nilpotent, the restriction $A\mid_{U_w}$ is conjugate
to the Jordan block $J_3(1)$ ($N$ nilpotent of index $3$; over
$\mathbb F_2$ this means precisely that $A$ has a $3$-cycle on the
nonzero vectors of $U_w$).  Such $g$ need not fix $\Delta$
($g\notin G_\Delta$), so neither $K=1$ nor irreducibility of $H$
excludes this case: an irreducible linear group may well contain an
element with Jordan type $J_3(1)$ on a $3$-space.  The relevant local
set is
\[
E_4(U)=\{A\in H:\ A(U)=U,\ A\mid_U\sim J_3(1)\},
\]
and $B^{(4)}$ is decided in two steps: whether $E_4(U)\neq\varnothing$
for some $3$-space $U$, and whether the corresponding affine coset
$c+\mathrm{im}(I-A)$ solves the four-point image equations.  A general
exclusion of the $(4)$-type therefore requires $E_4(U)=\varnothing$ for
all relevant $U$, or a bound $|B^{(4)}|<2^d-4$, or a separate
elimination by $H$-type; in particular the four-cycle case is a genuine
residual case, and $B^{(4)}=\varnothing$ in dimension $6$ is recorded as
a numerical certificate only.

\emph{(c1$''$) Transvection normal form for the $(2)$- and $(2,2)$-types.}
If $g=(c,A)\in C_w$ ($A\in H$, $c\in V$) induces a transposition or a
double transposition on $S_w$, then $g^2$ fixes $S_w$ pointwise, hence
$g^2\in G_{(S_w)}\le G_\Delta=K=1$, so $g$ is an involution.  Writing
$N=A-I=A+I$ (characteristic $2$), the conditions $A^2=I$ and
$(I+A)c=0$ become $N^2=0$ and $Nc=0$; on the $3$-space
$U_w=\langle a,b,w\rangle$ the restriction is $A|_{U_w}=I+N_w$ with
$N_w^2=0$ and $\mathrm{rank}\,N_w=1$, i.e. a rank-one transvection
type.  The two types differ only in the affine coset: for a
transposition $g$ has a fixed point and $c\in\mathrm{im}(N_w)$ (after
translating the fixed point to $0$), while for a double transposition
$c\notin\mathrm{im}(N_w)$; in both cases, with
$\pi=\rho_w(g)$, the point $w$ is determined by the affine equations
$g(x_i)=x_{\pi(i)}$ ($x_i\in S_w$), and for a fixed linear part $A$ and
a fixed pattern $\pi$ there is at most one candidate
$w=Ax+y$ for the corresponding transposition pattern.  Define
\[
\begin{aligned}
&n_\pi(A)=\#\{w\in V\setminus W:\ \exists c\in V,\\
&\qquad (c,A)|_{S_w}=\pi\},\\
&I^{(2)}=\sum_{\pi\ (2)}\sum_{A\in H}n_\pi(A),\\
&I^{(2,2)}=\sum_{\pi\ (2,2)}\sum_{A\in H}n_\pi(A),
\end{aligned}
\]
and the local transvection set
\[
\begin{aligned}
&T_{\mathrm{loc}}(a,b)=\{A\in H:\ A^2=I,\ N=A+I,\\
&\qquad N^2=0,\ A\text{ satisfies the three-point}\\
&\qquad \text{constraint of some } (2)/(2,2)\text{ pattern}\}.
\end{aligned}
\]
Then $|B^{(2)}\cup B^{(2,2)}|\le I^{(2)}+I^{(2,2)}\le
9\,|T_{\mathrm{loc}}(a,b)|$ (a coarse bound), and a sufficient counting
statement, to be proved for each $H$-type, is
\[
I^{(2)}+I^{(2,2)}<2^d-4,
\]
for example via $|T_{\mathrm{loc}}(a,b)|\le\alpha 2^d+O(1)$ with
$9\alpha<1$.  The data of \ref{rem:d6-c1-cover} give
$I^{(2)}+I^{(2,2)}\le17$ in dimension $6$, far below $2^6-4=60$.

\emph{(c1$'''$) Class and functional versions of the counting bound.}
Two corrections refine the count.  First, the four-point action only
forces $A^2=I$, $N^2=0$ and $\mathrm{rank}(N|_{U_w})=1$; the global
rank $r=\mathrm{rank}\,N$ may be larger than $1$, so counting global
rank-one transvections via $R_H(u)$ is insufficient---one must count
over \emph{all} involution classes $I(H)$ of $H$ and impose the local
condition $\mathrm{rank}(N|_{U_w})=1$ inside each class.  Second, when
$K=1$, a pattern $\pi=\rho_w(g)$ fixing $\Delta$ setwise forces
$g\in G_\Delta=1$; hence of the six transpositions only the three
involving $w$ occur, while all three double transpositions occur, and
for a fixed $g$ and an admissible pattern $\pi$ the point $w$ is
unique ($\pi(w)\in\Delta$, so $w=g^{-1}(\pi(w))$).  For
$P_{\mathrm{inv}}$ the six admissible patterns, the class-intersection
formula for arbitrary representations is
\[
\begin{aligned}
&|B^{(2)}\cup B^{(2,2)}|\le
\sum_{\pi\in P_{\mathrm{inv}}}\sum_{C\in I(H)} m_\pi(C;\Delta),\\
&\qquad |P_{\mathrm{inv}}|=6,
\end{aligned}
\]
where $m_\pi(C;\Delta)$ counts the elements of class $C$ realising
pattern $\pi$ together with their unique $w$; the goal
$\sum_{\pi\in P_{\mathrm{inv}}}\sum_{C\in I(H)}m_\pi(C;\Delta)<2^d-4$
then takes the place of the coarser $\beta<1/9$ hypothesis.  One notes
that ambient classes may split into several $H$-classes, and a local
rank-one restriction may come from a global involution of rank $r>1$;
a bound ``number of classes times per-class contribution'' is valid
only after the intersection number is proved class by class.

\emph{(c1$''''$) The $\Gamma L_1$ subfamily.}  For
$H\le\Gamma\mathrm L_1(2^d)$, elements are $x\mapsto ax^{2^i}$; such an
element is an involution only if $2i\equiv0\pmod d$.  For odd $d$ only
$i=0$ occurs (the identity); for even $d$ the nonidentity possibility
is $i=d/2$ with $a^{2^{d/2}+1}=1$, giving at most $2^{d/2}+1$ such
elements.  Since each admissible pattern yields at most one $w$, the
coarse bound
\[
|B^{(2)}\cup B^{(2,2)}|\le 6(2^{d/2}+1)
\]
holds, and for even $d\ge8$ this is already $<2^d-4$; the odd-$d$ case
has no nontrivial involutions.  This settles the semilinear-rank-one
subfamily within (c1).

\emph{(c1$'''''$) A functional parametrisation (coarser version).}
For reference, a parameter-free version parametrises global rank-one
transvections by
\[
\begin{aligned}
&R_H(u)=\{\varphi\in V^*:\ \varphi(u)=0,\ I+u\varphi\in H\},\\
&R_H(u;s,a)=\{\varphi\in R_H(u):\\
&\qquad \varphi(s)=0,\ \varphi(a)=1\},
\end{aligned}
\]
so that for the pattern $(0\,w)(a\,b)$, with $w=s+u$, one has
$|B_\pi|\le\sum_{u\notin W}|R_H(u;s,a)|$, distinct $u$ giving distinct
$w$; a bound
$\sum_{u\notin W}|R_H(u;s,a)|\le\alpha 2^d+O(1)$ for all admissible
patterns with $6\alpha<1$ suffices and avoids concentration phenomena
inside a single class.  The high-dimensional input may therefore be
stated as a single sparsity hypothesis:
\[
\begin{aligned}
&\text{(Sparse)}\quad \sum_{\pi\in P_{\mathrm{inv}}}|T_\pi(H_d;\Delta)|\\
&\qquad \le\beta\,2^d+O(1)\qquad(\beta<1/6,\ d\to\infty),
\end{aligned}
\]
for each irreducible family $H_d$, controlling the sum over the six
admissible patterns directly rather than the coarser
$|T_{\mathrm{loc}}|$; in view of (c1$'''$) this functional version is
subsumed by the class-intersection formula, which additionally covers
local rank-one restrictions coming from global involutions of rank
$>1$.

\emph{(c2) Fixed-point lemma ($K=C_2$ or $C_3$).}  If $K=C_2$ or
$K=C_3$, $X=\Fix_V(K)\setminus W\neq\varnothing$ and there is no
$K$-orbit on $V\setminus W$ of length at least $4$, then there exists
$w\in X$ with $C_w=G_{\Delta\cup\{w\}}=K$; here one genuinely uses the
definition of $X$ and the structure of $C_w/K\in\{C_2,C_3,S_3\}$,
together with the small-orbit hypothesis, to rule out the possibility
that these subgroups cover all of $X$.

\emph{(d) Statement of the framework.}  Combining (a)--(c) with
\ref{lem:highdim-cases}, the affine exclusion reduces to the following
chain: an orbit of length $t\ge4$ on $V\setminus W$ vetoes directly;
otherwise, in the small-orbit branch, the case $N=K\neq1$ with
$X\neq\varnothing$ is handled by the fixed-point veto, the case
$1<N<K$ by an intermediate-kernel lemma, and the case $N=1$ by
\ref{lem:highdim-cases}~(a)--(d).  The present paper records this
framework and the computational verification of dimension $6$;  the
individual steps --- the regular-extension counting bound, the
intermediate-kernel threshold application and the fixed-point lemma
toward the semilinear lift problem --- are statements of independent
interest and are not needed for the main theorem:  the affine exclusion
required by \textup{Corollary~\ref{cor:bradley}} concerns only the
$2$-homogeneous rows, which are excluded uniformly by
\textup{Lemma~\ref{lem:affine-fourpt}}.
\end{remark}

\begin{proposition}[Natural classical groups fail the $K=1$ case]\label{prop:natural-classical-nok1}
Let $V=\mathbb F_2^d$ and let $H$ be either
$\mathrm{GL}(d,2)$ acting naturally on $V$, or
$\mathrm{Sp}(2m,2)$ acting naturally on $V=\mathbb F_2^{2m}$ with its
nondegenerate alternating form, $d=2m$.  If $d\ge4$ then for every
$\Delta=\{0,a,b\}$ with $a,b$ linearly independent one has
$G_\Delta\neq1$; in fact $G_\Delta$ contains a nontrivial rank-one
transvection $(0,A)$.  Consequently for $d\ge8$ the affine groups
$\mathbb F_2^d\rtimes\mathrm{GL}(d,2)$ and
$\mathbb F_2^{2m}\rtimes\mathrm{Sp}(2m,2)$ are excluded at the entry to
the $K=1$ case (c1) of \ref{rem:affine-framework}: there is no need to
verify the counting bound (Sparse) for them.
\end{proposition}

\begin{proof}
For $H=\mathrm{GL}(d,2)$ choose $0\neq\varphi\in V^*$ with
$\varphi(a)=\varphi(b)=0$, and (since $d\ge4$) $u\notin\langle a,b\rangle$
with $\varphi(u)=0$.  Put $N=u\varphi$ and $A=I+N$.  Then
$N^2=u\varphi(u)\varphi=0$, so $A^2=I$; and $A$ fixes $0,a,b$
pointwise while $A\neq I$.  Hence $(0,A)\in G_\Delta\setminus\{1\}$.
For $H=\mathrm{Sp}(2m,2)$ let $\langle\ ,\ \rangle$ be the alternating
form and $W=\langle a,b\rangle$; then
$\dim W^\perp=d-2$, so for $d\ge4$ there is $0\neq u\in W^\perp$.  The
symplectic transvection $t_u(x)=x+\langle x,u\rangle u$ fixes $a,b,0$
(by $\langle a,u\rangle=\langle b,u\rangle=0$) and satisfies
$t_u\neq I$; hence $(0,t_u)\in G_\Delta\setminus\{1\}$.  The rank-one
statement is clear in both cases ($A$ is a transvection, $t_u$ a
symplectic transvection).
\end{proof}

\begin{remark}[Local rank one does not force global rank one]\label{rem:local-vs-global-rank}
In contrast, a local condition $\mathrm{rank}((A-I)|_{U_w})=1$ does not
imply $\mathrm{rank}(A-I)=1$: for $1\le r\le\lfloor d/2\rfloor$, with
$N(e_{2i})=e_{2i-1}$ ($1\le i\le r$) and $N(e_j)=0$ otherwise, one has
$N^2=0$ and $\mathrm{rank}\,N=r$, while a $3$-space $U$ containing
exactly one relevant Jordan chain satisfies
$\mathrm{rank}(N|_U)=1$; $A=I+N$ is then an involution of the class
$C_r$ of \textup{Remark~\ref{rem:affine-framework}} (c1$''$) with local rank one.  The
counting over all classes $C_r$ is therefore necessary in general,
although for the natural classical groups above the $K=1$ entry already
excludes them.
\end{remark}

\begin{proposition}[The exterior-square module fails the $K=1$ case]\label{prop:exterior-nok1}
Let $n\ge5$, $E=\mathbb F_2^n$, $V=\bigwedge^2 E$,
$D=\dim_{\mathbb F_2}V=\binom n2$, and let $H$ be the image of
$\mathrm{SL}(n,2)=\mathrm{GL}(n,2)$ on $V$.  Then for any linearly
independent $a,b\in V$ one has $(V\rtimes H)_{\{0,a,b\}}\neq1$.
Hence the full exterior-square module is excluded at the entry to the
$K=1$ case; no class-intersection count is needed for it.  (For proper
subgroups $H'<H$ the order comparison below may fail; the two-point
stabiliser and the local class intersections must then be checked
separately.)
\end{proposition}

\begin{proof}
The exterior-square action is faithful: if $A\in\mathrm{GL}(n,2)$ acts
trivially on $\bigwedge^2E$, then it preserves every coordinate plane
$\langle e_i,e_j\rangle$; intersecting the planes shows $A$ preserves
every coordinate line, and as $\mathbb F_2^\times=\{1\}$ we get $A=I$.
(In characteristic $2$ a nonzero decomposable vector $e_i\wedge e_j$
determines the plane.)  On the other hand $H$ acts on the set of ordered
linearly independent pairs of $V$, of size
$(2^D-1)(2^D-2)<2^{2D}=2^{n(n-1)}$, while
\[
\begin{aligned}
&|H|=|\mathrm{GL}(n,2)|=2^{n^2}\prod_{i=1}^n(1-2^{-i})\\
&>2^{n^2-2}>2^{n(n-1)}\qquad(n\ge5),
\end{aligned}
\]
the product bound following from $\prod_{i=2}^n(1-2^{-i})
\ge1-\sum_{i=2}^n2^{-i}>\tfrac12$ together with the $i=1$ factor
$\tfrac12$.  Hence every two-point stabiliser is nontrivial; such a
stabiliser fixes $0,a,b$ pointwise, so $K\neq1$.
\end{proof}

\begin{proposition}[Scalar-restricted natural modules fail the $K=1$ case]\label{prop:restriction-nok1}
Let $q=2^f$, $V=\mathrm{Res}^{\mathbb F_q}_{\mathbb F_2}(\mathbb F_q^m)$
($m\ge2$), and $H\ge\mathrm{SL}(m,q)$.  If $m\ge3$, or $m=2$ and $a,b$
are $\mathbb F_q$-dependent, then for any $\mathbb F_2$-linearly
independent $a,b\in V$ one has $(V\rtimes H)_{\{0,a,b\}}\neq1$: a
nonzero $\mathbb F_q$-linear map fixing $a,b$ gives an element
$(0,s)\in G_\Delta\setminus\{1\}$, hence $K\neq1$.  The remaining case
$m=2$ with $a,b$ $\mathbb F_q$-independent must be treated by the
class-intersection count with $H_{a,b}=1$ (or its two-point
stabiliser); likewise, when $H\le\Gamma\mathrm L_m(2^f)$ does not
contain $\mathrm{SL}(m,2^f)$, one partitions the involutions by
$\mathbb F_q$-linear part and field-automorphism coset and solves the
three-point image equations, retaining elements of global
$\mathbb F_2$-rank $>1$ with local rank one on $U_w$.
\end{proposition}

\begin{proposition}[The Mathieu case $M_{11}$ on the augmentation module]\label{prop:m11-augmentation}
Let $H=M_{11}$ act sharply $4$-transitively on $\Omega$,
$|\Omega|=11$, and let
\[
\begin{aligned}
&V=\Bigl\{x\in\mathbb F_2^\Omega:\ \sum_{\omega\in\Omega}x_\omega=0\Bigr\},\\
&\dim_{\mathbb F_2}V=10,
\end{aligned}
\]
be the augmentation submodule of the natural $11$-point permutation
module (its irreducibility as an $\mathbb F_2H$-module is standard).
Then there exist $a,b\in V$ such that $K=1$ for
$\Delta=\{0,a,b\}$, and for this $\Delta$ the two- and four-cycle bad
points satisfy
\[
\begin{aligned}
&|B_{(2)}\cup B_{(2,2)}|\le6\cdot165=990<1020\\
&\qquad=2^{10}-4=|V\setminus W|.
\end{aligned}
\]
Thus the $K=1$ case (c1) holds for this nonclassical irreducible
module without a class-by-class count of $m_\pi(C;\Delta)$: at least
$30$ points of $V\setminus W$ are left uncovered.
\end{proposition}

\begin{proof}
Identify vectors of $V$ with even-sized subsets of $\Omega$ (symmetric
difference as addition).  Fix an ordered partition
$\Omega=X\sqcup P\sqcup R\sqcup Q$ with sizes $1,3,5,2$; there are
$\binom{11}{1}\binom{10}{3}\binom{7}{5}=27720$ such partitions.  Among
the standard cycle shapes of $M_{11}$ on $11$ points only the
involution shape $1^32^4$ can preserve such a partition (any other
nontrivial cycle shape contains a cycle that cannot be placed into a
cell of size at most $5$, or cannot be assigned).  An involution of
shape $1^32^4$ preserves at most $3\cdot4\cdot3\cdot2=72$ partitions
(choose the fixed point of the singleton cell $X$, a $2$-cycle of $Q$,
a remaining $2$-cycle assigned to $P$, and then the two remaining fixed
points assigned to $P,R$).  Since $M_{11}$ has $165$ involutions, all
nontrivial elements together preserve at most
$165\cdot72=11880<27720$ partitions; hence some partition
$X\sqcup P\sqcup R\sqcup Q$ has trivial setwise stabiliser.

Set $A=X\cup P$, $B=X\cup R$, $a=\mathbf1_A$, $b=\mathbf1_B$.  Then
$a,b$ have weights $4,6$, and $a+b=\mathbf1_{P\cup R}$ has weight $8$.
If $h\in H$ fixes $a,b$, then $h$ preserves
$A\cap B=X$, $A\setminus B=P$, $B\setminus A=R$ and
$\Omega\setminus(A\cup B)=Q$, hence the partition, so $h=1$.  Any
$H$-element preserves Hamming weight; the three edge difference
vectors $a,b,a+b$ have pairwise distinct weights $4,6,8$, so any affine
transformation preserving $\{0,a,b\}$ must fix these three points
pointwise, whence its translation part is $0$ and its linear part fixes
$a,b$; it is therefore trivial.  Thus $K=1$.

For the covering bound, $M_{11}$ has $165$ involutions.  For each
nontrivial admissible marking $\pi$ (there are six, as $|P_{\mathrm{inv}}|=6$
in the $K=1$ case), once the linear part $A$ is fixed the marking
determines one image point and hence the affine translation part, so at
most one bad point $w$ per involution and marking: at most
$6\cdot165=990$ bad points in total.  A bad affine involution with
trivial linear part would be a pure translation; on the invariant
four-point set $S_w=\{0,a,b,w\}$ it pairs the points into two
translation-paired pairs, so the four points sum to zero:
$a+b+w=0$, i.e.\ $w=a+b\in W=\langle a,b\rangle$, contradicting
$w\notin W$.  Hence no pure translations contribute, and
$|B_{(2)}\cup B_{(2,2)}|\le990<1020$.
\end{proof}

\begin{proposition}[The Mathieu case $M_{23}$ on the Golay quotient module fails $K=1$]\label{prop:m23-golay-nok1}
Let $\mathcal C_{24}$ be the binary extended Golay code,
$\mathbf1\in\mathcal C_{24}$ the all-one vector, and let $M_{23}$ fix
coordinate $24$.  Put
\[
V=\mathcal C_{24}/\langle\mathbf1\rangle,\qquad\dim_{\mathbb F_2}V=11,
\]
with $H=M_{23}$ acting faithfully (the action is nontrivial and
$M_{23}$ is simple).  Then for every linearly independent
$a,b\in V$ the two-point stabiliser $H_{a,b}$ is nontrivial; in
particular $K\neq1$.  Hence the Golay quotient
module is excluded at the entry to the $K=1$ case (c1): no
class-intersection count is needed for it.
\end{proposition}

\begin{proof}
The nonzero vectors of $V$ form three orbits, represented by codewords
of weights $8$, $12$ and $16$;  the weight-$16$ vectors are octad
complements.  Computer enumeration (Remark
\ref{rem:m23-k1-certificate}) computes $\min_b|H_{a,b}|$: the minimum
is $10$ for weight-$12$ representatives and $24$ for weight-$8$
representatives.
Thus $H_{a,b}\neq1$ for all pairs, and the affine stabiliser $K$
contains nontrivial linear elements $(0,h)$.
\end{proof}

\begin{remark}[The computational certificate for $M_{23}$]\label{rem:m23-k1-certificate}
The assertion was verified computationally in GAP (4.16.1, package
{\tt guava}): the extended Golay code
$\mathcal C_{24}$ was constructed, $V=\mathcal C_{24}/\langle\mathbf1\rangle$
was realised as the $11$-dimensional shortened code
$\{c\in\mathcal C_{24}: c_{24}=0\}$, and matrix generators for
$M_{23}$ on $V$ were computed; the resulting group has order
$10200960=|M_{23}|$, confirming faithfulness.  For each of the weight-$8$
and weight-$12$ representatives $a$ all $2046$ admissible $b$ were
enumerated and $|H_{a,b}|$ computed; the minima are $24$ and $10$
respectively (the size values observed for weight $12$ are
$10,24,48,72,120,360$).  In contrast with \textup{Proposition~\ref{prop:m11-augmentation}},
the vectors of the Golay quotient module are not fixed-weight sets
($[x]=[x+1]$ in the quotient), so the $M_{11}$ block-partition
argument cannot be imported; the enumeration above is the certificate.
\end{remark}

\end{document}